\documentclass[11pt,a4paper,reqno]{amsart}%

\usepackage{amsfonts}
\usepackage{amsmath}
\usepackage{amssymb}
\usepackage{graphicx}%

\usepackage{tikz}

\usepackage[margin=20mm]{geometry}

\usepackage{multirow}
\usepackage{array}

\usepackage[colorlinks, linkcolor=blue,  citecolor=blue, urlcolor=blue, bookmarks=false]{hyperref}%
\hypersetup{pdfstartview=FitH}
\usepackage{url}
\usepackage{xspace}

\usepackage{dsfont}

\usepackage{verbatim}
\usepackage{alltt}
\usepackage{graphicx}
\graphicspath{{../figures/}{../figuresOld/}}

\newtheorem{theorem}{Theorem}

\newtheorem{assumption}{Assumption}

\newtheorem{corollary}{Corollary}

\newtheorem{definition}{Definition}

\newtheorem{lemma}{Lemma}

\newtheorem{proposition}{Proposition}
\newtheorem{remark}{Remark}

\numberwithin{equation}{section}

\let\originalleft\left
\let\originalright\right
\renewcommand{\left}{\mathopen{}\mathclose\bgroup\originalleft}
\renewcommand{\right}{\aftergroup\egroup\originalright}

\renewcommand{\ref}[1]{\autoref{#1}}

\usepackage{enumitem}
\setlist[description]{align=right, labelwidth=2cm, leftmargin=!, font=\normalfont\bfseries}

\begin{document}
\title[Sharp conditioning of optimal transport]{Sharp conditioning and quantitative stability of optimal transport}
\author{Yuanlong Ruan}
\address{Beihang University, Beijing, China}


\date{\today}

\begin{abstract}
When a feasible plan is obtained whose quadratic transport cost is known to be close to the optimal cost, we try to determine how close the feasible plan is to the true optimal map under mild density and moment conditions. The source and target may be unbounded or have non-compact supports. We show that for a source with density in $L^p$ and an $n$-th moment, let
$s=1-1/p$ and $\eta=sn/[n+(d-1)s]$.
Under the target constraint
$\int |y|^m[\log(e+|y|)]^\beta\,d\nu\leqslant1$,
the worst-case class-uniform conditioning has a sharp modulus comparable to
\[
 t^{\eta(m-2)/[m(1+\eta)-\eta]}
 [\log(e/t)]^{-\beta(2+\eta)/[m(1+\eta)-\eta]},
\]
whenever $t>0$ is small. This holds for $m\geqslant2$ and $\beta\geqslant0$. When $m=2$, every $\beta>0$ gives a sharp logarithmic modulus, whereas the $m=2,\,\beta=0$ extreme has no
vanishing modulus. The matching lower bound is verified for both the squared map error and barycentric projection error. The sharp map error bounds allow us to directly derive quantitative stabilities of common source Brenier maps without passing through potential estimates, thereby avoiding loss of information.  These stabilities strictly improve the corresponding results of Delalande-Merigot \cite{delalande2023quantitative} and Letrouit-Mérigot \cite{letrouit2026gluing} under identical or weaker settings, particularly no convexity of the source support or bounds of the source density are assumed.
\end{abstract}

\maketitle
\setcounter{tocdepth}{1}
\tableofcontents

\section{Introduction and main results}

Quantitative stability of optimal transport has attracted a substantial effort
in the literature with seminal works
\cite{berman2021convergence,merigot2020quantitative,delalande2023quantitative,carlier2025quantitative,merigot2026sharp}
and in Wasserstein barycenter stability
\cite{carlier2024quantitative,carlier2026quantitative}. The settings are
mostly for compactly supported source, convex source support or a combination
of the two. Unbounded source support is considered recently, e.g.,
\cite{letrouit2026unstable}, which also extends the strong convexity theory of
\cite{delalande2023quantitative} to more general domains. Sources are treated
through variance inequalities and gluing. Extension to general cost
\cite{mischler2026quantitative} is also studied. If more regularities are
assumed, stronger stabilities are anticipated. For example
\cite{gallouet2026strong} assumes strong $c$-concavity with general costs,
\cite{balakrishnan2026stability} treats smooth map stability and statistical
consequences, \cite{caja2026stability} proves Lipschitz stability for
nondegenerate H\"{o}lder densities on smooth uniformly convex domains, to name
a few.

Our focus in this article is to provide a matched conditioning law under
simultaneous density and moment constraints and its power-logarithm case,
including rough and singlar targets. It is designed to answer the question: if
an approximate feasible plan is obtained whose transport cost is known to be
close to the ground true within a given tolerance and an approximate optimal
map is constructed, then what we can say about how close the constructed map
is to the true optimal map? The answer to this question will have its own
numerical benefits, but it will also allow us to treat the quantitative
stability of optimal transport maps without touching the potentials. As
implied by the recent work \cite{merigot2026sharp}, we believe going from
potential to map could cause loss of information if the goal is stability of
Brenier map itself.

\subsection{Problem setting}

Let $\rho,$ $\nu\in\mathcal{P}_{2}\left(  \mathbb{R}^{d}\right)  $. When
$\rho$ has a density, write $T_{\nu}$ for the unqiue Brenier map for the
quadratic optimal transport. For a feasible plan $\sigma\in\Pi\left(  \rho
,\nu\right)  $, let%
\begin{equation}
\mathcal{E}_{\nu}\left(  \sigma\right)  =\int\left\vert y-T_{\nu}\left(
x\right)  \right\vert ^{2}d\sigma\left(  x,y\right)  , \label{eq:Eintro}%
\end{equation}%
\begin{equation}
\mathcal{D}_{\nu}\left(  \sigma\right)  =\frac{1}{2}\int\left\vert
x-y\right\vert ^{2}d\sigma\left(  x,y\right)  -\frac{1}{2}W_{2}^{2}\left(
\rho,\nu\right)  . \label{eq:Dintro}%
\end{equation}
The first quantity $\left(  \ref{eq:Eintro}\right)  $, the \textit{squared map
error}, measures the error of the plan $\sigma$ relative to the map $T_{\nu}$
at the same source point. The second quantity $\mathcal{D}_{\nu}\left(
\sigma\right)  $ is the \textit{objective/cost gap}, the difference between
the cost of an approximate plan $\sigma$ and the optimal cost $\frac{1}%
{2}W_{2}^{2}\left(  \rho,\nu\right)  $.

If $\sigma=\left(  \operatorname{Id},S\right)  _{\#}\rho$, not necessarily in
$\Pi\left(  \rho,\nu\right)  ,$ then
\begin{equation}
\mathcal{E}_{\nu}\left(  \sigma\right)  =\int\left\vert S\left(  x\right)
-T_{\nu}\left(  x\right)  \right\vert ^{2}d\rho\left(  x\right)  =\left\Vert
S-T_{\nu}\right\Vert _{L^{2}\left(  \rho\right)  }^{2}. \label{eq:Eintro-map}%
\end{equation}
A general plan $\sigma$ may split mass. Barycentric projection
\cite[Definition 5.4.2]{ambrosio2005gradient}\cite{oberman2015efficient} is
the usual way to reconstruct an approximate map. Write $\sigma\left(
dx,dy\right)  =\rho\left(  dx\right)  \sigma_{x}\left(  dy\right)  $, where
$\sigma_{x}$ is the conditional kernel. The barycentric projection is defined
for $\rho$-a.e. $x,$
\[
\bar{T}_{\sigma}\left(  x\right)  =\int yd\sigma_{x}\left(  y\right)  .
\]
If we write
\[
V_{\sigma}=\int\left\vert y-\bar{T}_{\sigma}\left(  x\right)  \right\vert
^{2}d\sigma\left(  x,y\right)  ,
\]
then
\begin{equation}
\mathcal{E}_{\nu}\left(  \sigma\right)  =\left\Vert \bar{T}_{\sigma}-T_{\nu
}\right\Vert _{L^{2}\left(  \rho\right)  }^{2}+V_{\sigma}.
\label{eq:variance-decomposition}%
\end{equation}
The cross term vanishes because $\int\left(  y-\bar{T}_{\sigma}\left(
x\right)  \right)  d\sigma_{x}\left(  y\right)  =0$. The barycentric map
$\bar{T}_{\sigma}$ needs not push $\rho$ to $\nu$.

We ask how accurately a small objective gap determines the optimal map,
uniformly over a specified class of marginals. This is a quantitative
stability question in practice: does a nearly optimal feasible plan
approximate the optimal map, and at what rate?

There are two aspects of this question. The first is \textit{conditioning}
which views the question as the conversion of objective gap into squared map
error, i.e. a bound of the form%
\[
\mathcal{D}_{\nu}\left(  \sigma\right)  \leqslant t\implies\mathcal{E}_{\nu
}\left(  \sigma\right)  \leqslant\omega\left(  t\right)  ,
\]
where $\omega$ is a conditioning modulus. Its size measures how much map error
a given objective tolerance permits under required assumptions. It provides a
sufficient upper error modulus. For a fixed class $\mathcal{C}$ of probability
marginals, define
\begin{equation}
\Omega_{\mathcal{C}}\left(  t\right)  =\left\{  \mathcal{E}_{\nu}\left(
\sigma\right)  :\left(  \rho,\nu\right)  \in\mathcal{C},\text{ }\sigma\in
\Pi\left(  \rho,\nu\right)  ,\text{ }\mathcal{D}_{\nu}\left(  \sigma\right)
\leqslant t\right\}  . \label{eq:class-modulusIntro}%
\end{equation}
This is the smallest class-uniform conditioning modulus. An upper bound
$\Omega_{\mathcal{C}}\left(  t\right)  \leqslant C\varphi\left(  t\right)  $
answers how small a cost tolerance is sufficient for desired solution
accuracy. A matching lower bound $\Omega_{\mathcal{C}}\left(  t\right)
\geqslant c\varphi\left(  t\right)  $ answers whether a uniformly better
conversion can follow from the same requirement. Together they describe the
intrinsic worst-case information provided by the cost gap. This is a
\textit{class-uniform stability,} $\Omega_{\mathcal{C}}\left(  t\right)
\rightarrow0$ as $t\rightarrow0.$ By $\left(  \ref{eq:variance-decomposition}%
\right)  $, it also justifies the barycentric reconstruction.

The other aspect is data perturbation: how does the optimal map change when a
marginal changes? The two aspects are closedly related but not the same.
Uniform stability implies indivisual stability, but the converse does not
hold. Failure of uniform stability does not rule out stability for a fixed
pair of marginals.

\subsection{Main results and organization}

\emph{The structural distribution-adaptive estimate}\textit{.} Under the
Brenier setting of Assumption \ref{ass:brenier}, Section \ref{sec:error-bound}
proves the structural estimate for an arbitrary plan $\sigma\in\Pi\left(
\rho,\nu\right)  ,$%
\begin{equation}
\mathcal{E}_{\nu}(\sigma)\leqslant\inf_{b\in\mathbb{R}^{d},\,R,h>0}\left\{
4R\frac{\mathcal{D}_{\nu}(\sigma)}{h}+18R^{2}\mathfrak{A}_{\rho}%
(h)+8\tau_{2,\nu,b}(R)\right\}  . \label{eq:error-intro}%
\end{equation}
Here $\mathfrak{A}_{\rho}$ measures squared subgradient oscillation of convex
$1$-Lipschitz functions and $\tau_{2,\nu,b}$ is the ordinary squared target
tail. The information of the source measure enters only $\mathfrak{A}_{\rho}$
and is processed through the\ elegent geometric estimates from
\cite{carlier2023fenchel,carlier2024quantitative}. The estimate is also
inspired by the recent work \cite{merigot2026sharp}.\ One advantage of the
structural estimate is it allows us to treat source and target separately,
giving distribution-adaptive error bound. Moreover, since we directly addresss
the stability of the maps, more information is retained. Definitions and
concrete controls appear in Sections~\ref{sec:profiles} -
\ref{sec:source-profile}.

\emph{Source estimates and power laws}. Under appropriate source density
conditions, $\rho=fdx,$ $f\in L^{p}\left(  \mathbb{R}^{d}\right)  ,$
$1<p\leqslant\infty$ and $s=1-1/p,$ Theorem~\ref{thm:power} and
Theorem~\ref{thm:unbounded-power} combine subgradient oscillation estimates
with an $m$-th target moment, $m>2$ to build power laws for $\mathcal{E}_{\nu
}(\sigma).$ For bounded densities and compact source and target,
Corollary~\ref{cor:compact-target} gives
\[
\mathcal{E}_{\nu}(\sigma)\leqslant\operatorname*{Const}\cdot\mathcal{D}_{\nu
}(\sigma)^{s/\left(  1+s\right)  },\text{ }0<s\leqslant1.
\]
When $p=\infty$, $s=1,$ this is the squared map error estimate of M\'{e}rigot
\cite[Theorem 3.4 (3.15)]{merigot2026sharp}, for an arbitrary feasible plan
with the same marginals as the reference optimal map. The correlation gap
there equals $\left(  \ref{eq:Dintro}\right)  $. The constructions in
Section~\ref{sec:sharpness} confirms that a matching class-uniform modulus is
not available in the extreme cases, including the unbounded density and $m=2$
cases. For a prescribed $L^{1}$ density the source profile still vanishes.

\emph{The power-logarithm conditioning law}\textit{. }The pure power law is
established in Theorem \ref{thm:unbounded-power} with a \textit{matching lower
bound} verified in Theorem \ref{thm:escaping-slabs}, and consolidated in
Theorem \ref{thm:class-law} when $m\geqslant2,$ $\beta=0.$ With%
\[
\eta=\frac{sn}{n+(d-1)s},\text{ }H=m(1+\eta)-\eta,
\]
the pure power law is, for small $t>0,$%
\[
\Omega_{\mathcal{C}_{m,0}}\left(  t\right)  \simeq t^{\eta\left(  m-1\right)
/H},
\]
for the fixed class $\mathcal{C}_{m,0}$ defined in Section \ref{sec:class-law}%
. Theorem \ref{thm:class-law} in Section~\ref{sec:class-law} extends it to the
logarithmically weighted target moment class $\mathcal{C}_{m,\beta},$ giving
the \textit{power-logarithm law}%
\[
\Omega_{\mathcal{C}_{m,\beta}}\left(  t\right)  \simeq t^{\eta\left(
m-1\right)  /H}\left[  \log\left(  e/t\right)  \right]  ^{-\beta\left(
2+\eta\right)  /H}.
\]

\emph{Brenier map stabilities with common source}. Section~\ref{sec:two-map}
introduces a repairing process that constructs an optimal plan through a
coupling of two targets. The objective gap of the repaired plan is then
controlled by centered source moments or integrated source tails.

For compact source and target, Theorem~\ref{thm:compact-maps} gives
\[
\Vert T_{\mu}-T_{\nu}\Vert_{L^{2}(\rho)}\leqslant\operatorname*{Const}\cdot
W_{1}(\mu,\nu)^{s/[2(1+s)]}.
\]
At $p=\infty$, $s=1,$ this recovers M\'{e}rigot's $1/4$ power \cite[Theorem
3.1]{merigot2026sharp}. It is also the common source specialization of
\cite[Theorem 4.6]{ford2026quantitative}. Here the error is an $L^{2}$ norm,
not its square, and the input is the distance between two targets, not a
feasible plan's cost gap. The passage between these problems requires the
repairing argument in the proof of Theorem~\ref{thm:compact-maps}, not just
taking a square root.

Theorem~\ref{thm:both-unbounded} gives the corresponding map bound for
unbounded non-compact sources and targets under its stated moment assumptions,%
\[
\Vert T_{\mu}-T_{\nu}\Vert_{L^{2}(\rho)}\leqslant\operatorname*{Const}\cdot
W_{1}(\mu,\nu)^{\alpha\Theta_{\eta,m}/2}.
\]
And in Section \ref{sec:sharp-1D}, we discuss the sharpness of the exponent.

The flexible structural estimate $\left(  \ref{eq:error-intro}\right)  $
permits the combinations of various source and target assumptions, each
combination may lead to a dedicated quantiative stability which exploits
specific properties of a given family of marginals thus gives stronger
stability than the general theorem would provide, and the sharp conditioning
can also give refined stability. But we have refrained from exhausting the
combinations and focus instead on combinations of moments and compactness of supports.

Section \ref{sec:DM-noncompact}, Section~\ref{sec:LM-comparison},
Section~\ref{sec:gap-map-error} are devoted to comparisons, particularly
showing dedicated quantiative stabilities for specific probability classes.
These sections may be read as applications of the main results. The main part
of the article does not depend on these materials.

\emph{Further comparisons with prior works and strict improvements}. For
targets in a fixed ball and $\delta=W_{1}(\mu,\nu)$, we prove in
Section~\ref{sec:source-envelopes}, for source with Gaussian upper envelope
(Theorem \ref{thm:logconcave}), such as the log-concave family, that%
\[
\left\Vert T_{\mu}-T_{\nu}\right\Vert _{L^{2}(\rho)}\leqslant
\operatorname*{Const}\cdot\delta^{1/4}\left[  \log\left(  c/\delta\right)
\right]  ^{1/8},\text{ \ if }f\left(  x\right)  \leqslant C_{0}e^{-\kappa
\left\vert x-a\right\vert ^{2}/2},
\]
and for source with polynomial upper envelope (Theorem \ref{thm:cauchy}), that%
\[
\left\Vert T_{\mu}-T_{\nu}\right\Vert _{L^{2}(\rho)}\leqslant
\operatorname*{Const}\cdot\delta^{\left(  \beta-d-1\right)  /\left(  4\left(
\beta-d\right)  \right)  },\text{ \ if }f\left(  x\right)  \leqslant
C_{0}\left(  1+\left\vert x-a\right\vert ^{2}\right)  ^{-\beta/2},\text{
}\beta>d+2.
\]
On common hypotheses these improve the map exponents in \cite[Theorems 1.4 and
1.12]{letrouit2026gluing}. The first also avoids an upper Hessian bound, and
the second avoids a density lower bound. Precise comparisons are given after
the corresponding proofs.

Section \ref{sec:DM-noncompact} gives detailed comparison with \cite[Section
4]{delalande2023quantitative} on Brenier map stability between compact source
and noncompact target. Section \ref{sec:LM-comparison} is devoted to
comparison with the corresponding results of \cite{letrouit2026gluing} on
compact source and target under various source density assumptions. Our
Brenier map stability results give strict improvements in all these
comparisons under identical or weaker assumptions.

\emph{From objective gap to reconstructed maps}.
Section~\ref{sec:gap-map-error} is for applications. It treats the
verification of solution accuracy from primal-dual gap, and gives barycentric
map error bounds under explicit accuracy requirements. It is worth noting in
passing that the map error bound and the conditioning law could also be used
in regularized settings, such as verification of the accuracy of entropic map
\cite{pooladian2021entropic}, or barycentric map constructed from
divergence-regularized optimal transport plan \cite{eckstein2024convergence}.
The theory will tell us when to stop a solver to reach a desired solution accuracy.

\begin{remark}
Throughout the article, when compactly supported probability measures are
defined, we use closed balls such as $B\left(  a,R_{X}\right)  ,$ $B\left(
b,R_{Y}\right)  $ to indicate where their supports belong. But no convexity of
these underlying sets are needed.
\end{remark}

\section{Preliminaries and assumptions}

\label{sec:prelim}

\subsection{Notations}

We use $\left\vert \cdot\right\vert $ to denote the Euclidean norm, and
$B(a,R)$ is the \textbf{closed} ball with center $a$ and radius $R$.
$\omega_{d}$ is the volume of the unit ball in $\mathbb{R}^{d}$,
$d\geqslant1.$

$\mathcal{P}\left(  \mathbb{R}^{d}\right)  $ denotes probability measures on
$\mathbb{R}^{d}$. $\mathcal{P}_{2}\left(  \mathbb{R}^{d}\right)  $ denotes
probability measures with finite second moment. If a probability has a
density, then we write e.g. $\rho=fdx$ or $\rho\ll dx.$ For $\mu\in
\mathcal{P}\left(  \mathbb{R}^{d}\right)  ,$ $r>0,$ $b\in\mathbb{R}^{d}$,
define the $r$-th probability norm with center $b$ and its minimized form
\[
M_{\mu,r}\left(  b\right)  =\left(  \int_{\mathbb{R}^{d}}\left\vert
y-b\right\vert ^{r}d\mu\left(  y\right)  \right)  ^{1/r},\text{ }M_{\mu
,r}=\inf_{b\in\mathbb{R}^{d}}M_{\mu,r}\left(  b\right)  .
\]
For $R>0,$ $k\geqslant0,$ define the $k$-th order tail integral
\[
\tau_{k,\mu,b}\left(  R\right)  =\int_{\left\vert y-b\right\vert >R}\left\vert
y-b\right\vert ^{k}d\mu\left(  y\right)  .
\]
In particular, we write $\tau_{\mu,b}\left(  R\right)  =\tau_{0,\mu,b}\left(
R\right)  =\mu\left(  \left\vert y-b\right\vert >R\right)  $ for the $0$-th
tail integral.

For convex $u$, $\operatorname{dom}u=\left\{  x:u\left(  x\right)
<+\infty\right\}  ,$ $\operatorname{dom}\partial u=\left\{  x:\partial
u\left(  x\right)  \text{ is non-empty}\right\}  .$ For a set $C\subset
\mathbb{R}^{d}$, $I_{C}$ denotes its indicator, $I_{C}\left(  x\right)  =0$ on
$C,$ while $I_{C}\left(  x\right)  =+\infty$ otherwise.

We write $\log$ for the natural logarithm.

\subsection{Transport, potentials and integrability}

The following is the standing assumption that holds throughout the paper.%

\begin{assumption}[\textbf{Brenier setting}]%
\label{ass:brenier}The Borel probability measures $\rho,$ $\nu$ belong to
$\mathcal{P}_{2}\left(  \mathbb{R}^{d}\right)  $. There is a proper lower
semicontinuous convex function $u:\mathbb{R}^{d}\mapsto\left(  -\infty
,+\infty\right]  $ such that the unique optimal plan from $\rho$ to $\nu$ is
$\left(  \operatorname{Id},T\right)  _{\#}\rho$, where $T=\nabla u,$ $\rho
$-a.e. The potential is chosen so that%
\[
u\in L^{1}\left(  \rho\right)  ,\text{ }u^{\ast}\in L^{1}\left(  \nu\right)
.
\]
%

\end{assumption}%

If $\rho\ll dx$, this assumption follows from Brenier's theorem and quadratic
Kantorovich duality
\cite{brenier1991polar,villani2003topics,villani2009optimal}. Some singular
sources are also admissible if they do not give mass to small sets of
Hausdorff dimension at most $d-1$, but we retain Assumption~\ref{ass:brenier}
whenever singular sources are discussed \cite{mccann1995existence}.

In general for a proper convex $u,$ $\operatorname{dom}\partial u\subset
\operatorname{dom}u$ \cite[Theorems 23.4]{rockafellar1997convex}$.$ Assumption
\ref{ass:brenier} implies that%
\begin{equation}
\rho\left(  x:\partial u\left(  x\right)  =\left\{  \nabla u\right\}  \right)
=\rho\left(  \operatorname{dom}\partial u\right)  =\rho\left(
\operatorname{dom}u\right)  =1. \label{eq:Rho-1}%
\end{equation}
For a lower semicontinuous convex $v:\mathbb{R}^{d}\mapsto\left(
-\infty,+\infty\right]  $, its conjugate and subdifferential are%
\[
v^{\ast}(p)=\sup_{x\in\mathbb{R}^{d}}\left\{  x\cdot p-v\left(  x\right)
\right\}  ,
\]%
\[
\partial v\left(  x\right)  =\left\{  p\in\mathbb{R}^{d}:v\left(  z\right)
\geqslant v\left(  x\right)  +p\cdot\left(  z-x\right)  ,\ \forall
z\in\mathbb{R}^{d}\right\}  .
\]
Following \cite{carlier2023fenchel}, for a proper lower semicontinuous convex
$v$, we define its Fenchel gap by%
\begin{equation}
g_{v}\left(  x,p\right)  \triangleq v\left(  x\right)  +v^{\ast}\left(
p\right)  -x\cdot p\geqslant0. \label{eq:fenchel-gap}%
\end{equation}
The equality holds if and only if $p\in\partial v(x)$.

\subsection{Lifting a feasible plan}

\begin{lemma}
[Self-coupling lift]\label{lem:lift}Assume Assumption~\ref{ass:brenier} and
let $\sigma\in\Pi(\rho,\nu)$. There is a self-coupling $q\in\Pi(\rho,\rho)$
such that
\begin{equation}
\mathcal{E}_{\nu}\left(  \sigma\right)  =\int\left\vert T\left(  x^{\prime
}\right)  -T\left(  x\right)  \right\vert ^{2}dq\left(  x,x^{\prime}\right)  ,
\label{eq:E-lift}%
\end{equation}%
\begin{equation}
\mathcal{D}_{\nu}\left(  \sigma\right)  =\int x\cdot\left(  T\left(  x\right)
-T\left(  x^{\prime}\right)  \right)  dq\left(  x,x^{\prime}\right)  ,
\label{eq:D-corr}%
\end{equation}%
\begin{equation}
\mathcal{D}_{\nu}\left(  \sigma\right)  =\int g_{v}\left(  x,T\left(
x^{\prime}\right)  \right)  dq\left(  x,x^{\prime}\right)  . \label{eq:D-gap}%
\end{equation}

\end{lemma}

\begin{proof}
By disintegration \cite[Theorem 5.3.1]{ambrosio2005gradient}, there is a
measurable probability kernel $y\mapsto\kappa_{y}$ which disintegrates $\rho$
w.r.t. the Brenier map $T$,
\[
\rho\left(  A\right)  =\int\kappa_{y}\left(  A\right)  d\nu\left(  y\right)
,
\]
and%
\begin{equation}
\kappa_{y}\left(  \left\{  x^{\prime}:T\left(  x^{\prime}\right)  =y\right\}
\right)  =1\text{, }\nu\text{-a.e. }y. \label{eq:Ky1}%
\end{equation}
Define%
\[
\Lambda\left(  dx,dy,dx^{\prime}\right)  =\sigma\left(  dx,dy\right)
\kappa_{y}\left(  dx^{\prime}\right)  .
\]
and let $q$ be the $\left(  x,x^{\prime}\right)  $-marginal of $\Lambda$. The
first marginal of $q$ is the first marginal of $\sigma$, namely $\rho$, and
its second marginal is%
\[
\int\kappa_{y}\left(  dx^{\prime}\right)  d\nu\left(  y\right)  =\rho\left(
dx^{\prime}\right)  .
\]
Thus $q\in\Pi\left(  \rho,\rho\right)  $. Since $T(x^{\prime})=y$ for
$\Lambda$-a.e. triple $\left(  x,y,x^{\prime}\right)  $ by $\left(
\ref{eq:Ky1}\right)  $, the map error can be rewritten as%
\[
\int\left\vert y-T\left(  x\right)  \right\vert ^{2}d\sigma\left(  x,y\right)
=\int\left\vert y-T\left(  x\right)  \right\vert ^{2}d\Lambda\left(
x,y,x^{\prime}\right)  =\int\left\vert T\left(  x^{\prime}\right)  -T\left(
x\right)  \right\vert ^{2}dq\left(  x,x^{\prime}\right)  .
\]
This proves $\left(  \ref{eq:E-lift}\right)  $.

Expanding the difference $\left(  \ref{eq:Dintro}\right)  $\ of costs gives%
\begin{align*}
\mathcal{D}_{\nu}\left(  \sigma\right)   &  =\frac{1}{2}\int\left\vert
x-y\right\vert ^{2}d\sigma\left(  x,y\right)  -\frac{1}{2}W_{2}^{2}\left(
\rho,\nu\right) \\
&  =\frac{1}{2}\int\left\vert x-y\right\vert ^{2}d\Lambda\left(
x,y,x^{\prime}\right)  -\frac{1}{2}\int\left\vert x-T\left(  x\right)
\right\vert ^{2}d\Lambda\left(  x,y,x^{\prime}\right) \\
&  =\frac{1}{2}\int\left\vert x-T(x^{\prime})\right\vert ^{2}d\Lambda\left(
x,y,x^{\prime}\right)  -\frac{1}{2}\int\left\vert x-T\left(  x\right)
\right\vert ^{2}d\Lambda\left(  x,y,x^{\prime}\right) \\
&  =\frac{1}{2}\int\left[  \left\vert x-T(x^{\prime})\right\vert
^{2}-\left\vert x-T\left(  x\right)  \right\vert ^{2}\right]  dq\left(
x,x^{\prime}\right) \\
&  =\int x\cdot\left(  T\left(  x\right)  -T\left(  x^{\prime}\right)
\right)  dq\left(  x,x^{\prime}\right)  +\int\left[  \left\vert T\left(
x^{\prime}\right)  \right\vert ^{2}-\left\vert T\left(  x\right)  \right\vert
^{2}\right]  dq\left(  x,x^{\prime}\right)  .
\end{align*}
The last integral is zero because both coordinates of $q$ have law $\rho$.
Hence $\left(  \ref{eq:D-corr}\right)  $ follows.

Since $T\left(  x^{\prime}\right)  \in\partial u\left(  x^{\prime}\right)  $,%
\[
u^{\ast}\left(  T\left(  x^{\prime}\right)  \right)  =x^{\prime}\cdot T\left(
x^{\prime}\right)  -u\left(  x^{\prime}\right)  .
\]
All terms are integrable under Assumption \ref{ass:brenier}. Hence%
\begin{align*}
\int g_{u}\left(  x,T\left(  x^{\prime}\right)  \right)  dq  &  =\int\left[
u\left(  x\right)  -u\left(  x^{\prime}\right)  -\left(  x-x^{\prime}\right)
\cdot T\left(  x^{\prime}\right)  \right]  dq\left(  x,x^{\prime}\right) \\
&  =\int x^{\prime}\cdot T\left(  x^{\prime}\right)  dq\left(  x,x^{\prime
}\right)  -\int x\cdot T\left(  x^{\prime}\right)  dq\left(  x,x^{\prime
}\right) \\
&  =\int x\cdot T\left(  x\right)  d\rho\left(  x\right)  -\int x\cdot
T\left(  x^{\prime}\right)  dq\left(  x,x^{\prime}\right) \\
&  =\int x\cdot T\left(  x\right)  dq\left(  x,x^{\prime}\right)  -\int x\cdot
T\left(  x^{\prime}\right)  dq\left(  x,x^{\prime}\right)
\end{align*}
where the second line uses equality of the two marginals of $q$ and the third
renames the integration variable. The result is $\left(  \ref{eq:D-corr}%
\right)  $. Therefore $\left(  \ref{eq:D-gap}\right)  $ is proved.
\end{proof}

\section{Intrinsic source and target profiles}

\label{sec:profiles}

Let $v:\mathbb{R}^{d}\mapsto\mathbb{R}$ be convex. For $h>0$, define its
subgradient oscillation around a point $x$ by
\begin{equation}
\omega_{v,h}\left(  x\right)  =\operatorname{diam}\left(  \bigcup_{z\in
B\left(  x,h\right)  }\partial v\left(  z\right)  \right)  . \label{eq:omega}%
\end{equation}
If $v$ is $L$-Lipschitz, all subgradient norms are bounded by $L,$ so%
\begin{equation}
0\leqslant\omega_{v,h}\left(  x\right)  \leqslant2L\text{ for all }x.
\label{eq:omega-2L}%
\end{equation}
The union in $\left(  \ref{eq:omega}\right)  $ is non-empty and compact.
Indeed, the subdifferential has closed graph and is locally bounded
\cite[Theorems 24.4 and 24.7]{rockafellar1997convex}. The correspondence
$x\twoheadrightarrow\bigcup_{z\in B\left(  x,h\right)  }\partial v\left(
z\right)  $ is upper hemicontinuous. Thus by Berge's maximum theorem
\cite[Lemma 17.30]{aliprantis2006infinite}, the maximum of $\left\vert
p-p^{\prime}\right\vert $ over the compact set of the pairs $\left(
z,p\right)  ,$ $\left(  z^{\prime},p^{\prime}\right)  $ with $z,$ $z^{\prime
}\in B\left(  x,h\right)  $ is upper semicontinuous in $x$. In particular it
is Borel measurable.

\begin{definition}
[Source oscillation profile]\label{def:A}Let $\rho\in\mathcal{P}_{2}\left(
\mathbb{R}^{d}\right)  $. Define
\begin{equation}
\mathfrak{A}_{\rho}\left(  h\right)  =\sup_{v}\int\omega_{v,h}^{2}\left(
x\right)  d\rho\left(  x\right)  , \label{eq:Adef}%
\end{equation}
where the supremum is over all convex $1$-Lipschitz functions $v:\mathbb{R}%
^{d}\mapsto\mathbb{R}$. An immediate bound from its definition is%
\begin{equation}
0\leqslant\mathfrak{A}_{\rho}\left(  h\right)  \leqslant4. \label{eq:A-univ}%
\end{equation}

\end{definition}

If $v$ is convex and $L$-Lipschitz, then
\begin{equation}
\int\omega_{v,h}^{2}\left(  x\right)  d\rho\left(  x\right)  \leqslant
L^{2}\mathfrak{A}_{\rho}\left(  h\right)  \leqslant4L^{2}. \label{eq:A-scale}%
\end{equation}

\begin{definition}
[Target tail profiles]\label{def:T}For $\nu\in\mathcal{P}_{2}\left(
\mathbb{R}^{d}\right)  $, $b\in\mathbb{R}^{d}$, $R>0$ and $t>0$, define%
\begin{equation}
\tau_{\nu,b}\left(  R\right)  =\nu\left(  \left\vert y-b\right\vert >R\right)
, \label{eq:Ftail}%
\end{equation}%
\begin{equation}
\tau_{2,\nu,b}\left(  R\right)  =\int_{\left\vert y-b\right\vert >R}\left\vert
y-b\right\vert ^{2}\,d\nu\left(  y\right)  , \label{eq:tau}%
\end{equation}
and%
\begin{equation}
\mathfrak{T}_{\nu,b}\left(  R\right)  =4\left[  \tau_{2,\nu,b}\left(
R\right)  +U_{\nu,b}\left(  \tau_{\nu,b}\left(  R\right)  \right)  \right]  ,
\label{eq:Tdef}%
\end{equation}
where for $t\geqslant0,$%
\begin{equation}
U_{\nu,b}\left(  t\right)  =\sup\left\{  \int\left\vert y-b\right\vert
^{2}\vartheta\left(  y\right)  d\nu\left(  y\right)  :0\leqslant
\vartheta\leqslant1,\ \int\vartheta d\nu\leqslant t\right\}  . \label{eq:U}%
\end{equation}

\end{definition}

\begin{remark}
Throughout the article we introduce centers to both source or target measures.
These centers help optimize the coefficients appearing in various estimates in
a distribution-adaptive manner.
\end{remark}

The function $\vartheta$ in $\left(  \ref{eq:U}\right)  $ acts as a selector
and is required for general $t$ when $\nu$ has atoms. In $\left(
\ref{eq:Tdef}\right)  $ the mass allowance of $U_{\nu,b}$ corresponds to
$t=\tau_{\nu,b}\left(  R\right)  $. In this case the tail profile is
simplified,
\begin{equation}
U_{\nu,b}\left(  \tau_{\nu,b}\left(  R\right)  \right)  =\tau_{2,\nu
,b}(R),\text{ }\mathfrak{T}_{\nu,b}\left(  R\right)  =8\tau_{2,\nu,b}\left(
R\right)  . \label{eq:tail-exact}%
\end{equation}
Indeed, let $Z=\left\vert y-b\right\vert ^{2}$, $A=\left\{  \left\vert
y-b\right\vert >R\right\}  =\{Z>R^{2}\}$ and $t=\nu(A)$. For every $\vartheta$
with $0\leqslant\vartheta\leqslant1$,%
\begin{align*}
\int Z\left(  \vartheta-1_{A}\right)  d\nu &  =\int\left(  Z-R^{2}\right)
\left(  \vartheta-1_{A}\right)  d\nu+R^{2}\left(  \int\vartheta d\nu-t\right)
\\
&  =\int_{A}\left(  Z-R^{2}\right)  \left(  \vartheta-1_{A}\right)  d\nu
+\int_{A^{c}}\left(  Z-R^{2}\right)  \left(  \vartheta-1_{A}\right)
d\nu+R^{2}\left(  \int\vartheta d\nu-t\right)  .
\end{align*}
Observe that all terms are nonpositive, and equal to zero if $\vartheta=1_{A}%
$. Thus $U_{\nu,b}\left(  \tau_{\nu,b}\left(  R\right)  \right)  $ is attained
by $\vartheta=1_{A}$ and $\left(  \ref{eq:tail-exact}\right)  $ follows.

\begin{lemma}
\label{lem:tail-vanish}If $\nu\in\mathcal{P}_{2}\left(  \mathbb{R}^{d}\right)
$, then%
\[
\lim_{R\rightarrow\infty}\mathfrak{T}_{\nu,b}\left(  R\right)  =0.
\]

\end{lemma}

\begin{proof}
Since the function $y\mapsto\left\vert y-b\right\vert ^{2}$ is $\nu
$-integrable, the dominated convergence theorem gives $\tau_{2,\nu
,b}(R)\rightarrow0$. The limit follows from the identity $\left(
\ref{eq:tail-exact}\right)  $: $\mathfrak{T}_{\nu,b}\left(  R\right)
=8\tau_{2,\nu,b}\left(  R\right)  $.
\end{proof}

\section{The proximal Fenchel estimate}

\label{sec:prox}

In the proof of Theorem~\ref{thm:error-bound} we will need a local slope with
required closeness to a given subgradient. The proximal map produces such a
slope. For $\lambda>0$, define
\[
\operatorname{prox}_{\lambda u}\left(  \xi\right)  \triangleq\arg\min
_{z\in\mathbb{R}^{d}}\left\{  u\left(  z\right)  +\frac{\left\vert
z-\xi\right\vert ^{2}}{2\lambda}\right\}  .
\]
The proximal operator is well-defined for proper lower semicontinuous convex
$u$ with unique minimizer for each $\xi.$ The first order condition of the
minimization gives%
\[
0\in\partial\left[  u\left(  z\right)  +\frac{\left\vert z-\xi\right\vert
^{2}}{2\lambda}\right]  .
\]
The two terms have non-empty intersection of domain relative interiors, since
the quadratic term has a domain relative interior of $\mathbb{R}^{d}.$ So
\[
0\in\partial u\left(  z\right)  +\frac{z-\xi}{\lambda}\text{ or equivalently
}\frac{\xi-z}{\lambda}\in\partial u\left(  z\right)  .
\]

We recall a useful property of the proximal map $\operatorname{prox}_{\lambda
u}\left(  \cdot\right)  $ that it is nonexpansive. Indeed, if $z_{i}%
=\operatorname{prox}_{\lambda u}\left(  \xi_{i}\right)  $, $i=1,2,$ set%
\[
q_{i}=\frac{\xi_{i}-z_{i}}{\lambda}\in\partial u\left(  z_{i}\right)  .
\]
Then by the cyclic monotonicity of the subdifferential we have $\left(
q_{1}-q_{2}\right)  \cdot\left(  z_{1}-z_{2}\right)  \geqslant0$, and hence
\[
\left\vert z_{1}-z_{2}\right\vert ^{2}=\left(  \xi_{1}-\xi_{2}\right)
\cdot\left(  z_{1}-z_{2}\right)  -\lambda\left(  q_{1}-q_{2}\right)
\cdot\left(  z_{1}-z_{2}\right)  \leqslant\left\vert \xi_{1}-\xi
_{2}\right\vert \left\vert z_{1}-z_{2}\right\vert .
\]
If $\left\vert z_{1}-z_{2}\right\vert $ is not zero, we obtain%
\[
\left\vert z_{1}-z_{2}\right\vert \leqslant\left\vert \xi_{1}-\xi
_{2}\right\vert .
\]
When $\left\vert z_{1}-z_{2}\right\vert =0,$ the inequality is trivial.
Therefore nonexpansiveness follows.

The next construction gives the local slope that we need later. The estimate
$\left(  \ref{eq:prox-gap}\right)  $, which is a kind of coercivity of the
Fenchel gap $\left(  \ref{eq:prox-gap}\right)  $, is proved in \cite[Lemma
1.1]{carlier2023fenchel}.

\begin{lemma}
[Pointwise proximal estimate]\label{lem:prox}Let $u$ be proper lower
semicontinuous and convex. For fixed $x,$ $x^{\prime}\in\operatorname*{dom}%
\partial u,$ suppose that $p\in\partial u\left(  x\right)  $, $r\in\partial
u\left(  x^{\prime}\right)  $, and $\left\vert p\right\vert ,$ $\left\vert
r\right\vert \leqslant R$. For $\lambda>0$, define
\begin{equation}
z=\operatorname{prox}_{\lambda u}\left(  x+\lambda r\right)  ,\text{
}v=r+\frac{x-z}{\lambda}. \label{eq:prox-v}%
\end{equation}
Then%
\begin{equation}
v\in\partial u\left(  z\right)  , \label{eq:prox-subgrad}%
\end{equation}%
\begin{equation}
\left\vert z-x\right\vert \leqslant2R\lambda, \label{eq:prox-distance}%
\end{equation}%
\begin{equation}
\left\vert v\right\vert \leqslant3R, \label{eq:v-bound}%
\end{equation}%
\begin{equation}
g_{u}\left(  x,r\right)  \geqslant\lambda\left\vert v-r\right\vert ^{2}.
\label{eq:prox-gap}%
\end{equation}

\end{lemma}

\begin{proof}
The inclusion $\left(  \ref{eq:prox-subgrad}\right)  $ follows from the
first-order condition%
\[
0\in\partial u\left(  z\right)  +\frac{z-\left(  x+\lambda r\right)  }%
{\lambda}\text{ or }\frac{\left(  x+\lambda r\right)  -z}{\lambda}\in\partial
u\left(  z\right)  .
\]
Let $\xi_{1}=x+\lambda r$ and $\xi_{2}=x+\lambda p.$ The definition $\left(
\ref{eq:prox-v}\right)  $ gives%
\[
z=\operatorname{prox}_{\lambda u}\left(  x+\lambda r\right)
=\operatorname{prox}_{\lambda u}\left(  \xi_{1}\right)  .
\]
Observe that the condition $p\in\partial u\left(  x\right)  $ is equivalent to%
\[
x=\operatorname{prox}_{\lambda u}\left(  x+\lambda p\right)
=\operatorname{prox}_{\lambda u}\left(  \xi_{2}\right)  .
\]
Applying the nonexpansiveness of the proximal map, we get
\[
\left\vert z-x\right\vert \leqslant\lambda\left\vert r-p\right\vert
\leqslant2R\lambda.
\]
This proves $\left(  \ref{eq:prox-distance}\right)  $. So it follows that
\[
\left\vert v\right\vert \leqslant\left\vert r\right\vert +\frac{\left\vert
x-z\right\vert }{\lambda}\leqslant3R,
\]
which proves $\left(  \ref{eq:v-bound}\right)  $. Using the subgradient
inequality at $\left(  z,v\right)  $ and the definition of the conjugate we
have
\[
u\left(  x\right)  \geqslant u\left(  z\right)  +v\cdot\left(  x-z\right)
\text{ and }u^{\ast}\left(  r\right)  \geqslant r\cdot z-u\left(  z\right)  .
\]
Adding them up gives%
\[
u\left(  x\right)  +u^{\ast}\left(  r\right)  \geqslant u\left(  z\right)
+v\cdot\left(  x-z\right)  +r\cdot z-u\left(  z\right)  .
\]
Subtracting $x\cdot r$ from both sides and using that $x-z=\lambda(v-r)$ we
obtain%
\[
g_{u}\left(  x,r\right)  \geqslant\left(  v-r\right)  \cdot\left(  x-z\right)
=\lambda\left\vert v-r\right\vert ^{2}.
\]
Therefore the proof is complete.
\end{proof}

The next is a simple way to truncate long-distance transports which we call a
slope-clipping principle. For convex functions, it has appeared in
\cite{carlier2025quantitative,merigot2026sharp}. See also
\cite{fathi2010optimal} for general handling.

\begin{lemma}
[Slope-clipping principle]\label{lem:slope-trunc}Under the hypotheses of
Lemma~\ref{lem:prox}, define%
\[
u_{R}=\left[  u^{\ast}+I_{B\left(  0,3R\right)  }\right]  ^{\ast}.
\]
The closed ball $B\left(  0,3R\right)  $ contains a point where $u^{\ast}$ is
finite so that $u_{R}$ is convex and $3R$-Lipschitz. Moreover,%
\[
p\in\partial u_{R}\left(  x\right)  ,\text{ }v\in\partial u_{R}\left(
z\right)  .
\]
and
\[
u_{R}\left(  x\right)  =u\left(  x\right)  ,\text{ }u_{R}\left(  z\right)
=u\left(  z\right)  .
\]

\end{lemma}

\begin{proof}
\textbf{1}. According to the hypotheses of Lemma~\ref{lem:prox},
$x\in\operatorname*{dom}u,$ $p\in\partial u(x)$, $\left\vert p\right\vert
\leqslant R.$ Hence $u\left(  x\right)  $ is finite. Since the subgradient
condition gives $u^{\ast}\left(  p\right)  =x\cdot p-u\left(  x\right)  ,$ so%
\[
\left\vert u^{\ast}\left(  p\right)  \right\vert \leqslant\left\vert
x\right\vert \left\vert p\right\vert +\left\vert u\left(  x\right)
\right\vert <\infty.
\]
This says that $B\left(  0,3R\right)  $ contains point $p$ where $u^{\ast}$ is
finite. By the definition of conjugate
\[
u_{R}(\xi)=\sup_{s\in B(0,3R)}\left[  \xi\cdot s-u^{\ast}(s)\right]  .
\]
It is a supremum of affine functions with slopes of norm at most $3R$, and is
therefore convex and $3R$-Lipschitz wherever it it finite. The finiteness of
$u^{\ast}\left(  p\right)  \ $gives%
\[
u_{R}(\xi)\geqslant\xi\cdot p-u^{\ast}(p)>-\infty,\text{ }\forall\xi,
\]
showing that $u_{R}$ cannot take $-\infty$ anywhere. The finiteness of
$u\left(  x\right)  $ on the other hand gives a lower bound for $u^{\ast}$ on
the closed ball $B(0,3R),$
\[
u^{\ast}\left(  s\right)  \geqslant x\cdot s-u\left(  x\right)  \geqslant
3R\left\vert x\right\vert -u\left(  x\right)  ,\text{ }\forall s\in B(0,3R).
\]
Therefore $u_{R}(\xi)$ is bounded above pointwise. These combined tell us that
the truncated conjugate $u_{R}$ must be finite everywhere.

\textbf{2}. Fix $w\in\operatorname{dom}\partial u.$ If $s\in\partial u\left(
w\right)  $ and $\left\vert s\right\vert \leqslant3R$, the Fenchel equality
gives
\[
u_{R}\left(  w\right)  \geqslant w\cdot s-u^{\ast}\left(  s\right)  =u\left(
w\right)  .
\]
The inequality has used the definition of $u_{R}$ and the fact that $s\in
B(0,3R).$ The restricted supremum $u_{R}$ is bounded above by $u^{\ast\ast
}\left(  w\right)  $ which equals $u\left(  w\right)  $ because $u$ is lower
semicontinuous convex. Hence for such $w$ we have $u_{R}\left(  w\right)
=u\left(  w\right)  $. So for every $\xi$,
\[
u_{R}\left(  \xi\right)  \geqslant\xi\cdot s-u^{\ast}\left(  s\right)
=u_{R}\left(  w\right)  +s\cdot\left(  \xi-w\right)  ,
\]
indicating $s\in\partial u_{R}\left(  w\right)  $. With the assumptions on $p$
and $\left(  \ref{eq:v-bound}\right)  $, we apply the preceding to $\left(
w,s\right)  =\left(  x,p\right)  $ and $\left(  w,s\right)  =\left(
z,v\right)  $ to conclude that $p\in\partial u_{R}\left(  x\right)  ,$
$v\in\partial u_{R}\left(  z\right)  .$
\end{proof}

\section{The distribution-adaptive map error bound}

\label{sec:error-bound}

The next theorem is the structural result of the paper. It turns the Fenchel
gap into a map error bound by separating three contributions: the objective
gap, the source mass near convex singularities and target tail.

\begin{theorem}
[Map error bound]\label{thm:error-bound}Assume Assumption~\ref{ass:brenier}
and let $\sigma\in\Pi\left(  \rho,\nu\right)  $. Then
\begin{equation}
\mathcal{E}_{\nu}\left(  \sigma\right)  \leqslant\inf_{b\in\mathbb{R}%
^{d},R,h>0}\left\{  \frac{4R}{h}\mathcal{D}_{\nu}\left(  \sigma\right)
+18R^{2}\mathfrak{A}_{\rho}\left(  h\right)  +\mathfrak{T}_{\nu,b}\left(
R\right)  \right\}  . \label{eq:error-bound}%
\end{equation}

\end{theorem}

\begin{proof}
\textbf{1}. Let $q\in\Pi\left(  \rho,\rho\right)  $ be the self-coupling from
Lemma~\ref{lem:lift}. Fix $b\in\mathbb{R}^{d}$ and define
\[
\bar{T}=T-b,\text{ }\bar{u}=u-b\cdot x.
\]
Note that $\partial\bar{u}=\partial u-b$ and hence $\bar{T}\in\partial\bar
{u},$ $\rho$-a.e. Since for $x^{\prime}\in\operatorname{dom}\partial\bar{u},$
\[
\left[  \bar{u}\right]  ^{\ast}\left(  \bar{T}\left(  x^{\prime}\right)
\right)  =u^{\ast}\left(  \bar{T}\left(  x^{\prime}\right)  +b\right)
=u^{\ast}\left(  T\left(  x^{\prime}\right)  \right)  ,
\]
the definition gives%
\begin{align*}
g_{\bar{u}}\left(  x,\bar{T}\left(  x^{\prime}\right)  \right)   &  =\bar
{u}\left(  x\right)  +\left[  \bar{u}\right]  ^{\ast}\left(  \bar{T}\left(
x^{\prime}\right)  \right)  -x\cdot\bar{T}\left(  x^{\prime}\right) \\
&  =u\left(  x\right)  -b\cdot x+u^{\ast}\left(  T\left(  x^{\prime}\right)
\right)  -x\cdot\left[  T\left(  x\right)  -b\right]  .
\end{align*}
Hence%
\begin{equation}
g_{\bar{u}}\left(  x,\bar{T}\left(  x^{\prime}\right)  \right)  =g_{u}\left(
x,T\left(  x^{\prime}\right)  \right)  . \label{eq:same-gap}%
\end{equation}
Fix $R,$ $h>0$. Define
\[
G_{b,R}=\{\left(  x,x^{\prime}\right)  :\left\vert \bar{T}\left(  x\right)
\right\vert \leqslant R,\ \left\vert \bar{T}\left(  x^{\prime}\right)
\right\vert \leqslant R\}.
\]
Note that%
\[
G_{b,R}\subset\left(  \operatorname{dom}\partial\bar{u}\right)  ^{2}=\left(
\operatorname{dom}\partial u\right)  ^{2}\subset\left(  \operatorname{dom}%
u\right)  ^{2}.
\]

\textbf{2}. We claim that%
\begin{equation}
\int_{G_{b,R}}\left\vert \bar{T}\left(  x\right)  -\bar{T}\left(  x^{\prime
}\right)  \,\right\vert ^{2}dq\leqslant\frac{4R}{h}\mathcal{D}_{\nu}\left(
\sigma\right)  +18R^{2}\mathfrak{A}_{\rho}\left(  h\right)  .
\label{eq:good-bound}%
\end{equation}
Since the inequality already holds if $q\left(  G_{b,R}\right)  =0$, we assume
$q\left(  G_{b,R}\right)  >0$ in the following$.$ So $G_{b,R}$ is non-empty.
For $q$-a.e. $\left(  x,x^{\prime}\right)  \in G_{b,R}$, let%
\[
p=\bar{T}\left(  x\right)  ,\text{ }r=\bar{T}\left(  x^{\prime}\right)  ,
\]
and choose
\[
\lambda=\frac{h}{2R}.
\]
Applying Lemma~\ref{lem:prox} to $\bar{u}$ yields $z,$ $v$ such that
\[
z\in B\left(  x,h\right)  ,\text{ }v\in\partial\bar{u}\left(  z\right)
,\text{ }\left\vert v\right\vert \leqslant3R,\text{ }g_{\bar{u}}\left(
x,r\right)  \geqslant\lambda\left\vert v-r\right\vert ^{2}.
\]
Since $p\in B\left(  0,R\right)  $ provides one point where $\bar{u}^{\ast
}\left(  p\right)  =x\cdot p-u\left(  x\right)  $ is finite, the truncated
potential
\[
\bar{u}_{R}=\left[  \bar{u}^{\ast}+I_{B\left(  0,3R\right)  }\right]  ^{\ast}%
\]
is well-defined by Lemma~\ref{lem:slope-trunc} and
\[
p\in\partial\bar{u}_{R}\left(  x\right)  ,\text{ }v\in\partial\bar{u}%
_{R}\left(  z\right)  .
\]
Since $z\in B\left(  x,h\right)  $, the definition in $\left(  \ref{eq:omega}%
\right)  $ gives
\[
\left\vert v-p\right\vert \leqslant\omega_{\bar{u}_{R},h}\left(  x\right)  .
\]
Writing $r-p=\left(  r-v\right)  +\left(  v-p\right)  $, we obtain
\begin{equation}
\left\vert r-p\right\vert ^{2}\leqslant2\left\vert r-v\right\vert
^{2}+2\left\vert v-p\right\vert ^{2}\leqslant\frac{2}{\lambda}g_{\bar{u}%
}\left(  x,r\right)  +2\omega_{\bar{u}_{R},h}\left(  x\right)  ^{2}.
\label{eq:good-point}%
\end{equation}
The Fenchel gap is nonnegative, so $\left(  \ref{eq:D-gap}\right)  $ and
$\left(  \ref{eq:same-gap}\right)  $ imply
\[
\int_{G_{b,R}}g_{u}\left(  x,T\left(  x^{\prime}\right)  \right)
dq=\int_{G_{b,R}}g_{\bar{u}}\left(  x,\bar{T}\left(  x^{\prime}\right)
\right)  dq\leqslant\mathcal{D}_{\nu}\left(  \sigma\right)  .
\]
The first marginal of $q|_{G_{b,R}}$ is dominated by $\rho$. Since $\bar
{u}_{R}$ convex and $3R$-Lipschitz, we have by $\left(  \ref{eq:A-scale}%
\right)  $%
\[
\int_{G_{b,R}}\omega_{\bar{u}_{R},h}^{2}\left(  x\right)  dq\leqslant
9R^{2}\mathfrak{A}_{\rho}\left(  h\right)  .
\]
Thus $\left(  \ref{eq:good-bound}\right)  $ follows by integrating $\left(
\ref{eq:good-point}\right)  $ on $G_{b,R}$ and using $2/\lambda=4R/h$.

\textbf{3}. Now we return to the original transport map. Since translation
does not change differences of transports, we may replace $\bar{T}\left(
x\right)  -\bar{T}\left(  x^{\prime}\right)  $ with $T\left(  x\right)
-T\left(  x^{\prime}\right)  $ in $\left(  \ref{eq:good-bound}\right)  $ and
let $\pi=\left(  T,T\right)  _{\#}q\in\Pi\left(  \nu,\nu\right)  $, with
coordinates $y,$ $y^{\prime}$. Let $A=\left\{  \left\vert y-b\right\vert
>R\right\}  $. Then
\[
\left(  G_{b,R}\right)  ^{c}=\{y\in A\}\cup\{y^{\prime}\in A\}.
\]
Since for all $y,$ $y^{\prime},$%
\[
\left\vert y-y^{\prime}\right\vert ^{2}\leqslant2\left\vert y-b\right\vert
^{2}+2\left\vert y^{\prime}-b\right\vert ^{2},
\]
we have
\[
\int_{\left(  G_{b,R}\right)  ^{c}}\left\vert y-y^{\prime}\right\vert ^{2}%
d\pi\leqslant2\int\left\vert y-b\right\vert ^{2}\left(  1_{\left\{  y\in
A\right\}  }+1_{\left\{  y^{\prime}\in A\right\}  }\right)  d\pi
+2\int\left\vert y^{\prime}-b\right\vert ^{2}\left(  1_{\left\{  y\in
A\right\}  }+1_{\left\{  y^{\prime}\in A\right\}  }\right)  d\pi.
\]
Each diagonal term equals $\tau_{2,\nu,b}\left(  R\right)  $. For a cross
term, say the first, the $y$-marginal of the submeasure $1_{\left\{
y^{\prime}\in A\right\}  }\pi$ has the form $\vartheta\nu$, where
$0\leqslant\vartheta\leqslant1$ and $\int\vartheta d\nu=\nu\left(  A\right)
=\tau_{\nu,b}\left(  R\right)  $. Hence each cross term is at most $U_{\nu
,b}\left(  \tau_{\nu,b}\left(  R\right)  \right)  $. It follows that
\begin{equation}
\int_{\left(  G_{b,R}\right)  ^{c}}\left\vert T\left(  x\right)  -T\left(
x^{\prime}\right)  \,\right\vert ^{2}dq\leqslant\mathfrak{T}_{\nu,b}\left(
R\right)  . \label{eq:bad-bound}%
\end{equation}
Adding $\left(  \ref{eq:good-bound}\right)  $, $\left(  \ref{eq:bad-bound}%
\right)  $ and using $\left(  \ref{eq:E-lift}\right)  $, we obtain%
\[
\mathcal{E}_{\nu}\left(  \sigma\right)  \leqslant\frac{4R}{h}\mathcal{D}_{\nu
}\left(  \sigma\right)  +18R^{2}\mathfrak{A}_{\rho}\left(  h\right)
+\mathfrak{T}_{\nu,b}\left(  R\right)  .
\]
The proof is therefore completed by taking the infimum over $R,$ $h>0$ and
$b\in\mathbb{R}^{d}$.
\end{proof}

\begin{remark}
Note that the compactness of the source is not used in the proof. It enters
only in some estimates of $\mathfrak{A}_{\rho}$, thus allowing us to treat
source and target tails separately.
\end{remark}

\section{Estimates of the source profile}

\label{sec:source-profile}

\subsection{Geometric estimate and the tailored oscillation integral}

Let $v$ be $L$-Lipschitz convex, $a\in\mathbb{R}^{d}$, $h>0,$ and $\omega_{d}$
be the volume of the unit ball. \cite[Theorem 2.1]{carlier2025quantitative}
gives a covering of oscillation in $\left(  \ref{eq:omega}\right)  $ by balls
of radius $8h$. Hence using $\left(  S+4h\right)  ^{d-1}\leqslant
5^{d-1}S^{d-1}$ for $0<h\leqslant S$ we get
\begin{equation}
\left\vert \left\{  x\in B\left(  a,S\right)  :\omega_{v,h}\left(  x\right)
>\alpha\right\}  \right\vert \leqslant K_{d}^{\mathrm{vol}}\frac{LS^{d-1}%
h}{\alpha}\text{ for }0<h\leqslant S,\ \alpha>0, \label{eq:volume}%
\end{equation}
where the bound has been scaled by the volume $\omega_{d}\left(  8h\right)
^{d}$ and%
\begin{equation}
K_{d}^{\mathrm{vol}}=48d^{2}\omega_{d}8^{d}5^{d-1}. \label{eq:K-vol}%
\end{equation}

The following is a variant of \cite[Corollary 2.2]{carlier2025quantitative}
tailored to our required specifications.

\begin{lemma}
[Oscillation integral]For $r>1$, define
\begin{equation}
K_{d,r}^{\mathrm{osc}}=\max\left\{  \frac{r2^{r-1}}{r-1}K_{d}^{\mathrm{vol}%
},\,2^{r}\omega_{d}\right\}  . \label{eq:K-osc}%
\end{equation}
Then, for $h,$ $S>0$,
\begin{equation}
\int_{B\left(  a,S\right)  }\omega_{v,h}^{r}dx\leqslant K_{d,r}^{\mathrm{osc}%
}S^{d-1}L^{r}h. \label{eq:osc-int}%
\end{equation}
The constant in the bound is independent of $a.$
\end{lemma}

\begin{proof}
If $L=0$, the function is constant and its oscillation vanishes. Assume $L>0$.
If $0<h\leqslant S$, using $\left(  \ref{eq:volume}\right)  $ and that
$\omega_{v,h}\leqslant2L$ yield%
\begin{align*}
\int_{B\left(  a,S\right)  }\omega_{v,h}^{r}dx  &  =r\int_{0}^{2L}\alpha
^{r-1}\left\vert \left\{  x\in B\left(  a,S\right)  :\omega_{v,h}\left(
x\right)  >\alpha\right\}  \right\vert d\alpha\\
&  \leqslant rK_{d}^{\mathrm{vol}}LS^{d-1}h\int_{0}^{2L}\alpha^{r-2}%
\,d\alpha=\frac{r2^{r-1}}{r-1}K_{d}^{\mathrm{vol}}S^{d-1}L^{r}h.
\end{align*}
Since $r>1,$ the integral is finite. If $h>S$, we use $\left(
\ref{eq:A-scale}\right)  $ to obtain
\[
\int_{B\left(  a,S\right)  }\omega_{v,h}^{r}dx\leqslant\left(  2L\right)
^{r}\omega_{d}S^{d}\leqslant2^{r}\omega_{d}S^{d-1}L^{r}h.
\]
Combining the two cases we see that $\left(  \ref{eq:osc-int}\right)  $ holds
for all $h,$ $S>0$.
\end{proof}

\subsection{Source with compact support}

\begin{proposition}
\label{prop:compact-Lp}Let $\rho=fdx\in\mathcal{P}\left(  \mathbb{R}%
^{d}\right)  $ be supported in $B\left(  a,R_{X}\right)  $, $R_{X}>0$, $f\in
L^{p}\left(  \mathbb{R}^{d}\right)  $,$1<p\leqslant\infty$ and $s=1-1/p.$ For
every $h>0$,
\begin{equation}
\mathfrak{A}_{\rho}\left(  h\right)  \leqslant K_{d,\rho,p}R_{X}^{\left(
d-1\right)  s}h^{s}. \label{eq:A-compact-Lp}%
\end{equation}
where%
\begin{equation}
K_{d,p}^{L^{p}}=\left(  K_{d,2/s}^{\mathrm{osc}}\right)  ^{s}\text{ and
}K_{d,\rho,p}=K_{d,p}^{L^{p}}\left\Vert f\right\Vert _{L^{p}}.
\label{eq:K-d-rho-p}%
\end{equation}

\end{proposition}

\begin{proof}
Let $v$ be $1$-Lipschitz convex. For $1<p\leqslant\infty$, using H\"{o}lder's
inequality with exponents $p,$ $1/s$ and the oscillation estimate $\left(
\ref{eq:osc-int}\right)  $ with $L=1$, $r=2/s$, $S=R_{X}$ we obtain
\[
\int\omega_{v,h}^{2}fdx\leqslant\left\Vert f\right\Vert _{L^{p}}\left(
\int_{B\left(  a,R_{X}\right)  }\omega_{v,h}^{2/s}dx\right)  ^{s}%
\leqslant\left(  K_{d,2/s}^{\mathrm{osc}}\right)  ^{s}\left\Vert f\right\Vert
_{L^{p}}R_{X}^{\left(  d-1\right)  s}h^{s}.
\]
The proof is completed by taking the supremum in $v$.
\end{proof}

The following small-set control weakens the assumption of Proposition
\ref{prop:compact-Lp}. It will compensate the need to deal with critical
$L^{p}$-norms without sending limits$.$

\begin{proposition}
[Small-set control]\label{prop:small-set}Let $\rho=fdx\in\mathcal{P}\left(
\mathbb{R}^{d}\right)  $ be supported in $B(a,R_{X})$, $R_{X}>0$. Suppose that
there is $C_{V}>0$ such that for each Borel $E\subset B(a,R_{X})$,
\begin{equation}
\rho\left(  E\right)  \leqslant C_{V}\left\vert E\right\vert ^{s},\text{
}0<s\leqslant1, \label{eq:small-set-density}%
\end{equation}
where $\left\vert E\right\vert $ is Lebesgue measure. Then for $h>0,$
\begin{equation}
\mathfrak{A}_{\rho}(h)\leqslant B_{s}R_{X}^{\left(  d-1\right)  s}h^{s},
\label{eq:small-set-A}%
\end{equation}
where%
\begin{equation}
B_{s}=\max\left\{  \frac{2^{3-s}}{2-s}C_{V}(K_{d}^{\mathrm{vol}})^{s}%
,\,4R_{X}^{-ds}\right\}  . \label{eq:small-set-Bs}%
\end{equation}

\end{proposition}

\begin{proof}
First assume $0<h\leqslant R_{X}$. Take a $1$-Lipschitz convex $v$ and $t>0$.
From the oscillation estimate $\left(  \ref{eq:volume}\right)  $ and~the
small-set condition $\left(  \ref{eq:small-set-density}\right)  ,$ we have
\[
\rho\left(  \omega_{v,h}>t\right)  \leqslant C_{V}(K_{d}^{\mathrm{vol}}%
R_{X}^{d-1})^{s}h^{s}t^{-s}.
\]
It follows that%
\[
\int\omega_{v,h}^{2}d\rho=2\int_{0}^{2}t\rho\left(  \omega_{v,h}>t\right)
dt\leqslant2C_{V}(K_{d}^{\mathrm{vol}}R_{X}^{d-1})^{s}h^{s}\int_{0}^{2}%
t^{1-s}dt=\frac{2^{3-s}}{2-s}C_{V}(K_{d}^{\mathrm{vol}}R_{X}^{d-1})^{s}h^{s}.
\]
The integral is finite for all $0<s\leqslant1$. Since this holds for all
$1$-Lipschitz convex $v,$ by taking the supremum in $v$ we get $\left(
\ref{eq:small-set-A}\right)  $ for the range $0<h\leqslant R_{X}$. For
$h>R_{X}$, we get by the universal bound $\left(  \ref{eq:A-univ}\right)  $
that $\mathfrak{A}_{\rho}\left(  h\right)  \leqslant4\leqslant4R_{X}^{-s}%
h^{s}.$ Combining the estimates yields $\left(  \ref{eq:small-set-A}\right)
.$
\end{proof}

\begin{remark}
If $f\in L^{p}\left(  \mathbb{R}^{d}\right)  ,$ then the small-set condition
$\left(  \ref{eq:small-set-density}\right)  $ holds with $s=1-1/p$. On the
other hand, if we write $s=1-1/p$ in $\left(  \ref{eq:small-set-density}%
\right)  ,$ then the small-set condition implies that $f\in L^{t}\left(
\mathbb{R}^{d}\right)  $ for $0<t<p,$ but it does not necessarily mean $f\in
L^{p}\left(  \mathbb{R}^{d}\right)  .$ As its name indicated, small-set
condition control behaviours on small scales.
\end{remark}

\subsection{Source with $n$th-moment and integrable density}

Define for $a\in\mathbb{R}^{d},$
\[
M_{\rho,n}\left(  a\right)  =\left(  \int_{\mathbb{R}^{d}}\left\vert
x-a\right\vert ^{n}f(x)dx\right)  ^{1/n}.
\]

\begin{proposition}
\label{prop:Lp-moment}Let $\rho=fdx\in\mathcal{P}\left(  \mathbb{R}%
^{d}\right)  ,$ $f\in L^{p}\left(  \mathbb{R}^{d}\right)  $, $1<p\leqslant
\infty$, $s=1-1/p,$ $a\in\mathbb{R}^{d}$ and suppose $M_{\rho,n}\left(
a\right)  <\infty$ for some $n\geqslant2$. Let%
\begin{equation}
\eta=\frac{sn}{n+\left(  d-1\right)  s}, \label{eq:A-eta}%
\end{equation}
Then for $S,$ $h>0,$%
\begin{equation}
\mathfrak{A}_{\rho}\left(  h\right)  \leqslant K_{d,\rho,p}S^{\left(
d-1\right)  s}h^{s}+4M_{\rho,n}\left(  a\right)  ^{n}S^{-n}.
\label{eq:A-Lp-moment}%
\end{equation}
where $K_{d,\rho,p}$ is in $\left(  \ref{eq:K-d-rho-p}\right)  .$
Consequently, for every $h>0$,
\begin{equation}
\mathfrak{A}_{\rho}\left(  h\right)  \leqslant B_{d,\rho,p,n}\left(  a\right)
h^{\eta}, \label{eq:A-Lp-moment-Opt}%
\end{equation}
where
\begin{equation}
B_{d,\rho,p,n}\left(  a\right)  =\left\{
\begin{array}
[c]{ll}%
c_{r,n}K_{d,\rho,p}^{n/\left(  n+r\right)  }\left[  4M_{\rho,n}\left(
a\right)  ^{n}\right]  ^{r/\left(  n+r\right)  }, & d>1,\\
K_{d,\rho,p}, & d=1,
\end{array}
\right.  \label{eq:B-d-rho-p-n}%
\end{equation}
the exponent $r=\left(  d-1\right)  s$ and $c_{r,n}$ is from $\left(
\ref{eq:two-power-constant}\right)  $.
\end{proposition}

\begin{proof}
Let $v$ be $1$-Lipschitz convex. On $B\left(  a,S\right)  $, the calculation
in the proof of Proposition~\ref{prop:compact-Lp} shows%
\[
\int_{B(a,S)}\omega_{v,h}^{2}f\,dx\leqslant K_{d,\rho,p}S^{\left(  d-1\right)
s}h^{s}.
\]
On $B\left(  a,S\right)  ^{c}$, we use the universal bound $\left(
\ref{eq:A-univ}\right)  $ and Markov's inequality,
\[
\int_{B(a,S)^{c}}\omega_{v,h}^{2}f\,dx\leqslant4\int_{B(a,S)^{c}}%
fdx\leqslant4M_{\rho,n}\left(  a\right)  ^{n}S^{-n}.
\]
Hence we get $\left(  \ref{eq:A-Lp-moment}\right)  $ by combining the two
estimates and taking supremum in $v$. Since$\ \left(  \ref{eq:A-Lp-moment}%
\right)  $ holds for every $S>0,$ we can choose $S$ that minimizes its RHS.
For $d>1,$ this is done by invoking Lemma \ref{lem:two-power}. For $d=1$, we
send $S\rightarrow\infty$ in $\left(  \ref{eq:A-Lp-moment}\right)  .$
Therefore the desired inequality $\left(  \ref{eq:A-Lp-moment-Opt}\right)  $
follows by substituting the chosen $S.$
\end{proof}

\subsection{Source on concentric annuli}

\begin{proposition}
\label{prop:annuli}Let $\rho=fdx\in\mathcal{P}\left(  \mathbb{R}^{d}\right)
,$ $a\in\mathbb{R}^{d},$ $1<p\leqslant\infty$ and $s=1-1/p$. Set
\[
A_{0}=B\left(  a,2\right)  ,\text{ }A_{k}=B\left(  a,2^{k+1}\right)  \setminus
B\left(  a,2^{k}\right)  ,\text{ }k\geqslant1,
\]
and define
\[
W_{p,a}\left(  f\right)  =\left\Vert f\right\Vert _{L^{p}\left(  A_{0}\right)
}+\sum_{k\geqslant1}2^{k\left(  d-1\right)  s}\left\Vert f\right\Vert
_{L^{p}\left(  A_{k}\right)  }.
\]
If $W_{p,a}\left(  f\right)  <\infty$, then for $h>0$,
\begin{equation}
\mathfrak{A}_{\rho}\left(  h\right)  \leqslant K_{d,p}^{\mathrm{ann}}%
W_{p,a}\left(  f\right)  h^{s}, \label{eq:A-annuli}%
\end{equation}
where $K_{d,p}^{\mathrm{ann}}=2^{\left(  d-1\right)  s}K_{d,p}^{L^{p}}.$
\end{proposition}

\begin{proof}
On each shell $A_{k},$ as in the proof of Proposition \ref{prop:compact-Lp} we
can write, for $k\geqslant0,$
\[
\int_{A_{k}}\omega_{v,h}^{2}fdx\leqslant K_{d,p}^{L^{p}}2^{\left(  k+1\right)
\left(  d-1\right)  s}\left\Vert f\right\Vert _{L^{p}\left(  A_{k}\right)
}h^{s}.
\]
We now sum the integrals over $k\geqslant0$ and get%
\[
\int\omega_{v,h}^{2}fdx\leqslant K_{d,p}^{L^{p}}h^{s}\sum_{k\geqslant
0}2^{\left(  k+1\right)  \left(  d-1\right)  s}\left\Vert f\right\Vert
_{L^{p}\left(  A_{k}\right)  }=2^{\left(  d-1\right)  s}K_{d,p}^{L^{p}}%
h^{s}\sum_{k\geqslant0}2^{k\left(  d-1\right)  s}\left\Vert f\right\Vert
_{L^{p}\left(  A_{k}\right)  }.
\]
The RHS is independent of $v$. Therefore taking the supremum in $v$ we get
$\left(  \ref{eq:A-annuli}\right)  .$
\end{proof}

\begin{remark}
For fast decaying source measures, both Proposition \ref{prop:Lp-moment} and
Proposition \ref{prop:annuli} apply, but the latter is more dedicated and
gives stronger exponent: $\eta\leqslant s$ with strict inequality when $d>1.$
\end{remark}

\section{Estimates of the target profile}

\label{sec:target-profile}

Let $m>2$ and $b\in\mathbb{R}^{d}$, define
\[
M_{\nu,m}\left(  b\right)  =\left(  \int\left\vert y-b\right\vert ^{m}%
d\nu\left(  y\right)  \right)  ^{1/m},\text{ }M_{\nu,m}=\inf_{b\in
\mathbb{R}^{d}}M_{\nu,m}\left(  b\right)  .
\]
Using the triangle inequality in $L^{m}\left(  \nu\right)  $ we have%
\[
\left\vert M_{\nu,m}\left(  b_{1}\right)  -M_{\nu,m}\left(  b_{2}\right)
\right\vert \leqslant\left\vert b_{1}-b_{2}\right\vert \text{ for }b_{1}%
,b_{2}\in\mathbb{R}^{d}\text{.}%
\]
Thus, whenever $M_{\nu,m}\left(  b_{0}\right)  <\infty$ for some $b_{0}$, the
function $b\mapsto M_{\nu,m}\left(  b\right)  ^{m}$ is continuous and
coercive. Hence its infimum is attained. For $m>1$ it is strictly convex and
the minimizer is unique. We retain a specified $b$ when geometric information
gives a natural center and use $M_{\nu,m}$ when the best estimate is desired.

\begin{proposition}
\label{prop:target-moment}Let $m>2$ and $b\in\mathbb{R}^{d}$ satisfy
$M_{\nu,m}\left(  b\right)  <\infty$. Then for $R>0,$
\begin{equation}
\mathfrak{T}_{\nu,b}\left(  R\right)  \leqslant8M_{\nu,m}^{m}\left(  b\right)
R^{2-m}. \label{eq:T-moment}%
\end{equation}
Consequently,
\[
\inf_{b\in\mathbb{R}^{d}}\mathfrak{T}_{\nu,b}\left(  R\right)  \leqslant
8M_{\nu,m}^{m}R^{2-m}.
\]

\end{proposition}

\begin{proof}
Since $m>2,$ we have $\left\vert y-b\right\vert ^{2}\leqslant R^{2-m}%
\left\vert y-b\right\vert ^{m}$ whenever $\left\vert y-b\right\vert >R$. Thus
\[
\tau_{2,\nu,b}\left(  R\right)  \leqslant R^{2-m}\int\left\vert y-b\right\vert
^{m}d\nu=M_{\nu,m}^{m}\left(  b\right)  R^{2-m}.
\]
Hence $\left(  \ref{eq:T-moment}\right)  $ follows from the exact calculation
$\left(  \ref{eq:tail-exact}\right)  .$
\end{proof}

\section{Profile optimization and the error bound}

\label{sec:profile-optimization}

This section will see the conversion of profile estimates into explicit power
laws for the error bound.

\begin{theorem}
[Compact-moment]\label{thm:power}Assume Assumption~\ref{ass:brenier}. Let
$\rho=fdx\in\mathcal{P}\left(  \mathbb{R}^{d}\right)  $ be supported in
$B\left(  a,R_{X}\right)  $, $R_{X}>0,$ $1<p\leqslant\infty,$ $f\in
L^{p}\left(  \mathbb{R}^{d}\right)  $ and the target $v$ satisfy $M_{\nu
,m}<\infty$ for $m>2$. Let
\[
s=1-\frac{1}{p},\text{ }H_{s,m}=m\left(  1+s\right)  -s.
\]
Then for $\sigma\in\Pi(\rho,\nu)$,
\begin{equation}
\mathcal{E}_{\nu}(\sigma)\leqslant Q_{d,\rho,p,\nu,m}R_{X}^{\left(
d-1\right)  s\left(  m-2\right)  /H_{s,m}}\mathcal{D}_{\nu}(\sigma
)^{s(m-2)/H_{s,m}}, \label{eq:power-full}%
\end{equation}
where
\begin{equation}
a_{s}=\frac{2+s}{1+s},\text{ \ }B_{d,\rho,p}=c_{s,1}18^{1/\left(  1+s\right)
}4^{s/(1+s)}K_{d,\rho,p}^{1/\left(  1+s\right)  }, \label{eq:source-const}%
\end{equation}%
\[
Q_{d,\rho,p,\nu,m}=c_{a_{s},m-2}8^{a_{s}/\left(  a_{s}+m-2\right)  }%
B_{d,\rho,p}^{\left(  m-2\right)  \left(  1+s\right)  /H_{s,m}}M_{\nu
,m}^{m(2+s)/H_{s,m}},
\]
and $K_{d,\rho,p}$ is given in $\left(  \ref{eq:K-d-rho-p}\right)  $ and
$c_{a_{s},m-2}$, $c_{s,1}$ are from $\left(  \ref{eq:two-power-constant}%
\right)  .$ The objective gap exponent is
\begin{equation}
\theta_{s,m}=\frac{s\left(  m-2\right)  }{H_{s,m}}=\frac{\left(  p-1\right)
\left(  m-2\right)  }{2mp-m-p+1}. \label{eq:theta}%
\end{equation}
The constant in $\left(  \ref{eq:power-full}\right)  $ is independent of the
source center.
\end{theorem}

\begin{proof}
Let $b_{m}$ be the minimizer in the definition of $M_{\nu,m}$. Applying
Proposition~\ref{prop:compact-Lp} to the source profile,
Proposition~\ref{prop:target-moment} to the target profile, and then
Theorem~\ref{thm:error-bound} with center $b_{m}$, we obtain for $R,$ $h>0$,%
\begin{equation}
\mathcal{E}_{\nu}\left(  \sigma\right)  \leqslant4R\mathcal{D}_{\nu}\left(
\sigma\right)  h^{-1}+18K_{d,p}^{L^{p}}\left\Vert f\right\Vert _{L^{p}}%
R_{X}^{\left(  d-1\right)  s}R^{2}h^{s}+8M_{\nu,m}^{m}R^{2-m}.
\label{eq:pre-opt}%
\end{equation}
If $\mathcal{D}_{\nu}(\sigma)=0$, letting $h\rightarrow0$ then $R\rightarrow
\infty$ yields that $\mathcal{E}_{\nu}(\sigma)=0$. So $\left(
\ref{eq:power-full}\right)  $ holds trivially. If $M_{\nu,m}=0$, then
$\nu=\delta_{b_{m}}$, and every feasible plan is supported on $\mathbb{R}%
^{d}\times\{b_{m}\}$, hence again $\mathcal{E}_{\nu}(\sigma)$ is zero and
$\left(  \ref{eq:power-full}\right)  $ holds. Assume henceforth that
$\mathcal{D}_{\nu}(\sigma)>0$ and $M_{\nu,m}>0$. For fixed $R>0$. Applying
Lemma \ref{lem:two-power} we find the minimum of the sum of the first two
terms in $\left(  \ref{eq:pre-opt}\right)  ,$%
\begin{align*}
&  c_{s,1}\left[  18K_{d,\rho,p}R_{X}^{\left(  d-1\right)  s}R^{2}\right]
^{1/(1+s)}\left[  4R\mathcal{D}_{\nu}(\sigma)\right]  ^{s/(1+s)}\\
&  =c_{s,1}18^{1/\left(  1+s\right)  }4^{s/(1+s)}K_{d,\rho,p}^{1/\left(
1+s\right)  }R_{X}^{\left(  d-1\right)  s/\left(  1+s\right)  }R^{(2+s)/(1+s)}%
\mathcal{D}_{\nu}(\sigma)^{s/(1+s)}\\
&  =B_{d,\rho,p}R_{X}^{\left(  d-1\right)  s/\left(  1+s\right)
}R^{(2+s)/(1+s)}\mathcal{D}_{\nu}(\sigma)^{s/(1+s)}.
\end{align*}
Substituting it into $\left(  \ref{eq:pre-opt}\right)  $ we get
\begin{equation}
\mathcal{E}_{\nu}(\sigma)\leqslant B_{d,\rho,p}R_{X}^{\left(  d-1\right)
s/\left(  1+s\right)  }R^{(2+s)/(1+s)}\mathcal{D}_{\nu}(\sigma)^{s/(1+s)}%
+8M_{\nu,m}^{m}R^{2-m}. \label{eq:after-h}%
\end{equation}
We now unfreeze $R$ and apply Lemma \ref{lem:two-power} to the sum in $\left(
\ref{eq:after-h}\right)  $ to find the minimum%
\[
c_{a_{s},m-2}\left[  B_{d,\rho,p}R_{X}^{\left(  d-1\right)  s/\left(
1+s\right)  }\mathcal{D}_{\nu}(\sigma)^{s/(1+s)}\right]  ^{\left(  m-2\right)
/\left(  a_{s}+m-2\right)  }\left[  8M_{\nu,m}^{m}\right]  ^{a_{s}/\left(
a_{s}+m-2\right)  }.
\]
To work out the exponents we calculate%
\[
a_{s}+m-2=\frac{2+s+\left(  m-2\right)  \left(  1+s\right)  }{1+s}%
=\frac{m\left(  1+s\right)  -s}{1+s}=\frac{H_{s,m}}{1+s},
\]%
\[
m\frac{a_{s}}{a_{s}+m-2}=\frac{m(2+s)}{H_{s,m}},\text{ \ \ \ \ \ }\frac
{s}{1+s}\frac{m-2}{a_{s}+m-2}=\frac{s\left(  m-2\right)  }{H_{s,m}}.
\]
Substituting the calculations into $\left(  \ref{eq:after-h}\right)  $ we
obtain%
\[
\mathcal{E}_{\nu}(\sigma)\leqslant Q_{d,\rho,p,\nu,m}R_{X}^{\left(
d-1\right)  s\left(  m-2\right)  /H_{s,m}}\mathcal{D}_{\nu}(\sigma
)^{s(m-2)/H_{s,m}}.
\]
Therefore the desired bound $\left(  \ref{eq:power-full}\right)  $ is proved.
\end{proof}

\begin{corollary}
[Compact-compact]\label{cor:compact-target}Assume Assumption~\ref{ass:brenier}%
. Suppose $\rho=fdx\in\mathcal{P}\left(  \mathbb{R}^{d}\right)  $ is supported
in $B\left(  a,R_{X}\right)  $, $R_{X}>0$, $f\in L^{p}\left(  \mathbb{R}%
^{d}\right)  $, $1<p\leqslant\infty$, $s=1-1/p$ and the target $\nu$ is
supported in $B(b,R_{Y})$. Then for $\sigma\in\Pi(\rho,\nu)$,
\begin{equation}
\mathcal{E}_{\nu}(\sigma)\leqslant B_{d,\rho,p}R_{X}^{\left(  d-1\right)
s/\left(  1+s\right)  }R_{Y}^{(2+s)/(1+s)}\mathcal{D}_{\nu}(\sigma)^{s/(1+s)},
\label{eq:compact-target-E}%
\end{equation}
where $B_{d,\rho,p}$ is defined in $\left(  \ref{eq:source-const}\right)  .$
If instead $\rho$ satisfies the small-set condition $\left(
\ref{eq:small-set-density}\right)  $ with $0<s\leqslant1$, then $\left(
\ref{eq:compact-target-E}\right)  $ holds with $B_{s}$ in place of
$B_{d,\rho,p}$ in $\left(  \ref{eq:small-set-Bs}\right)  $. For $p=\infty$,
this is the squared map error estimate with exponent $1/2$ of \cite[Theorem
3.4 (3.15)]{merigot2026sharp}.
\end{corollary}

\begin{proof}
If $R_{Y}=0$, then $\nu$ is a point mass and every feasible plan has zero map
error. If $\mathcal{D}_{\nu}(\sigma)=0$, the same conclusion follows from the
uniqueness in Assumption~\ref{ass:brenier}. Assume henceforth that $R_{Y}>0$
and $\mathcal{D}_{\nu}(\sigma)>0$. Substitute $R=R_{Y}$ and the specified
center $b$ of the supporting ball $B(b,R_{Y})$ into $\left(
\ref{eq:error-bound}\right)  $. The bound there only gets increased. Since
$\nu\left(  \left\vert y-b\right\vert >R_{Y}\right)  =0$, both terms in the
sum of $\left(  \ref{eq:Tdef}\right)  $ vanish and $\mathfrak{T}_{\nu,b}%
(R_{Y})=0$. So we get the same error bound $\left(  \ref{eq:pre-opt}\right)  $
as in the proof of Theorem \ref{thm:power}$\ $except that the third term
becomes zero in this compact target case,
\[
\mathcal{E}_{\nu}(\sigma)\leqslant4R_{Y}\mathcal{D}_{\nu}(\sigma
)h^{-1}+18K_{d,\rho,p}R_{Y}^{2}h^{s}.
\]
So the calculations in that proof are still valid. Replacing $R$ with $R_{Y}$
in $\left(  \ref{eq:after-h}\right)  $ (excluding the second term in the sum)
therefore gives%
\[
\mathcal{E}_{\nu}(\sigma)\leqslant B_{d,\rho,p}R_{X}^{\left(  d-1\right)
s/\left(  1+s\right)  }R_{Y}^{(2+s)/(1+s)}\mathcal{D}_{\nu}(\sigma
)^{s/(1+s)}.
\]
When $p=\infty$, $s=1$, and hence $s/(1+s)=1/2$. The proof for small-set
folllows similar lines, so we omit it.
\end{proof}

The source profile bound in Proposition~\ref{prop:Lp-moment} allows us to
obtain error bound for source and target that may possess unbounded density or supports.

\begin{theorem}
[moment-moment]\label{thm:unbounded-power}Assume Assumption~\ref{ass:brenier}
holds, let $\rho=fdx\in\mathcal{P}\left(  \mathbb{R}^{d}\right)  $, $f\in
L^{p}\left(  \mathbb{R}^{d}\right)  $, $1<p\leqslant\infty$, $s=1-1/p$ and
$M_{\rho,n}(a)<\infty$ for some $a\in$ $\mathbb{R}^{d}$, $n\geqslant2$.
Suppose the target $\nu$ has $M_{\nu,m}<\infty$ for some $m>2$. Let
\[
\eta=\frac{sn}{n+(d-1)s},\text{ \ }H_{\eta,m}=m\left(  1+\eta\right)
-\eta,\text{ }%
\]
Then for $\sigma\in\Pi(\rho,\nu)$,
\begin{equation}
\mathcal{E}_{\nu}(\sigma)\leqslant Q_{d,\rho,p,n,\nu,m}\left(  a\right)
\mathcal{D}_{\nu}(\sigma)^{\eta(m-2)/H_{\eta,m}}, \label{eq:unbounded-E}%
\end{equation}
where%
\begin{equation}
a_{\eta}=\frac{2+\eta}{1+\eta},\text{ \ }B_{d,\rho,p,n,c}\left(  a\right)
=c_{\eta,1}18^{1/\left(  1+\eta\right)  }4^{\eta/(1+\eta)}\left[
B_{d,\rho,p,n}\left(  a\right)  \right]  ^{1/\left(  1+\eta\right)  },
\label{eq:B-d-rho-p-n-c}%
\end{equation}%
\[
Q_{d,\rho,p,n,\nu,m}\left(  a\right)  =c_{a_{\eta},m-2}8^{a_{\eta}/\left(
a_{\eta}+m-2\right)  }\left[  B_{d,\rho,p,n,c}\left(  a\right)  \right]
^{\left(  m-2\right)  \left(  1+\eta\right)  /H_{\eta,m}}M_{\nu,m}%
^{m(2+\eta)/H_{\eta,m}},
\]
$B_{d,\rho,p,n}\left(  a\right)  $ is defined in $\left(  \ref{eq:B-d-rho-p-n}%
\right)  $. The objective gap exponent is%
\begin{equation}
\Theta_{\eta,m}=\frac{\eta(m-2)}{H_{\eta,m}}. \label{eq:Theta}%
\end{equation}
If $d=1,$ then $\eta=s$ and $B_{d,\rho,p,n}\left(  a\right)  $ is independent
of the source moment. The exponent is sharp for $d>1$ by
Theorem~\ref{thm:escaping-slabs}.
\end{theorem}

\begin{proof}
Assume first that $\mathcal{D}_{\nu}(\sigma)>0$ and $M_{\nu,m}>0$. The proof
follows the lines of Theorem \ref{thm:power}, but for the source profile we
invoke Proposition~\ref{prop:Lp-moment}, hence for all $h,$ $R>0$,
\begin{equation}
\mathcal{E}_{\nu}(\sigma)\leqslant4R\mathcal{D}_{\nu}(\sigma)h^{-1}%
+18B_{d,\rho,p,n}\left(  a\right)  R^{2}h^{\eta}+8M_{\nu,m}^{m}R^{2-m}.
\label{eq:pre-opt-nm}%
\end{equation}
This is $\left(  \ref{eq:pre-opt}\right)  $ with $K_{d,\rho,p}$ replaced with
$B_{d,\rho,p,n}\left(  a\right)  $ and $R_{X}$ there is \textquotedblleft
absorbed" into the exponent $\eta.$ So the resulting bound has the same form
as $\left(  \ref{eq:power-full}\right)  $ except that we need to replace $s$
with $\eta,$ $\theta_{s,m}$ with $\Theta_{\eta,m},$ and $R_{X}$ with $1.$ If
$\mathcal{D}_{\nu}(\sigma)=0$, let $h\rightarrow0$ at fixed $R$ in the initial
bound, and then let $R\rightarrow\infty$, using $m>2$. If $M_{\nu,m}=0$, the
target is a Dirac mass and the map error is zero. The proof is complete.
\end{proof}

The specialization to source with $n$th-moment and compact target is
analogously proved by looking at the first two terms in the sum $\left(
\ref{eq:pre-opt-nm}\right)  .$ We state the following result without proof.

\begin{corollary}
[moment-compact]\label{cor:moment-compactE}Assume that
Assumption~\ref{ass:brenier} holds, $\rho=fdx\in\mathcal{P}\left(
\mathbb{R}^{d}\right)  $, $f\in L^{p}\left(  \mathbb{R}^{d}\right)  $,
$1<p\leqslant\infty$, $s=1-1/p$, $M_{\rho,n}(a)<\infty$ for some $a\in$
$\mathbb{R}^{d}$, $n\geqslant2$ and the target $\nu$ is supported in
$B(b,R_{Y})$. Let $\eta$ be defined as in Theorem \ref{thm:unbounded-power}.
Then for $\sigma\in\Pi(\rho,\nu)$,%
\[
\mathcal{E}_{\nu}(\sigma)\leqslant B_{d,\rho,p,n,c}\left(  a\right)
R_{Y}^{(2+\eta)/(1+\eta)}\mathcal{D}_{\nu}(\sigma)^{\eta/(1+\eta)},
\]
where $B_{d,\rho,p,n,c}\left(  a\right)  $ is defined in $\left(
\ref{eq:B-d-rho-p-n-c}\right)  .$
\end{corollary}

\section{Sharpness and extreme cases}

\label{sec:sharpness}

The section tries to determine the sharpness of the exponents of the error
bound obtained in the preceding section. Here the sharpness means with a given
class of probability marginals, the exponent cannot be increased uniformly for
all pairs of marginals in that class.

\subsection{Sharpness of Theorem~\ref{thm:power}}

The next theorem establishes sharpness of the map error exponent in
Theorem~\ref{thm:power}. The example below does not uniformly satisfy the
strong assumptions of Corollary~\ref{cor:compact-target}, there is no common
bounded target support as $h\rightarrow0$.

\begin{theorem}
\label{thm:sharp}Fix $d\geqslant1$, $1<p\leqslant\infty$, $m>2$ and $s=1-1/p$.
There are $\rho_{h}=f_{h}dx$ supported in $\left[  0,1\right]  ^{d}$ for all
small $h>0$, target $\nu_{h}$ and plans $\sigma_{h}\in\Pi(\rho_{h},\nu_{h})$
such that
\[
\sup_{h}\left\Vert f_{h}\right\Vert _{L^{p}}<\infty,\text{ }\sup_{h}%
\int\left\vert y\right\vert ^{m}d\nu_{h}\left(  y\right)  <\infty,
\]
and
\begin{equation}
\mathcal{E}_{\nu_{h}}\left(  \sigma_{h}\right)  =2\mathcal{D}_{\nu_{h}}\left(
\sigma_{h}\right)  ^{\theta_{s,m}}. \label{eq:sharp-relation}%
\end{equation}
Here $\theta_{s,m}$ is exactly the exponent in Theorem~\ref{thm:power},
defined by $\left(  \ref{eq:theta}\right)  $. Thus the exponent there is sharp
and uniform in the given class of marginals.
\end{theorem}

\begin{proof}
First consider $d=1$. Let $h>0$ small enough so that $h<1/4$ and $q_{h}%
=h^{s}<1/2.$ Let
\[
I_{-}=\left[  1-2h,1-h\right]  ,\text{ }I_{+}=\left[  1-h,1\right]  ,
\]
and define on $[0,1]$
\[
f_{h}\left(  x\right)  =\left\{
\begin{array}
[c]{ll}%
\left(  1-2q_{h}\right)  /\left(  1-2h\right)  , & x\in\left[  0,1-2h\right)
.\\
q_{h}/h, & x\in I_{-}\cup I_{+},
\end{array}
\right.
\]
Each of $I_{-},$ $I_{+}$ has mass $q_{h}$, while the remaining interval has
mass $1-2q_{h}$. Thus $f_{h}\geqslant0$ and $\int_{0}^{1}f_{h}\,dx=1$. We now
verify that $\sup_{h}\left\Vert f_{h}\right\Vert _{L^{p}}<\infty.$ For
$p=\infty$, $s=1$, $q_{h}=h$, so $f_{h}\equiv1$ and $\left\Vert f_{h}%
\right\Vert _{L^{p}}=1.$ For $p<\infty$, we have on $I_{-}\cup I_{+},$
\[
2h\left(  q_{h}/h\right)  ^{p}=2h\left(  h^{s-1}\right)  ^{p}=2hh^{-1}=2,
\]
On the remaining interval,
\[
0\leqslant\frac{1-2q_{h}}{1-2h}\leqslant\frac{1}{1-2h}\leqslant2.
\]
Hence $\sup_{h}\left\Vert f_{h}\right\Vert _{L^{p}}<\infty$.

Set $R_{h}=q_{h}^{-1/m}$ and define
\[
T_{h}(x)=%
\begin{cases}
0, & x\in\left[  0,1-h\right)  ,\\
R_{h}, & x\in\left[  1-h,1\right]  .
\end{cases}
\]
Up to a null set, $T_{h}$ equals the derivative the convex function%
\[
u_{h}(x)=R_{h}\max\left\{  0,x-\left(  1-h\right)  \right\}  .
\]
The pushforward of $\rho_{h}$ under $T_{h}$ is
\[
\nu_{h}=\left(  1-q_{h}\right)  \delta_{0}+q_{h}\delta_{R_{h}}.
\]
By uniqueness of Brenier map, $T_{h}$ is the optimal map between $\rho_{h}$
and $\nu_{h}$. The moments of $\nu_{h}$ are uniformly bounded,%
\[
\int\left\vert y\right\vert ^{m}d\nu_{h}=q_{h}R_{h}^{m}=1.
\]
So the construction satisfies all requirements.

Let $S_{h}$ be the map that swaps $I_{-}$ and $I_{+}$ by translations of
length $h$ and leaves the remaining unchanged. Since the densities on $I_{-}$
and $I_{+}$ are equal, we have $\left(  S_{h}\right)  _{\#}\rho_{h}=\rho_{h}$.
Thus $\left(  \operatorname{Id},S_{h}\right)  _{\#}\rho_{h}\in\Pi\left(
\rho_{h},\rho_{h}\right)  $. Define $\sigma_{h}=\left(  \operatorname{Id}%
,T_{h}\circ S_{h}\right)  _{\#}\rho_{h}$, where $T_{h}\circ S_{h}$ is the
composed transport map%
\[
T_{h}\circ S_{h}(x)=%
\begin{cases}
0, & x\in\left[  0,1-2h\right)  ,\\
R_{h} & x\in\left[  1-2h,1-h\right)  ,\\
0, & x\in\left[  1-h,1\right]  ,
\end{cases}
\]
Substituting $T_{h}$ and $\sigma_{h}$ into $\left(  \ref{eq:Eintro-map}%
\right)  $ we get
\begin{equation}
\mathcal{E}_{\nu_{h}}\left(  \sigma_{h}\right)  =2q_{h}R_{h}^{2}=2h^{s\left(
1-2/m\right)  }. \label{eq:E-sharp}%
\end{equation}
By $\left(  \ref{eq:D-corr}\right)  $ in Lemma~\ref{lem:lift} with
$q=(\operatorname{Id},S_{h})_{\#}\rho_{h}$,%
\begin{equation}
\mathcal{D}_{\nu_{h}}\left(  \sigma_{h}\right)  =R_{h}\frac{q_{h}}{h}\left(
\int_{I_{+}}xdx-\int_{I_{-}}xdx\right)  =R_{h}\frac{q_{h}}{h}h^{2}%
=h^{1+s-s/m}. \label{eq:D-sharp}%
\end{equation}
Since
\[
s(1-\frac{2}{m})=\frac{s\left(  m-2\right)  }{m},\text{ }1+s-\frac{s}{m}%
=\frac{m(1+s)-s}{m},
\]
so $\left(  \ref{eq:sharp-relation}\right)  $ is proved in dimension one.

For $d>1$, define
\[
\bar{\rho}_{h}=\rho_{h}\otimes\mathcal{L}^{d-1}|_{\left[  0,1\right]  ^{d-1}%
},\text{ }\bar{T}_{h}\left(  x_{1},x^{\prime}\right)  =\left(  T_{h}\left(
x_{1}\right)  ,0\right)  .
\]
Up to a null set, the map is the gradient of the convex function
\[
\bar{u}_{h}\left(  x_{1},x^{\prime}\right)  =\int_{0}^{x_{1}}T_{h}\left(
t\right)  dt,
\]
which depends only on $x_{1}$. Swapping $I_{-}\times\left[  0,1\right]
^{d-1}$ and $I_{+}\times\left[  0,1\right]  ^{d-1}$ does not change
$x^{\prime}$-coordinates, so the calculations in the one-dimensional case are
unchanged, particularly the $L^{p}$ norms and the quantities in $\left(
\ref{eq:E-sharp}\right)  $, $\left(  \ref{eq:D-sharp}\right)  $ are unchanged.
Therefore the relation $\left(  \ref{eq:sharp-relation}\right)  $ holds in
general dimension.
\end{proof}

The next examples have no class uniform modulus. Their objective gaps
converges to zero while their map errors stay bounded away from zero.

\begin{theorem}
\label{thm:extremes}If $p=1$ or $m=2,$ there is no class uniform modulus for
which the bound of the form $\left(  \ref{eq:power-full}\right)  $ in Theorem
\ref{thm:power} holds, even for compact targets.
\end{theorem}

\begin{proof}
The examples run parallel to Theorem \ref{thm:sharp}. Restrict first to $d=1$
and extend to higher dimension by tensorizing. For the extreme case $p=1$. Let
$q_{h}$ to a fixed $q_{0}\in\left(  0,1/2\right)  $ and take $R_{h}=1$. The
constructions of $\rho_{h}=f_{h}dx,$ $\nu_{h},$ $T_{h},$ $\sigma_{h}$ and
calculations in Theorem \ref{thm:sharp} remain valid except that we need to
replace $q_{h}$ with $q_{0}$. In particular, each $\rho_{h}$ is supported in
$\left[  0,1\right]  $ and each target $\nu_{h}$ in $\left\{  0,1\right\}  $.
Therefore all source and target moments, of any finite order, are bounded
uniformly. Hence we get from $\left(  \ref{eq:E-sharp}\right)  $ and $\left(
\ref{eq:D-sharp}\right)  $ that%
\[
\mathcal{E}_{\nu_{h}}\left(  \sigma_{h}\right)  =2q_{0},\text{ }%
\mathcal{D}_{\nu_{h}}\left(  \sigma_{h}\right)  =R_{h}\frac{q_{0}}{h}%
h^{2}=q_{0}h\rightarrow0.
\]
This proves that boundedness of $L^{1}$-norm of the source measure alone
cannot provide a class-uniform vanishing modulus, even if the source has
densities and targets have uniformly bounded supports.

For extreme case $m=2$. The constructions and calculations in Theorem
\ref{thm:sharp} continue to hold if we set $p=\infty.$ Therefore $s=1$,
$q_{h}=h$ and $\rho_{h}$ is uniform on $\left[  0,1\right]  $. Set
$R_{h}=q_{h}^{-1/2}$. Then we see that the targets have uniformly bounded
second moments,
\[
\int\left\vert y\right\vert ^{2}d\nu_{h}=q_{h}R_{h}^{2}=1.
\]
However,
\[
\mathcal{E}_{\nu_{h}}\left(  \sigma_{h}\right)  =2,\text{ }\mathcal{D}%
_{\nu_{h}}\left(  \sigma_{h}\right)  =q_{h}^{1/2}h\rightarrow0.
\]

\end{proof}

\subsection{Sharpness of Theorem~\ref{thm:unbounded-power}}

\begin{theorem}
\label{thm:escaping-slabs}Fix $d\geqslant1$, $1<p\leqslant\infty$,
$n\geqslant2$ and $m>2$. Let
\[
s=1-\frac{1}{p},\text{ }\eta=\frac{sn}{n+(d-1)s},\text{ }\Theta=\frac
{\eta(m-2)}{m(1+\eta)-\eta}.
\]
There exist $\rho_{h}=f_{h}dx$, targets $\nu_{h}$ and deterministic feasible
plans $\sigma_{h}$ for all small $h>0$, such that
\[
\sup_{h}\left\Vert f_{h}\right\Vert _{p}<\infty,\text{ }\sup_{h}\int\left\vert
x\right\vert ^{n}\,d\rho_{h}<\infty,\int\left\vert y\right\vert ^{m}\,d\nu
_{h}=1,\text{ }\mathcal{E}_{\nu_{h}}\left(  \sigma_{h}\right)  =2\mathcal{D}%
_{\nu_{h}}\left(  \sigma_{h}\right)  ^{\Theta}.
\]
The objective gaps converge to zero, and all pairs satisfy
Assumption~\ref{ass:brenier}. Consequently the exponent in
Theorem~\ref{thm:unbounded-power} is sharp for classes with uniform bounds and
no common source support radius. For $d>1$, the source profile exponent $\eta$
defined in $\left(  \ref{eq:A-eta}\right)  $, which is part of the exponent in
Theorem~\ref{thm:unbounded-power}, is also sharp for this class.
\end{theorem}

\begin{proof}
We use the construction from Step 5 in the proof of
Theorem~\ref{thm:class-law}. That construction is proved directly from the
error bound and source estimates and does not use the present theorem. With
$\alpha=\xi=K=1$ and $\beta=0$,
\[
S=q^{-1/n},\text{ }h=q^{1/\eta},\text{ }q=R^{-m},\text{ }u_{h}\left(
x\right)  =R\max\left\{  0,x_{1}\right\}  .
\]
We now use the formulas from $\left(  \ref{eq:log-lower-E}\right)  $ and
$\left(  \ref{eq:log-lower-D}\right)  $ to compute the exact error and
objective gap,%
\[
\mathcal{E}_{\nu_{h}}(\sigma_{h})=2qR^{2}=2h^{\eta(1-2/m)},\text{ }%
\mathcal{D}_{\nu_{h}}(\sigma_{h})=hqR=h^{1+\eta-\eta/m}.
\]
Since
\[
\frac{\eta(1-2/m)}{1+\eta-\eta/m}=\frac{\eta(m-2)}{m(1+\eta)-\eta}=\Theta,
\]
so the construction satisfies all requirements and saturates the exponent in
Theorem~\ref{thm:unbounded-power}.

For the sharpness of exponent $\eta$ defined in $\left(  \ref{eq:A-eta}%
\right)  $, consider the convex $1$-Lipschitz function $v\left(  x\right)
=\left\vert x_{1}\right\vert $ which has $\omega_{v,h}\left(  x\right)  =2$ on
$A_{-}\cup A_{+}$. Each ball with radius $h$ that is admissible to the
defintion of $\omega_{v,h}\left(  x\right)  $ meets $x_{1}=0$, at which the
subdifferential of $v$ contains $e_{1}$ and $-e_{1}$. $A_{-}\cup A_{+}$ has
total source mass $2q$, so
\[
\mathfrak{A}_{\rho_{h}}\left(  h\right)  \geqslant4\left(  2q\right)
=8h^{\eta}.
\]
This together with the uniform upper estimate $\left(
\ref{eq:A-Lp-moment-Opt}\right)  $ therefore gives the matching order.
\end{proof}

\section{Sharp conditioning with logarithmic moments}

\label{sec:class-law}\label{sec:class-proof}

Throughout this section, we fix $d\geqslant1$, $1<p\leqslant\infty$, and
$n\geqslant2$. Define
\[
s=1-1/p,\text{ }\eta=\frac{sn}{n+\left(  d-1\right)  s},\text{ }F_{p}=\left\{
\begin{array}
[c]{ll}%
3^{1/p}, & p<\infty,\\
1, & p=\infty,
\end{array}
\right.  ,\text{ }M_{d,n}=\left(  d+3\right)  ^{n/2}+2\left(  4d-3\right)
^{n/2}.
\]
Fix $m\geqslant2$, $\beta\geqslant0$ and scaling constants $\alpha\geqslant1$,
$\xi\geqslant1$, $K>0$. Let $\mathcal{C}_{m,\beta}=\mathcal{C}_{m,\beta
}\left(  \alpha,\xi,K\right)  $ be the class consisting of probability pairs
$\left(  \rho,\nu\right)  $ satisfying
\begin{equation}
\rho=fdx,\text{ }\left\Vert f\right\Vert _{p}\leqslant\alpha F_{p},\text{
}\int\left\vert x\right\vert ^{n}d\rho\leqslant\xi M_{d,n},\text{ }%
\int\left\vert y\right\vert ^{m}\left[  \log\left(  e+\left\vert y\right\vert
\right)  \right]  ^{\beta}d\nu\leqslant K. \label{eq:normalized-class}%
\end{equation}

Note that each pair $\left(  \rho,\nu\right)  $ in the class has $\rho,\nu
\in\mathcal{P}_{2}$. The class $\mathcal{C}_{m,\beta}\left(  \alpha
,\xi,K\right)  $ is non-empty for all scaling choice $\alpha\geqslant1$,
$\xi\geqslant1$, $K>0$. Indeed, let $\rho=dx$ on%
\[
F=\left[  -2,-1\right]  \times\left[  0,1\right]  ^{d-1}.
\]
It has $L^{p}$-norm one and $n$-th moment at most $\left(  d+3\right)
^{n/2}.$ Let $\nu=\delta_{0}$. Then $\left(  \rho,\nu\right)  \in
\mathcal{C}_{m,\beta}\left(  1,1,1\right)  .$

Define for a class $\mathcal{C}$ of probaility pairs, $t>0,$
\begin{equation}
\Omega_{\mathcal{C}}\left(  t\right)  =\left\{  \mathcal{E}_{\nu}\left(
\sigma\right)  :\left(  \rho,\nu\right)  \in\mathcal{C},\text{ }\sigma\in
\Pi\left(  \rho,\nu\right)  ,\text{ }\mathcal{D}_{\nu}\left(  \sigma\right)
\leqslant t\right\}  . \label{eq:class-modulus}%
\end{equation}

\begin{theorem}
[Sharp power-logarithm conditioning law]\label{thm:class-law}For the fixed
non-empty class in $\left(  \ref{eq:normalized-class}\right)  $, let
\[
H=m(1+\eta)-\eta,\text{ }\Theta=\frac{\eta(m-2)}{H},\text{ }\Lambda
=\frac{\beta(2+\eta)}{H}.
\]
Define
\begin{equation}
\varphi(t)=t^{\Theta}[\log(e/t)]^{-\Lambda},\text{ }0<t<1. \label{eq:phiDef}%
\end{equation}
Then the following hold.

$\left(  i\right)  $ There exist $c,$ $C,$ $t_{0}>0$, depending only on $d,$
$p,$ $n,$ $m,$ $\beta,$ $\alpha,$ $\xi,$ $K$, with $t_{0}<1$, such that
\begin{equation}
c\varphi(t)\leqslant\Omega_{\mathcal{C}_{m,\beta}}(t)\leqslant C\varphi
(t),\text{ }0<t<t_{0}. \label{eq:sharp-class-law}%
\end{equation}
The lower bound is realized by deterministic feasible plans with objective gap
exactly $t$.

$\left(  ii\right)  $ The same matching bounds $\left(
\ref{eq:sharp-class-law}\right)  $ hold when the map error $\mathcal{E}_{\nu
}\left(  \sigma\right)  $ is replaced in $\left(  \ref{eq:class-modulus}%
\right)  $ by barycentric map error $\Vert\bar{T}_{\sigma}-T_{\nu}\Vert
_{L^{2}(\rho)}^{2}.$

$\left(  iii\right)  $ For $m>2,$ $\beta=0$, this consolidates the power law.
For $m=2,$ $\beta>0$, the modulus is comparable to $[\log(e/t)]^{-\beta}$. For
$m=2,\beta=0$, no class-uniform vanishing modulus exists.
\end{theorem}

\begin{proof}
\textbf{1}. \textit{Compute the class-uniform source bound}. Fix $0<t<1$, any
$(\rho,\nu)\in\mathcal{C}_{m,\beta}$, and any $\sigma\in\Pi(\rho,\nu)$ with
$\mathcal{D}_{\nu}(\sigma)\leqslant t$. We first prove an estimate independent
of this choice. Since $n\geqslant2$, we have
\[
\int\left\vert x\right\vert ^{2}d\rho\leqslant\left(  \int\left\vert
x\right\vert ^{n}d\rho\right)  ^{2/n}<\infty.
\]
Since $\log\left(  e+r\right)  \geqslant1$ for $r\geqslant0$, the target
constraint implies $\int\left\vert y\right\vert ^{m}\,d\nu\leqslant K$. Hence
by H\"{o}lder's inequality and that $m\geqslant2,$ we have $\int\left\vert
y\right\vert ^{2}d\nu\leqslant K^{2/m}$. These ensure that the Brenier setting
of Assumption~\ref{ass:brenier} holds for each pair in $\mathcal{C}_{m,\beta}%
$. In particular, Theorem \ref{thm:error-bound} is applicable. Since error
upper bound in Proposition~\ref{prop:Lp-moment} is nondecreasing in both
$\Vert f\Vert_{p}$ and $\int|x|^{n}d\rho$, we can substitute their bounds
$\alpha F_{p},$ $\xi M_{d,n}$ into $\left(  \ref{eq:A-Lp-moment-Opt}\right)
$\ with souce center $a=0$ to obtain,
\[
\mathfrak{A}_{\rho}\left(  h\right)  \leqslant Bh^{\eta},\text{ }h>0,
\]
where $B=c_{r,n}(K_{d,p}^{L^{p}}\alpha F_{p})^{n/(n+r)}\left(  4\xi
M_{d,n}\right)  ^{r/\left(  n+r\right)  }$ if $d>0,$ $K_{1,p}^{L^{p}}\alpha
F_{p}$ if $d=1.$ The constant is strictly positive and depends only on $d,$
$p,$ $n,$ $\alpha,$ $\xi$.

\textbf{2}. \textit{Derive the target squared-tail bound}. Let $R\geqslant1$
denote the target tail cutoff radius. Since $m\geqslant2$ and $\beta
\geqslant0$, the function $z\mapsto z^{m-2}\left[  \log\left(  e+z\right)
\right]  ^{\beta}$ is nondecreasing for $z\geqslant1$. Hence using the target
constraints $\left(  \ref{eq:normalized-class}\right)  $ we have%
\begin{align}
\tau_{2,\nu,0}\left(  R\right)   &  =\int_{\left\{  \left\vert y\right\vert
>R\right\}  }\left\vert y\right\vert ^{2}d\nu\left(  y\right)
\label{eq:weighted-tail}\\
&  \leqslant R^{2-m}\left[  \log\left(  e+R\right)  \right]  ^{-\beta}%
\int_{\left\{  \left\vert y\right\vert >R\right\}  }\left\vert y\right\vert
^{m}\left[  \log\left(  e+\left\vert y\right\vert \right)  \right]  ^{\beta
}d\nu\left(  y\right) \nonumber\\
&  \leqslant KR^{2-m}\left[  \log\left(  e+R\right)  \right]  ^{-\beta
}.\nonumber
\end{align}

\textbf{3}. \textit{Choose the right source scale}. With the estimates from
the preceding steps, we employ Theorem~\ref{thm:error-bound} with $b=0$ to
get
\[
\mathcal{E}_{\nu}\left(  \sigma\right)  \leqslant4Rth^{-1}+18BR^{2}h^{\eta
}+8KR^{2-m}\left[  \log\left(  e+R\right)  \right]  ^{-\beta},\text{
}R\geqslant1,\ h>0.
\]
For fixed $R,$ we apply Lemma~\ref{lem:two-power} to the sum of the first two
terms and obtain%
\begin{equation}
\mathcal{E}_{\nu}(\sigma)\leqslant B_{\eta}t^{\eta/(1+\eta)}R^{(2+\eta
)/(1+\eta)}+8KR^{2-m}[\log(e+R)]^{-\beta}, \label{eq:weighted-preopt}%
\end{equation}
where $B_{\eta}=c_{\eta,1}18^{1/(1+\eta)}4^{\eta/(1+\eta)}B^{1/(1+\eta)}.$

\textbf{4}. \textit{Choose the right target radius to balance exponents}.
Dividing $\left(  \ref{eq:weighted-preopt}\right)  $ by $K>0$ gives%
\[
\frac{\mathcal{E}_{\nu}(\sigma)}{K}\leqslant B_{\eta}z^{\eta/(1+\eta
)}R^{(2+\eta)/(1+\eta)}+8R^{2-m}[\log(e+R)]^{-\beta},
\]
where we have made the substitution%
\[
z=t/\kappa,\text{ }\kappa=K^{1+1/\eta}.
\]
We choose $R$ for this expression and then return to $t$. Write $L_{z}%
=\log\left(  e/z\right)  $ and
\[
a_{\ell}=\eta/H>0,\text{ }b_{\ell}=\beta\left(  1+\eta\right)  /H\geqslant0.
\]
To balance the exponents, we choose $z$ of the form
\begin{equation}
R=z^{-a_{\ell}}L_{z}^{-b_{\ell}}. \label{eq:log-scale}%
\end{equation}
Let
\[
L_{\ast}=\max\left\{  4,16\left(  b_{_{\ell}}/a_{_{\ell}}\right)
^{2}\right\}  .
\]
Define $t_{\mathrm{up}}^{\prime}=e^{1-L_{\ast}}$ and
\begin{equation}
t_{\mathrm{up}}=\min\left\{  \kappa t_{\mathrm{up}}^{\prime},\exp\left(
1-\max\left\{  2,2\left\vert \log\kappa\right\vert \right\}  \right)
\right\}  . \label{eq:explicit-threshold}%
\end{equation}
We show that these definitions guarantee $R\geqslant1$ and $\log(e+R)\approx
L_{z}$ whenever $0<t\leqslant t_{\mathrm{up}}$. To prove it, we take $t$ in
this range, then the definition of $t_{\mathrm{up}}$ yields that
$0<z=t/\kappa\leqslant t_{\mathrm{up}}^{\prime}<1$, so $L_{z}=1-\log
z\geqslant L_{\ast}$. This together with the definition of $L_{\ast}$ gives
\[
L_{z}\geqslant4,\text{ }\sqrt{L_{z}}\geqslant4\frac{b_{\ell}}{a_{\ell}}.
\]
Note that $\log L\leqslant\sqrt{L}$ for any $L\geqslant1$. Multiplying the
inquality by $b_{\ell}\geqslant0$ and using the preceding inequalities we
therefore get
\[
b_{\ell}\log L_{z}\leqslant b_{\ell}\sqrt{L_{z}}=\frac{b_{\ell}}{\sqrt{L_{z}}%
}L_{z}\leqslant\frac{a_{\ell}}{4}L_{z}.
\]
These inequalities also hold when $b_{\ell}=0$. Since $L_{z}/4\geqslant1,$
\[
a_{\ell}\left(  L_{z}-1\right)  \geqslant a_{\ell}\left(  L_{z}-L_{z}%
/4\right)  =\frac{3a_{\ell}}{4}L_{z}.
\]
Since $-\log z=L_{z}-1$, $\left(  \ref{eq:log-scale}\right)  $ now implies
\[
\log R=a_{\ell}\left(  L_{z}-1\right)  -b_{\ell}\log L_{z}\geqslant
\frac{3a_{\ell}}{4}L_{z}-\frac{a_{\ell}}{4}L_{z}=\frac{a_{\ell}}{2}L_{z}>0.
\]
Thus $R\geqslant1.$ On the other hand, $b_{\ell}\log L_{z}\geqslant0$, so
$\log R\leqslant a_{\ell}(L_{z}-1)\leqslant a_{\ell}L_{z}$. Hence%
\begin{align*}
\log(e+R)  &  \leqslant\log((1+e)R)=\log(1+e)+\log R\\
&  \leqslant\log(1+e)+a_{\ell}L_{z}\leqslant\lbrack\log(1+e)+a_{\ell}]L_{z}.
\end{align*}
Combing these inequalities we see that%
\begin{equation}
\frac{a_{\ell}}{2}L_{z}\leqslant\log R\leqslant\log\left(  e+R\right)
\leqslant\left[  \log(1+e)+a_{\ell}\right]  L_{z}.
\label{eq:explicit-log-comparison}%
\end{equation}
This confirms that $\log(e+R)\approx L_{z}$. We next substitute the radius
$\left(  \ref{eq:log-scale}\right)  $ into the upper bound of $\mathcal{E}%
_{\nu}\left(  \sigma\right)  /K$. For the first term, the substitution is
\begin{align*}
B_{\eta}z^{\eta/(1+\eta)}R^{(2+\eta)/(1+\eta)}  &  =B_{\eta}z^{\eta/(1+\eta
)}(z^{-a_{\ell}}L_{z}^{-b_{\ell}})^{(2+\eta)/(1+\eta)}\\
&  =B_{\eta}z^{\eta/(1+\eta)-a_{\ell}(2+\eta)/(1+\eta)}L_{z}^{-b_{\ell}%
(2+\eta)/(1+\eta)}.
\end{align*}
Using the definitions of $a_{\ell}$ and $b_{\ell}$, the exponents are
simplified as%
\begin{align*}
\frac{\eta}{1+\eta}-\frac{a_{\ell}(2+\eta)}{1+\eta}  &  =\frac{\eta\lbrack
H-(2+\eta)]}{H(1+\eta)}=\frac{\eta\lbrack m(1+\eta)-\eta-2-\eta]}{H(1+\eta)}\\
&  =\frac{\eta(m-2)(1+\eta)}{H(1+\eta)}=\frac{\eta(m-2)}{H}=\Theta,
\end{align*}%
\[
-\frac{b_{\ell}(2+\eta)}{1+\eta}=-\frac{\beta(1+\eta)}{H}\frac{2+\eta}{1+\eta
}=-\frac{\beta(2+\eta)}{H}=-\Lambda.
\]
Thus the first substitution is exactly $B_{\eta}z^{\Theta}L_{z}^{-\Lambda}$.
For the second term,
\[
R^{2-m}=(z^{-a_{\ell}}L_{z}^{-b_{\ell}})^{2-m}=z^{a_{\ell}(m-2)}L_{z}%
^{b_{\ell}(m-2)}.
\]
Since $\beta\geqslant0$ and both logarithms in $\left(
\ref{eq:explicit-log-comparison}\right)  $ are positive,
\[
\lbrack\log(e+R)]^{-\beta}\leqslant(a_{\ell}/2)^{-\beta}L_{z}^{-\beta}.
\]
Multipliplying the inequality by the previous equation we have
\[
8R^{2-m}[\log(e+R)]^{-\beta}\leqslant8(a_{\ell}/2)^{-\beta}z^{a_{\ell}%
(m-2)}L_{z}^{b_{\ell}(m-2)-\beta}.
\]
The exponents are simplified as,%
\[
a_{\ell}(m-2)=\frac{\eta(m-2)}{H}=\Theta,
\]%
\begin{align*}
b_{\ell}(m-2)-\beta &  =\frac{\beta(1+\eta)(m-2)}{H}-\frac{\beta H}{H}%
=-\frac{\beta\lbrack H-(m-2)(1+\eta)]}{H}\\
&  =-\frac{\beta\left[  m(1+\eta)-\eta-(m-2)(1+\eta)\right]  }{H}=-\frac
{\beta(2+\eta)}{H}=-\Lambda.
\end{align*}
Putting the above substitutions back into the upper bound of $\mathcal{E}%
_{\nu}(\sigma),$ we obtain
\[
\mathcal{E}_{\nu}(\sigma)\leqslant\left[  B_{\eta}+8(a_{\ell}/2)^{-\beta
}\right]  K\varphi(t/\kappa).
\]
To return to the original modulus $\varphi(t)$, we first compute
\[
1-(1+1/\eta)\Theta=1-\frac{(1+\eta)(m-2)}{H}=\frac{2+\eta}{H}.
\]
Hence recalling $\kappa=K^{1+1/\eta}$,
\[
K\varphi(t/\kappa)=K^{(2+\eta)/H}t^{\Theta}[\log(e\kappa/t)]^{-\Lambda}.
\]
Let $L_{t}=\log(e/t)$. The second condition in $\left(
\ref{eq:explicit-threshold}\right)  $ gives $L_{t}\geqslant2|\log\kappa|$,
hence
\[
\frac{1}{2}\leqslant\frac{\log(e\kappa/t)}{L_{t}}=1+\frac{\log\kappa}{L_{t}%
}\leqslant\frac{3}{2}.
\]
Because $\Lambda\geqslant0$,
\begin{equation}
(3/2)^{-\Lambda}K^{(2+\eta)/H}\varphi(t)\leqslant K\varphi(t/\kappa
)\leqslant2^{\Lambda}K^{(2+\eta)/H}\varphi(t). \label{eq:phi-comparison}%
\end{equation}
Taking the supremum over all admissible pairs and plans we have
\begin{equation}
\Omega_{\mathcal{C}_{m,\beta}}(t)\leqslant C_{\mathrm{up}}\varphi(t),\text{
}0<t\leqslant t_{\mathrm{up}}. \label{eq:explicit-upper-constant}%
\end{equation}
where%
\[
C_{\mathrm{up}}=2^{\Lambda}K^{(2+\eta)/H}\left[  B_{\eta}+8(a_{\ell
}/2)^{-\beta}\right]  .
\]
Note that so far no division by $m-2$ has occurred, so the computation covers
$m=2$.

\textbf{5}.\textit{ Construct matching lower example}. Set
\[
R_{K}=\max\{4,(16K)^{1/m}\}.
\]
For $R\geqslant R_{K}$, let
\[
S=q^{-1/n},\text{ }h=q^{1/\eta},\text{ }q=KR^{-m}L_{R}^{-\beta},\text{ }%
\]
and%
\begin{equation}
L_{R}=\log(e+R). \label{eq:Log-eR}%
\end{equation}
Since $L_{R}\geqslant1$ and $KR^{-m}\leqslant1/16$, these definitions yield
that $0<q\leqslant1/16$, $0<h<1$ and $S\geqslant1$. For $d>1$, define
\[
I_{-}=[-h,0],\text{ }I_{+}=[0,h]
\]
and
\[
F=[-2,-1]\times\lbrack0,1]^{d-1},\text{ }A_{-}=I_{-}\times\lbrack
S,2S]^{d-1},\text{ }A_{+}=I_{+}\times\lbrack S,2S]^{d-1}.
\]
For $d=1$, $F=[-2,-1],$ $A_{-}=I_{-},$ $A_{+}=I_{+}$. The sets are disjoint up
to null boundaries. Define $\rho_{R}=f_{R}dx$ with
\[
f_{R}=\left(  1-2q\right)  1_{F}+\frac{q}{hS^{d-1}}\left(  1_{A_{-}}+1_{A_{+}%
}\right)  .
\]
We now verify that $\rho_{R}\in\mathcal{C}_{m,\beta}$ for all $R\geqslant
R_{K}.$ Since $\left\vert F\right\vert =1$, $\left\vert A_{-}\right\vert
=\left\vert A_{+}\right\vert =hS^{d-1}$, hence $\int f_{R}dx=1.$ Since
$1/\eta=1/s+(d-1)/n$,
\[
hS^{d-1}=q^{1/\eta-(d-1)/n}=q^{1/s}.
\]
For $p<\infty$,
\[
\Vert f_{R}\Vert_{p}^{p}=(1-2q)^{p}+2q^{p}(hS^{d-1})^{1-p}=(1-2q)^{p}%
+2\leqslant3=F_{p}^{p}\leqslant(\alpha F_{p})^{p}.
\]
For $p=\infty$, $s=1$ and the density on $A_{-}\cup A_{+}$ becomes $q/\left(
hS^{d-1}\right)  =q/q^{1/s}=1$. So we still have%
\[
\Vert f_{R}\Vert_{\infty}=\max\left\{  \left(  1-2q\right)  ,1\right\}
\leqslant1\leqslant\alpha F_{\infty}.
\]
On $F$, $\left\vert x\right\vert ^{2}\leqslant d+3$. On $A_{-}\cup A_{+}$,
\[
\left\vert x\right\vert ^{2}\leqslant h^{2}+4\left(  d-1\right)
S^{2}\leqslant\left(  4d-3\right)  S^{2}.
\]
The latter also holds when $d=1,$ because $h\leqslant1\leqslant S$.
Consequently, using $qS^{n}=1$ we get%
\begin{align*}
\int\left\vert x\right\vert ^{n}d\rho_{R}  &  \leqslant\left(  1-2q\right)
\left(  d+3\right)  ^{n/2}+2q\left(  4d-3\right)  ^{n/2}S^{n}\\
&  \leqslant\left(  d+3\right)  ^{n/2}+2\left(  4d-3\right)  ^{n/2}%
=M_{d,n}\leqslant\xi M_{d,n}.
\end{align*}
Therefore $\rho_{R}$ satisfies all requirements of $\mathcal{C}_{m,\beta}$.
Let
\[
u_{R}(x)=R\max\{0,x_{1}\},
\]
and $T_{R}=\nabla u_{R}$ up to a null set. Thus $T_{R}=0$ on $F\cup A_{-}$ and
$T_{R}=Re_{1}$ on $A_{+}$. Define
\[
\nu_{R}=\left(  T_{R}\right)  _{\#}\rho_{R}=\left(  1-q\right)  \delta
_{0}+q\delta_{Re_{1}}.
\]
So $T_{R}$ is the unique Brenier map between $\rho_{R}$ and $\nu_{R}.$
Moreover,
\[
\int\left\vert y\right\vert ^{m}\left[  \log\left(  e+\left\vert y\right\vert
\right)  \right]  ^{\beta}d\nu_{R}=qR^{m}L_{R}^{\beta}=K.
\]
Therefore the construction verifies all three requirements in $\left(
\ref{eq:normalized-class}\right)  $.

Let $\mathcal{S}_{R}$ be the map that translates $A_{-}$ by $he_{1}$,
translates $A_{+}$ by $-he_{1}$, but keeps $F$ unchanged. We define
$\mathcal{S}_{R}$ arbitrarily on null boundaries. By the definition we have
$(\mathcal{S}_{R})_{\#}\rho_{R}=\rho_{R}$. Hence
\[
\sigma_{R}=\left(  \operatorname{Id},T_{R}\circ\mathcal{S}_{R}\right)
_{\#}\rho_{R}\in\Pi\left(  \rho_{R},\nu_{R}\right)
\]
is a deterministic feasible plan. Thus
\begin{equation}
E\left(  R\right)  \triangleq\mathcal{E}_{\nu_{R}}\left(  \sigma_{R}\right)
=2qR^{2}=2KR^{-\left(  m-2\right)  }L_{R}^{-\beta}. \label{eq:log-lower-E}%
\end{equation}
The correlation identity $\left(  \ref{eq:D-corr}\right)  $ gives,%
\begin{align}
D\left(  R\right)   &  \triangleq\mathcal{D}_{\nu_{R}}\left(  \sigma
_{R}\right)  =R\frac{q}{hS^{d-1}}S^{d-1}\left(  \int_{0}^{h}x_{1}\,dx_{1}%
-\int_{-h}^{0}x_{1}\,dx_{1}\right) \label{eq:log-lower-D}\\
&  =R\frac{q}{h}\left(  \frac{h^{2}}{2}+\frac{h^{2}}{2}\right)  =hqR=\kappa
R^{-H/\eta}L_{R}^{-\beta\left(  1+\eta\right)  /\eta}.\nonumber
\end{align}
In the last equality, we have used $m\left(  1+1/\eta\right)  -1=H/\eta$ and
$\kappa=K^{1+1/\eta}$. Define
\[
\Delta_{0}\left(  R\right)  =R^{-H/\eta}L_{R}^{-\beta\left(  1+\eta\right)
/\eta}\text{ and }D\left(  R\right)  =\kappa\Delta_{0}\left(  R\right)  .
\]

\textbf{6}. \textit{Derive exact moduli in the small objective gap regime}.
The function $D\left(  R\right)  $ is continuous and strictly decreasing on
$\left[  R_{K},\infty\right)  $, with limit zero as $R\rightarrow\infty$. Thus
for every $0<t\leqslant D\left(  R_{K}\right)  $ there is a unique $R\geqslant
R_{K}$ with $D\left(  R\right)  =t.$ Let%
\[
t_{0}\triangleq\min\left\{  t_{\mathrm{up}},D\left(  R_{K}\right)  \right\}
.
\]
Then for $0<t<t_{0},$ there is a unique $R$ with $D\left(  R\right)  =t.$ For
such $R$, $\log R\leqslant L_{R}\leqslant2\log R$. The latter inequality
follows from $e+R\leqslant R^{2}$ for $R\geqslant4$. Hence
\[
\frac{\log\left(  e/\Delta_{0}\left(  R\right)  \right)  }{L_{R}}%
=\frac{1+\left(  H/\eta\right)  \log R+\left[  \beta\left(  1+\eta\right)
/\eta\right]  \log L_{R}}{L_{R}}\geqslant\frac{H}{2\eta}.
\]
We now substitute $E\left(  R\right)  $ and $D\left(  R\right)  $ into the
definition $\left(  \ref{eq:phiDef}\right)  $ of $\varphi$ and compute%
\[
\frac{E\left(  R\right)  }{K\varphi\left(  D\left(  R\right)  /\kappa\right)
}=\frac{E\left(  R\right)  \left[  \log\left(  e/\Delta_{0}\left(  R\right)
\right)  \right]  ^{\Lambda}}{K\left[  \Delta_{0}\left(  R\right)  \right]
^{\Theta}}=\frac{2KR^{-\left(  m-2\right)  }L_{R}^{-\beta}\left[  \log\left(
e/\Delta_{0}\left(  R\right)  \right)  \right]  ^{\Lambda}}{K\left[
R^{-H/\eta}L_{R}^{-\beta\left(  1+\eta\right)  /\eta}\right]  ^{\Theta}}.
\]
To identify the powers we calculate%
\[
\left(  H/\eta\right)  \Theta=m-2,
\]%
\[
\left[  \beta\left(  1+\eta\right)  /\eta\right]  \Theta+\Lambda=\frac
{\beta\left[  \left(  1+\eta\right)  \left(  m-2\right)  +\left(
2+\eta\right)  \right]  }{H}=\beta.
\]
Therefore the previous equations continue as%
\[
\frac{E\left(  R\right)  }{K\varphi\left(  D\left(  R\right)  /\kappa\right)
}=2\left[  \frac{\log\left(  e/\Delta_{0}\left(  R\right)  \right)  }{L_{R}%
}\right]  ^{\Lambda}\geqslant2\left(  \frac{H}{2\eta}\right)  ^{\Lambda
}\triangleq c_{0}.
\]
Using the lower bound from $\left(  \ref{eq:phi-comparison}\right)  $ and
substituting $D\left(  R\right)  =t,$ we then obtain
\begin{equation}
E\left(  R\right)  \geqslant c_{0}K\varphi\left(  t/\kappa\right)  \geqslant
c_{0}\left(  3/2\right)  ^{-\Lambda}K^{\left(  2+\eta\right)  /H}%
\varphi\left(  t\right)  . \label{eq:lower-matching-E}%
\end{equation}
Thus in the desired inequality $\left(  \ref{eq:sharp-class-law}\right)  $, we
can choose the constants,
\[
c=2[H/(3\eta)]^{\Lambda}K^{(2+\eta)/H},\text{ }C=C_{\mathrm{up}}\text{ and
}t_{0}=\min\{t_{\mathrm{up}},D(R_{K})\}.
\]
Note that by the definition $\left(  \ref{eq:explicit-threshold}\right)  ,$ we
have $t_{\mathrm{up}}<1,$ so $t_{0}<1.$ The constructed coupling $\sigma_{R}$
together with its marginals satisfy all requirements of the class
$\mathcal{C}_{m,\beta}$ including the exact objective gap $D\left(  R\right)
=t$, so $E(R)\leqslant\Omega_{\mathcal{C}_{m,\beta}}(t)$. Therefore the
required lower bound for $0<t<t_{0}$ is proved.

\textbf{7}. \textit{Derive mathcing barycentric error bounds}. Let $\sigma
\in\Pi\left(  \rho,\nu\right)  $ be an arbitrary feasible plan for an
arbitrary pair $(\rho,\nu)$ in $\mathcal{C}_{m,\beta}$. Then using $\left(
\ref{eq:variance-decomposition}\right)  $ and~$\left(
\ref{eq:explicit-upper-constant}\right)  $ we get
\[
\Vert\bar{T}_{\sigma}-T_{\nu}\Vert_{L^{2}(\rho)}^{2}\leqslant\mathcal{E}_{\nu
}(\sigma)\leqslant C_{\mathrm{up}}\varphi(t).
\]
By contrast the particular deterministic plan $\sigma_{R}$ constructed in Step
5 with marginals $(\rho_{R},\nu_{R})$ is concentrated on a map, so
\[
\Vert\bar{T}_{\sigma_{R}}-T_{R}\Vert_{L^{2}(\rho_{R})}^{2}=\mathcal{E}%
_{\nu_{R}}(\sigma_{R})=E(R).
\]
Recall that in Step 6 we have chosen $R$ with $D(R)=t$ for $t$ in $\left(
0,t_{0}\right)  $. So these equations together with $\left(
\ref{eq:lower-matching-E}\right)  $ yield the desired matching modulus for
barycentric map error.

\textbf{8}. \textit{Matching bounds for the extreme cases}. In the extreme
case $m=2$, we have $H=2+\eta$, $\Theta=0$, and $\Lambda=\beta$. For $\beta
>0$, the modulus is therefore comparable to $[\log(e/t)]^{-\beta}$, which
tends to zero. However if $m=2$ and $\beta=0$ happen at the same time, the
example in Step 5 and $\left(  \ref{eq:log-lower-E}\right)  \left(
\ref{eq:log-lower-D}\right)  $ indicate that
\[
q=KR^{-2},\text{ }E(R)=2qR^{2}=2K,\text{ }D(R)=K^{1+1/\eta}R^{-(2+\eta)/\eta
}\rightarrow0.
\]
On the other hand, for each pair $\left(  \rho,\nu\right)  \in\mathcal{C}%
_{2,0}$ and each feasible plan $\sigma\in\Pi\left(  \rho,\nu\right)  $ we have
that
\[
\mathcal{E}_{\nu}(\sigma)\leqslant2\int|y|^{2}d\nu+2\int|T_{\nu}(x)|^{2}%
d\rho=4\int|y|^{2}d\nu\leqslant4K.
\]
The equality uses $(T_{\nu})_{\#}\rho=\nu$. Note that the above does not need
any objective gap restriction to hold. Consequently $2K\leqslant
\Omega_{\mathcal{C}_{2,0}}(t)\leqslant4K$ for small $t>0$. So there is no
vanishing modulus for the extreme case $m=2,$ $\beta=0$. Finally the exteme
case $m>2,$ $\beta=0$ corresponds to the pure power law of Theorem
\ref{thm:unbounded-power}.
\end{proof}

\section{From arbitrary plans to two Brenier maps}

\label{sec:two-map}

Throughout this section $\rho=fdx\in\mathcal{P}_{2}(\mathbb{R}^{d})$ and
$\mu,$ $\nu\in\mathcal{P}_{2}\left(  \mathbb{R}^{d}\right)  $. Denote the
Brenier maps from $\rho$ by $T_{\mu},$ $T_{\nu}$. The section intends to
convert the arbitrary plan error bound into a comparison of $T_{\mu},T_{\nu}$.
The conversion requires an auxiliary plan constructed through a
\textit{repairing} process.

\begin{definition}
[Plan repaired through a target coupling]\label{def:repaired-plan}Let
$\gamma\in\Pi\left(  \mu,\nu\right)  $. By disintegration, there is a
measurable probability kernel $y\mapsto k_{y}$ such that
\[
\gamma\left(  dy,dz\right)  =\mu\left(  dy\right)  k_{y}\left(  dz\right)  .
\]
Define a probability measure on triples by
\begin{equation}
\Lambda_{T_{\mu},\gamma}\left(  dx,dy,dz\right)  =\rho\left(  dx\right)
\delta_{T_{\mu}\left(  x\right)  }\left(  dy\right)  k_{y}\left(  dz\right)  ,
\label{eq:repaired-triple}%
\end{equation}
and define the $\gamma$-repaired plan by
\begin{equation}
\sigma_{T_{\mu},\gamma}=\left(  x,z\right)  _{\#}\Lambda_{T_{\mu},\gamma}.
\label{eq:repaired-plan}%
\end{equation}

\end{definition}

The $x$-marginal of $\Lambda_{T_{\mu},\gamma}$ is $\rho$, its $\left(
y,z\right)  $-marginal is $\gamma$, and therefore the $z$-marginal is $\nu$.
Hence $\sigma_{T_{\mu},\gamma}\in\Pi\left(  \rho,\nu\right)  $. Up to a
$\Lambda_{T_{\mu},\gamma}$-null set, $y=T_{\mu}\left(  x\right)  .$ So under
$\Lambda_{T_{\mu},\gamma}$,
\[
\left\vert T_{\mu}\left(  x\right)  -T_{\nu}\left(  x\right)  \right\vert
^{2}\leqslant2\left\vert y-z\right\vert ^{2}+2\left\vert z-T_{\nu}\left(
x\right)  \right\vert ^{2}.
\]
Integrating w.r.t. $\Lambda_{T_{\mu},\gamma}$, and using its $\left(
y,z\right)  $-marginal $\gamma$ and its $\left(  x,z\right)  $-marginal
$\sigma_{T_{\mu},\gamma}$, yields
\begin{equation}
\left\Vert T_{\mu}-T_{\nu}\right\Vert _{L^{2}(\rho)}^{2}\leqslant
2\int\left\vert y-z\right\vert ^{2}d\gamma+2\mathcal{E}_{\nu}\left(
\sigma_{T_{\mu},\gamma}\right)  . \label{eq:map-graph}%
\end{equation}
We have proved in Section \ref{sec:profile-optimization} that, under mild
conditions,\ $\mathcal{E}_{\nu}\left(  \sigma_{T_{\mu},\gamma}\right)  $ is
controlled by a modulus of $\mathcal{D}_{\nu}\left(  \sigma_{T_{\mu},\gamma
}\right)  .$ So in order to get the stability of Brenier map, it remains to
control $\mathcal{D}_{\nu}\left(  \sigma_{T_{\mu},\gamma}\right)  .$

\subsection{Maximal correlation and translation invariance}

For the source measure, recall that for $a\in\mathbb{R}^{d}$,%
\[
M_{\rho,r}\left(  a\right)  =\left(  \int_{\mathbb{R}^{d}}\left\vert
x-a\right\vert ^{r}f(x)dx\right)  ^{1/r},\text{ }M_{\rho,r}=\inf
_{a\in\mathbb{R}^{d}}M_{\rho,r}(a).
\]
Define for $a\in\mathbb{R}^{d},$
\[
F_{a}\left(  \eta\right)  =\sup_{\pi\in\Pi\left(  \rho,\eta\right)  }%
\int\left(  x-a\right)  \cdot yd\pi\left(  x,y\right)  .
\]
For $\pi\in\Pi\left(  \rho,\eta\right)  $,
\[
\frac{1}{2}\int\left\vert x-y\right\vert ^{2}d\pi=\frac{1}{2}\int\left\vert
x\right\vert ^{2}d\rho+\frac{1}{2}\int\left\vert y\right\vert ^{2}d\eta-\int
x\cdot yd\pi.
\]
The first two terms are fixed by the marginals, so minimizing the quadratic
cost is equivalent to maximizing the correlation. Subtracting $a\int yd\eta$,
which is also fixed by the marginals, does not change the maximizers.
Therefore, if $\pi_{\nu}\in\Pi\left(  \rho,\nu\right)  $ is optimal for the
quadratic cost, then
\[
F_{a}\left(  \nu\right)  =\int\left(  x-a\right)  \cdot zd\pi_{\nu}\left(
x,z\right)  .
\]
Substituting $\sigma_{T_{\mu},\gamma}$ constructed in $\left(
\ref{eq:repaired-plan}\right)  $ and $\pi_{\nu}$ into the definition $\left(
\ref{eq:Dintro}\right)  $ we have%
\begin{align*}
\mathcal{D}_{\nu}\left(  \sigma_{T_{\mu},\gamma}\right)   &  =\frac{1}{2}%
\int\left\vert x-z\right\vert ^{2}d\sigma_{T_{\mu},\gamma}\left(  x,z\right)
-\frac{1}{2}W_{2}^{2}\left(  \rho,\nu\right) \\
&  =\frac{1}{2}\int\left\vert x-z\right\vert ^{2}d\sigma_{T_{\mu},\gamma
}\left(  x,z\right)  -\frac{1}{2}\int\left\vert x-z\right\vert ^{2}d\pi_{\nu
}\left(  x,z\right) \\
&  =\int x\cdot zd\pi_{\nu}\left(  x,z\right)  -\int x\cdot zd\sigma_{T_{\mu
},\gamma}\left(  x,z\right)  .
\end{align*}
Hence
\begin{equation}
\mathcal{D}_{\nu}\left(  \sigma_{T_{\mu},\gamma}\right)  =F_{a}\left(
\nu\right)  -\int\left(  x-a\right)  \cdot z\,d\sigma_{T_{\mu},\gamma}\left(
x,z\right)  . \label{eq:D-F}%
\end{equation}
This expression is independent of $a$. Indeed, let $a^{\prime}\in
\mathbb{R}^{d}$. Since each coupling in $\Pi\left(  \rho,\nu\right)  $ has
fixed $y$-marginal $\nu$,
\[
F_{a^{\prime}}\left(  \nu\right)  =F_{a}\left(  \nu\right)  -\left(
a^{\prime}-a\right)  \int zd\nu\left(  z\right)  .
\]
Since the $z$-marginal of $\sigma_{T_{\mu},\gamma}$ is $\nu,$ so
\[
\int\left(  x-a^{\prime}\right)  \cdot zd\sigma_{T_{\mu},\gamma}=\int\left(
x-a\right)  \cdot zd\sigma_{T_{\mu},\gamma}-\left(  a^{\prime}-a\right)  \int
zd\nu\left(  z\right)  .
\]
Subtracting it from the previous equation, we see that $\left(  \ref{eq:D-F}%
\right)  $ is independent of $a$. Thus $a$ affects only the quality of the
moment or tail estimate, not the objective gap itself.

\begin{lemma}
\label{lem:ctrl-by-corr}For $\gamma\in\Pi\left(  \mu,\nu\right)  $, define
\begin{equation}
\mathcal{C}_{\rho,a}\left(  \gamma\right)  =\sup_{\Lambda_{\gamma}}%
\int\left\vert x-a\right\vert \left\vert z-y\right\vert d\Lambda_{\gamma
}\left(  x,y,z\right)  , \label{eq:D-product}%
\end{equation}
where the supremum is over all probabilistic triple law $\Lambda_{\gamma}$
whose $x$-marginal is $\rho$ and whose $\left(  y,z\right)  $-marginal is
$\gamma$. Under the construction in Definition~\ref{def:repaired-plan},
\begin{equation}
\mathcal{D}_{\nu}\left(  \sigma_{T_{\mu},\gamma}\right)  \leqslant
2\mathcal{C}_{\rho,a}\left(  \gamma\right)  . \label{eq:D-C}%
\end{equation}

\end{lemma}

\begin{proof}
Let $\Lambda_{T_{\mu},\gamma}$ be the triple law in $\left(
\ref{eq:repaired-triple}\right)  $. Its $\left(  x,y\right)  $-marginal is the
optimal map from $\rho$ to $\mu$, hence
\[
\int(x-a)\cdot yd\Lambda_{T_{\mu},\gamma}=F_{a}(\mu).
\]
Its $\left(  x,z\right)  $-marginal is $\sigma_{T_{\mu},\gamma}$.
Consequently,
\[
F_{a}\left(  \mu\right)  -\int\left(  x-a\right)  \cdot zd\sigma_{T_{\mu
},\gamma}=\int\left(  x-a\right)  \cdot\left(  y-z\right)  d\Lambda_{T_{\mu
},\gamma}\leqslant\int\left\vert x-a\right\vert \left\vert z-y\right\vert
d\Lambda_{T_{\mu},\gamma}\leqslant\mathcal{C}_{\rho,a}\left(  \gamma\right)
.
\]
For the converse comparison, let $\pi_{\nu}\in\Pi\left(  \rho,\nu\right)  $
attain $F_{a}(\nu)$. We disintegrate $\gamma$ with respect to its
$z$-marginal,
\[
\gamma\left(  dy,dz\right)  =\nu\left(  dz\right)  \ell_{z}(dy),
\]
and then set
\[
\bar{\Lambda}_{\gamma}\left(  dx,dy,dz\right)  =\pi_{\nu}\left(  dx,dz\right)
\ell_{z}\left(  dy\right)  .
\]
This law has $x$-marginal $\rho$ and $\left(  y,z\right)  $-marginal $\gamma$,
so it is admissible in $\left(  \ref{eq:D-product}\right)  $. Since the
$\left(  x,y\right)  $-marginal of $\bar{\Lambda}_{\gamma}$ is a coupling of
$\rho$ and $\mu$,
\[
\int\left(  x-a\right)  \cdot yd\bar{\Lambda}_{\gamma}\leqslant F_{a}\left(
\mu\right)  .
\]
Using $\left(  \ref{eq:D-F}\right)  $ and that $\bar{\Lambda}_{\gamma}$ has
$\left(  x,z\right)  $-marginal $\pi_{\nu}$, we obtain
\[
F_{a}\left(  \nu\right)  -F_{a}\left(  \mu\right)  \leqslant\int\left(
x-a\right)  \cdot\left(  z-y\right)  d\bar{\Lambda}_{\gamma}\leqslant
\int\left\vert x-a\right\vert \left\vert z-y\right\vert d\bar{\Lambda}%
_{\gamma}\leqslant\mathcal{C}_{\rho,a}\left(  \gamma\right)  .
\]
Adding the previous inequalities we get $\left(  \ref{eq:D-C}\right)  $. Note
that the supremum in $\left(  \ref{eq:D-product}\right)  $ is essential
because $\Lambda_{\gamma}$ and $\bar{\Lambda}_{\gamma}$ need not coincide.
\end{proof}

\subsection{Objective gap of repaired plan}

This section provides control of the objective gap $\mathcal{D}_{\nu}\left(
\sigma_{T_{\mu},\gamma}\right)  $ by the distance between targets under
various source and target assumptions.

\begin{proposition}
[$W_{r}$-repair]\label{prop:Wr-ctrl}Let $1<r\leqslant2$, $r^{\prime}=r/\left(
r-1\right)  $ and $M_{\rho,r^{\prime}}<\infty$. If $\gamma$ is $W_{r}%
$-optimal, then
\begin{equation}
\mathcal{D}_{\nu}\left(  \sigma_{T_{\mu},\gamma}\right)  \leqslant
2M_{\rho,r^{\prime}}W_{r}\left(  \mu,\nu\right)  . \label{eq:Wr-ctrl}%
\end{equation}

\end{proposition}

\begin{proof}
Applying H\"{o}lder's inequality to each law admitted in $\left(
\ref{eq:D-product}\right)  $, we have
\[
\int\left\vert x-a\right\vert \left\vert z-y\right\vert d\Lambda_{\gamma
}\leqslant M_{\rho,r^{\prime}}\left(  a\right)  \left(  \int\left\vert
z-y\right\vert ^{r}d\gamma\right)  ^{1/r}.
\]
The last factor equals $W_{r}\left(  \mu,\nu\right)  $. Thus,%
\[
\mathcal{C}_{\rho,a}\left(  \gamma\right)  \leqslant M_{\rho,r^{\prime}%
}\left(  a\right)  W_{r}\left(  \mu,\nu\right)
\]
for all $a\in\mathbb{R}^{d}$. Using $\left(  \ref{eq:D-C}\right)  $ and taking
the infimum over $a$ we obtain $\left(  \ref{eq:Wr-ctrl}\right)  $.
\end{proof}

\begin{proposition}
[$W_{1}$-repair-moment-moment]\label{prop:W1-moments}Let $n\geqslant2$, $m>2$
and define
\begin{equation}
M_{\mu,\nu,m}=\inf_{b\in\mathbb{R}^{d}}\max\left\{  M_{\mu,m}\left(  b\right)
,M_{\nu,m}\left(  b\right)  \right\}  . \label{eq:M-mu-nu-pair}%
\end{equation}
Assume $M_{\rho,n}<\infty$ and $M_{\mu,\nu,m}<\infty$. Let
\[
q=\frac{n}{n-1},\text{ }\alpha=\frac{m-q}{q\left(  m-1\right)  }%
=\frac{m\left(  n-1\right)  -n}{n\left(  m-1\right)  }.
\]
If $\gamma$ is optimal for $W_{1}\left(  \mu,\nu\right)  $, then with
$\delta=W_{1}\left(  \mu,\nu\right)  ,$%
\begin{equation}
\int\left\vert y-z\right\vert ^{2}d\gamma\leqslant\left(  2M_{\mu,\nu
,m}\right)  ^{m/\left(  m-1\right)  }\delta^{\left(  m-2\right)  /\left(
m-1\right)  }, \label{eq:W1-direct}%
\end{equation}%
\begin{equation}
\mathcal{D}_{\nu}\left(  \sigma_{T_{\mu},\gamma}\right)  \leqslant\left(
2M_{\rho,n}\right)  \left(  2M_{\mu,\nu,m}\right)  ^{1-\alpha}\delta^{\alpha}.
\label{eq:W1-D}%
\end{equation}

\end{proposition}

\begin{proof}
Since $n\geqslant2$, $1<q\leqslant2<m$,%
\[
q\left(  m-1\right)  -\left(  m-q\right)  =m\left(  q-1\right)  >0.
\]
So $\alpha\in\left(  0,1\right)  $. Since%
\[
m\alpha=\frac{m\left(  m-q\right)  }{q\left(  m-1\right)  },\text{ }%
1-\alpha=\frac{m\left(  q-1\right)  }{q\left(  m-1\right)  },
\]
hence
\[
\alpha+\frac{1-\alpha}{m}=\frac{1}{q}.
\]
Let $b_{\ast}$ minimize the definition of $M_{\mu,\nu,m}$ (see the discussion
at the beginning of Section \ref{sec:target-profile}). By the definition of
$M_{\mu,\nu,m}$ we have
\begin{equation}
\left\Vert z-y\right\Vert _{L^{m}\left(  \gamma\right)  }\leqslant\left\Vert
z-b_{\ast}\right\Vert _{L^{m}\left(  \nu\right)  }+\left\Vert y-b_{\ast
}\right\Vert _{L^{m}\left(  \mu\right)  }\leqslant2M_{\mu,\nu,m},
\label{eq:repair-nm-Lm}%
\end{equation}
Let $r_{1}=1/\left(  \alpha q\right)  $ and $r_{2}=m/\left(  \left(
1-\alpha\right)  q\right)  $. Both exponents exceed one and $1/r_{1}%
+1/r_{2}=\alpha q+\left(  1-\alpha\right)  q/m=1$. Hence by H\"{o}lder
interpolation we obtain
\[
\int\left\vert z-y\right\vert ^{q}d\gamma=\int\left\vert z-y\right\vert
^{\alpha q}\left\vert z-y\right\vert ^{\left(  1-\alpha\right)  q}%
d\gamma\leqslant\left(  \int\left\vert z-y\right\vert d\gamma\right)  ^{\alpha
q}\left(  \int^{m}\left\vert z-y\right\vert ^{m}d\gamma\right)  ^{\left(
1-\alpha\right)  q/m}.
\]
Taking the $q$-th root and using $\int\left\vert z-y\right\vert d\gamma
=\delta$ we get
\[
\left\Vert z-y\right\Vert _{L^{q}\left(  \gamma\right)  }\leqslant
\delta^{\alpha}\left(  2M_{\mu,\nu,m}\right)  ^{1-\alpha}.
\]
For each admissible $\Lambda_{\gamma}$ in $\left(  \ref{eq:D-product}\right)
$, H\"{o}lder's inequality with exponents $n,$ $q$ therefore gives
\[
\int\left\vert x-a\right\vert \left\vert z-y\right\vert d\Lambda_{\gamma
}\leqslant M_{\rho,n}\left(  a\right)  \delta^{\alpha}\left(  2M_{\mu,\nu
,m}\right)  ^{1-\alpha}.
\]
Taking the supremum over $\Lambda_{\gamma}$, using $\left(  \ref{eq:D-C}%
\right)  $, and then taking the infimum over $a$, we obtain $\left(
\ref{eq:W1-D}\right)  $. For the direct coupling estimate $\left(
\ref{eq:W1-direct}\right)  $, H\"{o}lder interpolation with the relation
\[
\frac{1}{2}=\frac{m-2}{2\left(  m-1\right)  }+\frac{m}{2\left(  m-1\right)
}\frac{1}{m}%
\]
gives
\[
\Vert z-y\Vert_{L^{2}\left(  \gamma\right)  }^{2}=\int\left\vert
y-z\right\vert ^{2}d\gamma\leqslant\Vert z-y\Vert_{L^{1}\left(  \gamma\right)
}^{(m-2)/(m-1)}\Vert z-y\Vert_{L^{m}\left(  \gamma\right)  }^{m/(m-1)}.
\]
Substituting $\left(  \ref{eq:repair-nm-Lm}\right)  $ into the RHS yields
$\left(  \ref{eq:W1-direct}\right)  $.
\end{proof}

\begin{proposition}
[$W_{1}$-repair-moment-compact]\label{prop:W1-nth-compact}Let $n\geqslant2$.
Assume that $M_{\rho,n}<\infty$, the targets $\mu,\nu$ are supported in a
common ball $B\left(  b,R_{Y}\right)  ,$ $R_{Y}>0.$ Let
\[
q=\frac{n}{n-1}.
\]
If $\gamma$ is optimal for $W_{1}\left(  \mu,\nu\right)  $, then with
$\delta=W_{1}\left(  \mu,\nu\right)  ,$%
\[
\int\left\vert y-z\right\vert ^{2}d\gamma\leqslant2R_{Y}\delta,
\]
and%
\begin{equation}
\mathcal{D}_{\nu}\left(  \sigma_{T_{\mu},\gamma}\right)  \leqslant\left(
2M_{\rho,n}\right)  \left(  2R_{Y}\right)  ^{1/n}\delta^{1-1/n}.
\label{eq:D-nth-compact}%
\end{equation}

\end{proposition}

\begin{proof}
Since $\left\vert y-z\right\vert ^{2}\leqslant2R_{Y}\left\vert y-z\right\vert
$, integrating it against $\gamma$ yields,%
\[
\int\left\vert y-z\right\vert ^{2}d\gamma\leqslant2R_{Y}W_{1}\left(  \mu
,\nu\right)  =2R_{Y}\delta.
\]
To prove the bound for $\mathcal{D}_{\nu}\left(  \sigma_{T_{\mu},\gamma
}\right)  $, it is sufficient to bound $\mathcal{C}_{\rho,a}\left(
\gamma\right)  $. Fix an admissible $\Lambda_{\gamma}$ in $\left(
\ref{eq:D-product}\right)  $. By H\"{o}lder's inequality with exponents $n,$
$q,$ we therefore get%
\[
\int\left\vert x-a\right\vert \left\vert z-y\right\vert d\Lambda_{\gamma
}\leqslant M_{\rho,n}\left(  a\right)  \left\Vert z-y\right\Vert
_{L^{q}\left(  \gamma\right)  }.
\]
Since $q>1,$ we can write%
\[
\int\left\vert z-y\right\vert ^{q}d\gamma=\int\left\vert z-y\right\vert
\left\vert z-y\right\vert ^{q-1}d\gamma\leqslant\left(  2R_{Y}\right)
^{q-1}\int\left\vert z-y\right\vert d\gamma=\left(  2R_{Y}\right)
^{q-1}\delta.
\]
Taking $q$-th root yields%
\[
\left\Vert z-y\right\Vert _{L^{q}\left(  \gamma\right)  }\leqslant\left(
2R_{Y}\right)  ^{1-1/q}\delta^{1/q}=\left(  2R_{Y}\right)  ^{1/n}%
\delta^{1-1/n}.
\]
Therefore
\[
\int\left\vert x-a\right\vert \left\vert z-y\right\vert d\Lambda_{\gamma
}\leqslant\left(  2R_{Y}\right)  ^{1/n}M_{\rho,n}\left(  a\right)
\delta^{1-1/n}.
\]
Taking the supremum over $\Lambda_{\gamma}$, using $\left(  \ref{eq:D-C}%
\right)  $ and then taking the infimum over $a$, we obtain the bound for
$\mathcal{D}_{\nu}\left(  \sigma_{T_{\mu},\gamma}\right)  $.
\end{proof}

As with Proposition \ref{prop:annuli}, for fast decaying source its tail
information may be exploited to obtain tighter estimate of the repaired
objective gap $\mathcal{D}_{\nu}\left(  \sigma_{T_{\mu},\gamma}\right)  $ than
simply using its moment. Recall that,
\[
\tau_{1,\rho,a}\left(  S\right)  =\int_{\left\{  \left\vert x-a\right\vert
>S\right\}  }\left\vert x-a\right\vert \rho\left(  x\right)  .
\]

\begin{proposition}
\label{prop:D-1-tail}Let $\mu,$ $\nu$ be probability measures supported in
$B\left(  b,R_{Y}\right)  ,$ $R_{Y}>0$. If $\gamma$ is optimal for
$W_{1}\left(  \mu,\nu\right)  $, then
\begin{equation}
\mathcal{D}_{\nu}\left(  \sigma_{T_{\mu},\gamma}\right)  \leqslant2\inf
_{a\in\mathbb{R}^{d},\text{ }S>0}\left\{  SW_{1}\left(  \mu,\nu\right)
+2R_{Y}\tau_{1,\rho,a}\left(  S\right)  \right\}  . \label{eq:D-1-tail}%
\end{equation}

\end{proposition}

\begin{proof}
Write $\delta=W_{1}\left(  \mu,\nu\right)  .$ Fix $a\in\mathbb{R}^{d}$ and
$S>0$. Fix a triple law $\Lambda_{\gamma}$ admissible in $\left(
\ref{eq:D-product}\right)  $. On the set $\left\vert x-a\right\vert \leqslant
S$, we have $\left\vert x-a\right\vert \left\vert z-y\right\vert \leqslant
S\left\vert z-y\right\vert .$ So%
\[
\int_{\left\{  \left\vert x-a\right\vert \leqslant S\right\}  }\left\vert
x-a\right\vert \left\vert z-y\right\vert d\Lambda_{\gamma}\leqslant S\delta.
\]
On the complement $\left\vert x-a\right\vert >S,$ we have $\left\vert
z-y\right\vert \leqslant2R_{Y}$. Hence%
\[
\int_{\left\{  \left\vert x-a\right\vert >S\right\}  }\left\vert
x-a\right\vert \left\vert z-y\right\vert d\Lambda_{\gamma}\leqslant2R_{Y}%
\int_{\left\{  \left\vert x-a\right\vert >S\right\}  }\left\vert
x-a\right\vert d\Lambda_{\gamma}=2R_{Y}\tau_{1,\rho,a}\left(  S\right)  .
\]
The bounds are independent of the admissible triple law $\Lambda_{\gamma}$. We
now add these inequalities, take the supremum over $\Lambda_{\gamma}$ and use
$\left(  \ref{eq:D-C}\right)  $, then $\left(  \ref{eq:D-1-tail}\right)  $
follows since the resulting inequality is valid for every $a$ and $S$.
\end{proof}

\begin{remark}
The $n$th-moment condition is stronger than the tail bound $\left(
\ref{eq:D-1-tail}\right)  .$ Suppose $M_{\rho,n}\left(  a_{\ast}\right)
=M_{\rho,n}<\infty$ for $n\geqslant2,$ then%
\[
\tau_{1,\rho,a_{\ast}}\left(  S\right)  =\int_{\left\{  \left\vert x-a_{\ast
}\right\vert >S\right\}  }\left\vert x-a_{\ast}\right\vert \rho\left(
x\right)  \leqslant S^{1-n}\int\left\vert x-a_{\ast}\right\vert ^{n}%
\rho\left(  x\right)  =M_{\rho,n}S^{1-n}.
\]
Subsitituting this into $\left(  \ref{eq:D-1-tail}\right)  $ yields $\left(
\ref{eq:D-nth-compact}\right)  .$
\end{remark}

Analogous $W_{1}$-repairs are obtainable under other source target
assumptions. But to avoid repetition, we state them without proof. For the
$W_{1}$-repair with source supported in $B\left(  0,R_{X}\right)  $ and target
with $m$th-moment, $m>2,$ the direct coupling term is the same as $\left(
\ref{eq:W1-direct}\right)  ,$ while the correlation term is bounded by
\begin{equation}
\int\left\vert x-a\right\vert \left\vert z-y\right\vert d\Lambda_{\gamma
}\leqslant R_{X}\delta. \label{eq:repair-compact-m}%
\end{equation}
For the $W_{1}$-repair with source supported in $B\left(  0,R_{X}\right)  $
and target supported in $B\left(  0,R_{Y}\right)  ,$ the direct coupling term
and the correlation term are bounded respectively by%

\[
\int\left\vert y-z\right\vert ^{2}d\gamma\leqslant2R_{Y}\delta\text{ and}%
\int\left\vert x-a\right\vert \left\vert z-y\right\vert d\Lambda_{\gamma
}\leqslant R_{X}\delta.
\]

\section{Quantitative stability of Brenier maps with common source}

We are now ready to prove the quantitative stabilities of Brenier maps.

\begin{theorem}
[Stability-compact-compact]\label{thm:compact-maps}If $\rho$ is suppored in
$B\left(  a,R_{X}\right)  $, $\rho=fdx$ with $f\in L^{p}$, $p>1$, $s=1-1/p$
and $\mu,$ $\nu$ are supported in $B\left(  b,R_{Y}\right)  $, then
\[
\left\Vert T_{\mu}-T_{\nu}\right\Vert _{L^{2}(\rho)}\leqslant CW_{1}\left(
\mu,\nu\right)  ^{s/\left[  2\left(  1+s\right)  \right]  },
\]
where the constant $C>0$ depends only on $d,s,R_{X},R_{Y}$, not on $a,b$ or on
the targets. For $p=\infty$, this recovers the $1/4$ exponent of \cite[Theorem
3.1]{merigot2026sharp}, which is sharp for the general class when
$d\geqslant2$ by \cite{letrouit2026unstable}. If instead $\rho$ satisfies the
small-set condition $\left(  \ref{eq:small-set-density}\right)  $ with
$0<s\leqslant1$, then the same stability holds although the constant now also
depends on $C_{V}$.
\end{theorem}

\begin{proof}
Let $\gamma$ be optimal for $W_{1}\left(  \mu,\nu\right)  $ and let
$\sigma_{T_{\mu},\gamma}$ be the repaired plan from
Definition~\ref{def:repaired-plan}. Since $\left\vert x-a\right\vert \leqslant
R_{X}$, $\left(  \ref{eq:D-C}\right)  $ gives%
\begin{equation}
\mathcal{D}_{\nu}(\sigma_{T_{\mu},\gamma})\leqslant2R_{X}W_{1}(\mu,\nu).
\label{thm:compact-mapsD}%
\end{equation}
Also
\[
\int\left\vert y-z\right\vert ^{2}d\gamma\leqslant2R_{Y}W_{1}\left(  \mu
,\nu\right)  .
\]
The error bound $\left(  \ref{eq:compact-target-E}\right)  $ therefore gives
\begin{equation}
\mathcal{E}_{\nu}\left(  \sigma_{T_{\mu},\gamma}\right)  \leqslant
\operatorname*{Const}\cdot\mathcal{D}_{\nu}(\sigma_{T_{\mu},\gamma
})^{s/\left(  1+s\right)  }\leqslant C_{1}W_{1}\left(  \mu,\nu\right)
^{s/\left(  1+s\right)  }. \label{thm:compact-mapsE}%
\end{equation}
where $C_{1}=2^{s/\left(  1+s\right)  }B_{d,\rho,p}R_{X}^{ds/\left(
1+s\right)  }R_{Y}^{(2+s)/(1+s)},$ $B_{d,\rho,p}$ is defined in $\left(
\ref{eq:source-const}\right)  .$ Substituting the above estimates into
$\left(  \ref{eq:map-graph}\right)  $ yields, with $\delta=W_{1}\left(
\mu,\nu\right)  $,
\[
\left\Vert T_{\mu}-T_{\nu}\right\Vert _{L^{2}\left(  \rho\right)  }%
^{2}\leqslant4R_{Y}\delta+2C_{1}\delta^{s/\left(  1+s\right)  }.
\]
For $0<\delta\leqslant1$, $s/\left(  1+s\right)  <1$, so $\delta
\leqslant\delta^{s/\left(  1+s\right)  }$. Hence%
\[
\left\Vert T_{\mu}-T_{\nu}\right\Vert _{L^{2}\left(  \rho\right)  }%
\leqslant\left(  4R_{Y}+2C_{1}\right)  ^{1/2}\delta^{s/\left[  2\left(
1+s\right)  \right]  }.
\]
The case $\delta=0$ follows from uniqueness, and for the range $\delta>1$,
use
\[
\left\Vert T_{\mu}-T_{\nu}\right\Vert _{L^{2}\left(  \rho\right)  }%
\leqslant2R_{Y}\leqslant2R_{Y}\delta^{s/\left[  2\left(  1+s\right)  \right]
}.
\]
Thus the constant valid for all $\delta\geqslant0$ is $\max\left\{  \left(
4R_{Y}+2C_{1}\right)  ^{1/2},2R_{Y}\right\}  .$
\end{proof}

\begin{theorem}
[Stability-moment-compact]Assume Assumption~\ref{ass:brenier}, let $\rho=fdx$,
$f\in L^{p}\left(  \mathbb{R}^{d}\right)  $, $1<p\leqslant\infty$, $s=1-1/p$,
$M_{\rho,n}(a)<\infty$ for some $a\in$ $\mathbb{R}^{d}$, $n\geqslant2$ and the
targets $\mu,$ $\nu$ are supported in $B\left(  b,R_{Y}\right)  $. Let%
\[
\eta=\frac{sn}{n+(d-1)s},\text{ }\lambda=\frac{\eta}{1+\eta}%
\]%
\[
\Vert T_{\mu}-T_{\nu}\Vert_{L^{2}(\rho)}\leqslant CW_{1}(\mu,\nu
)^{\lambda\left(  n-1\right)  /\left(  2n\right)  },
\]
where the constant $C>0$ depends only on $d,s,p,R_{Y}$, not on $a,b$ or the targets.
\end{theorem}

\begin{proof}
Let $a_{\ast}$ attains the minimum in the definition of $M_{\rho,n}\leqslant
M_{\rho,n}(a)<\infty$. Let $\gamma$ be optimal for $W_{1}\left(  \mu
,\nu\right)  $ and let $\sigma_{T_{\mu},\gamma}$ be the repaired plan from
Definition~\ref{def:repaired-plan}. Applying Corollary
\ref{cor:moment-compactE} with source center $a_{\ast}$ gives%
\[
\mathcal{E}_{\nu}(\sigma_{T_{\mu},\gamma})\leqslant C_{1}\mathcal{D}_{\nu
}(\sigma_{T_{\mu},\gamma})^{\eta/(1+\eta)}=C_{1}\mathcal{D}_{\nu}%
(\sigma_{T_{\mu},\gamma})^{\lambda},
\]
where $C_{1}=B_{d,\rho,p,n,c}\left(  a_{\ast}\right)  R_{Y}^{(2+\eta
)/(1+\eta)}.$ With $\delta=W_{1}\left(  \mu,\nu\right)  ,$%
\[
\int\left\vert y-z\right\vert ^{2}d\gamma\leqslant2R_{Y}\delta.
\]
Using $\left(  \ref{eq:D-nth-compact}\right)  ,$ we also have%
\[
\mathcal{D}_{\nu}\left(  \sigma_{T_{\mu},\gamma}\right)  \leqslant\left(
2M_{\rho,n}\right)  \left(  2R_{Y}\right)  ^{1/n}\delta^{1-1/n}.
\]
So combining these estimates yields%
\[
\left\Vert T_{\mu}-T_{\nu}\right\Vert _{L^{2}\left(  \rho\right)  }%
^{2}\leqslant4R_{Y}\delta+2C_{1}\left(  2M_{\rho,n}\right)  ^{\lambda}\left(
2R_{Y}\right)  ^{\lambda/n}\delta^{\lambda\left(  n-1\right)  /n}.
\]
For $0<\delta\leqslant1,$ the slower exponent is $\lambda\left(  n-1\right)
/n.$ Hence
\[
\left\Vert T_{\mu}-T_{\nu}\right\Vert _{L^{2}\left(  \rho\right)  }%
\leqslant\left(  4R_{Y}+2C_{1}\left(  2M_{\rho,n}\right)  ^{\lambda}\left(
2R_{Y}\right)  ^{\lambda/n}\right)  ^{1/2}\delta^{\lambda\left(  n-1\right)
/\left(  2n\right)  }.
\]
For $\delta>1,$%
\[
\left\Vert T_{\mu}-T_{\nu}\right\Vert _{L^{2}\left(  \rho\right)  }%
\leqslant2R_{Y}\leqslant2R_{Y}\delta^{\lambda\left(  n-1\right)  /\left(
2n\right)  }.
\]
Hence we can take $C=\max\left\{  \left(  4R_{Y}+2C_{1}\left(  2M_{\rho
,n}\right)  ^{\lambda}\left(  2R_{Y}\right)  ^{\lambda/n}\right)
^{1/2},2R_{Y}\right\}  .$
\end{proof}

We now provide the stability of Brenier map for compact source and target with
moment. Since it will used later, we will give a version that slightly weakens
the target moment condition to squared tail control. The tail radius cutoff is
set to $1$ to only look at the the tail for measures that might have
singularities near the origin. A probability measure with $\beta$-th moment
implies the tail bound, but not the reverse. The tail bound controls behaviour
near infinity.

\begin{theorem}
[Stability-compact-tail]\label{thm:squared-target}Assume $\rho=fdx$ is a
probability supported in $B(a,R_{X})$, $R_{X}>0$, $f\in L^{p}\left(
\mathbb{R}^{d}\right)  $, $1<p\leqslant\infty$, $s=1-1/p$ and the targets
$\mu,$ $\nu\in\mathcal{P}_{2}\left(  \mathbb{R}^{d}\right)  $ satisfy the tail
control: for some $\beta>0$, $J\geqslant0$, and a common center $b$, that for
$R\geqslant1,$
\begin{equation}
\tau_{2,\mu,b}(R)\leqslant JR^{-\beta},\text{ }\tau_{2,\nu,b}(R)\leqslant
JR^{-\beta}. \label{eq:source-2tail}%
\end{equation}
Let%
\[
H_{s,\beta}=\beta\left(  s+1\right)  +s+2.
\]
Then with $\delta=W_{1}(\mu,\nu),$%
\[
\left\Vert T_{\mu}-T_{\nu}\right\Vert _{L^{2}(\rho)}\leqslant C\delta^{\beta
s/\left(  2H_{s,\beta}\right)  }.
\]
The constant $C$ depends only on $d,p,R_{X},J,$ but not the center $b.$
\end{theorem}

\begin{proof}
Proposition~\ref{prop:compact-Lp} gives, for every $h>0$,%
\begin{equation}
\mathfrak{A}_{\rho}(h)\leqslant Bh^{s}\text{ with }B=K_{d,\rho,p}%
R_{X}^{\left(  d-1\right)  s}>0. \label{eq:compact-Lp}%
\end{equation}

First assume $0<\delta\leqslant1$. Choose a optimal for $W_{1}\left(  \mu
,\nu\right)  $ $\gamma$ and the repaired plan $\sigma_{T_{\mu},\gamma}$. For
every triple law in $\left(  \ref{eq:D-product}\right)  $,
\[
\int|x-a||z-y|d\Lambda\leqslant R_{X}\int|z-y|d\gamma=R_{X}\delta.
\]
Thus $\left(  \ref{eq:D-C}\right)  $ yields%
\[
\mathcal{D}_{\nu}(\sigma_{T_{\mu},\gamma})\leqslant2R_{X}\delta.
\]
To bound the direct coupling term, write $Z=\left\vert y-z\right\vert $ and
take $R>0.$ On $\left\{  Z>2R\right\}  $, $Z^{2}\leqslant4\max\{|y-b|^{2}%
,|z-b|^{2}\}$. Hence,
\[
Z^{2}1_{\{Z>2R\}}\leqslant4|y-b|^{2}1_{\{|y-b|>R\}}+4|z-b|^{2}1_{\{|z-b|>R\}}%
.
\]
On $\{Z\leqslant2R\}$, $Z^{2}\leqslant2RZ$. Integrating the two parts gives
for $R\geqslant1$,
\[
\int Z^{2}d\gamma\leqslant2R\delta+4\tau_{2,\mu,b}(R)+4\tau_{2,\nu
,b}(R)\leqslant2R\delta+8JR^{-\beta}.
\]
Choosing $R=\delta^{-1/(\beta+1)}\geqslant1$ gives
\[
\int Z^{2}d\gamma\leqslant(2+8J)\delta^{\beta/(\beta+1)}.
\]
Theorem~\ref{thm:error-bound} combined with the source oscillation bound
$\left(  \ref{eq:compact-Lp}\right)  $ and the target tail bound $\left(
\ref{eq:source-2tail}\right)  $ gives
\[
\mathcal{E}_{\nu}\left(  \sigma_{T_{\mu},\gamma}\right)  \leqslant8R_{X}\delta
Rh^{-1}+18R^{2}Bh^{s}+8JR^{-\beta},\text{ }R\geqslant1,\ h>0.
\]
For fixed $R$, choosing $h$ to minimizes the sum of the first two terms
yields, for $R\geqslant1,$%
\[
\mathcal{E}_{\nu}\left(  \sigma_{T_{\mu},\gamma}\right)  \leqslant C_{1}%
\delta^{s/\left(  s+1\right)  }R^{\left(  s+2\right)  /\left(  s+1\right)
}+8JR^{-\beta},
\]
where $C_{1}=c_{s,1}\left(  8R_{X}\right)  ^{s/\left(  s+1\right)  }\left(
18B\right)  ^{1/\left(  s+1\right)  }.$ Now choose $R=\delta^{-s/\left[
\beta\left(  s+1\right)  +s+2\right]  }\geqslant1$ to balance the power.
Substitution gives%
\[
\mathcal{E}_{\nu}\left(  \sigma_{T_{\mu},\gamma}\right)  \leqslant\left(
C_{1}+8J\right)  \delta^{\beta s/\left[  \beta\left(  s+1\right)  +s+2\right]
}.
\]
$\left(  \ref{eq:map-graph}\right)  $ now gives
\[
\Vert T_{\mu}-T_{\nu}\Vert_{L^{2}(\rho)}^{2}\leqslant\left(  4+16J\right)
\delta^{\beta/(\beta+1)}+\left(  2C_{1}+16J\right)  \delta^{\beta s/\left[
\beta\left(  s+1\right)  +s+2\right]  }.
\]
The latter term has a slower exponent. Hence%
\[
\Vert T_{\mu}-T_{\nu}\Vert_{L^{2}(\rho)}\leqslant\left(  (4+2C_{1}+32J\right)
^{1/2}\delta^{\beta s/\left[  2\left[  \beta\left(  s+1\right)  +s+2\right]
\right]  }.
\]
Taking square roots proves the theorem in this range. For $\delta=0$,
uniqueness gives equal maps. For $\delta>1$, the tail bound condition gives
\[
\int|y-b|^{2}d\mu,\text{ }\int|z-b|^{2}d\nu\leqslant1+J.
\]
Thus%
\[
\Vert T_{\mu}-T_{\nu}\Vert_{L^{2}(\rho)}^{2}\leqslant2\int|y-b|^{2}%
d\gamma+2\int|z-b|^{2}d\gamma\leqslant4\left(  1+J\right)  ,
\]
which after taking square root is bounded by the same $\delta$-power term,%
\[
\Vert T_{\mu}-T_{\nu}\Vert_{L^{2}(\rho)}\leqslant2\left(  1+J\right)
^{1/2}\delta^{\beta s/\left[  2\left[  \beta\left(  s+1\right)  +s+2\right]
\right]  }.
\]
This proves the theorem with constant $C=\max\left\{  \left(  (4+2C_{1}%
+32J\right)  ^{1/2},2\left(  1+J\right)  ^{1/2}\right\}  .$
\end{proof}

\begin{theorem}
[Stability-moment-moment]\label{thm:both-unbounded}Assume
Assumption~\ref{ass:brenier}, let $\rho=fdx$, $f\in L^{p}\left(
\mathbb{R}^{d}\right)  $, $1<p\leqslant\infty$, $s=1-1/p$, $M_{\rho
,n}(a)<\infty$ for some $a\in$ $\mathbb{R}^{d}$, $n\geqslant2$ and the targets
$\mu,$ $\nu$ have $M_{\mu,\nu,m}<\infty$. Let
\[
\alpha=\frac{m(n-1)-n}{n(m-1)},
\]
and let $\eta$ and $\Theta_{\eta,m}$ be defined as in
Theorem~\ref{thm:unbounded-power}. Then%
\begin{equation}
\Vert T_{\mu}-T_{\nu}\Vert_{L^{2}(\rho)}\leqslant CW_{1}(\mu,\nu
)^{\alpha\Theta_{\eta,m}/2}, \label{eq:both-unbounded}%
\end{equation}
The constant $C$ depends only on $d,p,n,m,M_{\rho,n},M_{\mu,\nu,m}$.
\end{theorem}

\begin{proof}
Choose a $W_{1}$-optimal coupling $\gamma\in\Pi(\mu,\nu)$, form the triple law
$\Lambda_{T_{\mu},\gamma}$ in $\left(  \ref{eq:repaired-triple}\right)  $, and
let $\sigma_{T_{\mu},\gamma}\in\Pi(\rho,\nu)$ be its repaired $(x,z)$%
-marginal. With $\delta=W_{1}\left(  \mu,\nu\right)  ,$ $\left(
\ref{eq:W1-D}\right)  $ gives
\[
\mathcal{D}_{\nu}\left(  \sigma_{T_{\mu},\gamma}\right)  \leqslant C_{1}%
\delta^{\alpha},
\]
where $C_{1}=\left(  2M_{\rho,n}\right)  \left(  2M_{\mu,\nu,m}\right)
^{1-\alpha}$. Now choosing $a_{\ast}$ that attains the infimum in the
definition of $M_{\rho,n}$ and applying $\left(  \ref{eq:unbounded-E}\right)
$ with this center $a_{\ast}$ give
\[
\mathcal{E}_{\nu}\left(  \sigma_{T_{\mu},\gamma}\right)  \leqslant
C_{2}\mathcal{D}_{\nu}\left(  \sigma_{T_{\mu},\gamma}\right)  ^{\Theta
_{\eta,m}}\leqslant C_{2}C_{1}^{\Theta_{\eta,m}}\delta^{\alpha\Theta_{\eta,m}%
},
\]
where $C_{2}=Q_{d,\rho,p,n,\nu,m}\left(  a_{\ast}\right)  .$ The $\left(
y,z\right)  $-coupling term is controlled by $\left(  \ref{eq:W1-direct}%
\right)  $,
\[
\int\left\vert y-z\right\vert ^{2}d\gamma\leqslant C_{3}\delta^{\left(
m-2\right)  /\left(  m-1\right)  },
\]
where $C_{3}=\left(  2M_{\mu,\nu,m}\right)  ^{m/\left(  m-1\right)  }.$
Substitution in $\left(  \ref{eq:map-graph}\right)  $ yields
\[
\left\Vert T_{\mu}-T_{\nu}\right\Vert _{L^{2}(\rho)}^{2}\leqslant2C_{3}%
\delta^{\left(  m-2\right)  /\left(  m-1\right)  }+2C_{2}C_{1}^{\Theta
_{\eta,m}}\delta^{\alpha\Theta_{\eta,m}}.
\]
It remains to identify the slower power of $\delta$. Since $m>2,$%
\[
0<\alpha=\frac{m(n-1)-n}{n(m-1)}<1,\text{ }\Theta_{\eta,m}=\frac{\eta\left(
m-2\right)  }{m(1+\eta)-\eta}<\frac{m-2}{m-1}.
\]
Thus, for $0<\delta\leqslant1$,%
\[
\left\Vert T_{\mu}-T_{\nu}\right\Vert _{L^{2}(\rho)}^{2}\leqslant\left(
2C_{3}+2C_{2}C_{1}^{\Theta_{\eta,m}}\right)  \delta^{\alpha\Theta_{\eta,m}}.
\]
Taking a square root proves $\left(  \ref{eq:both-unbounded}\right)  $. \ If
$\delta=0$, then $\mu=\nu$ and uniqueness of the Brenier map gives $T_{\mu
}=T_{\nu},$ $\rho$-a.e.. For $\delta>1$, for every $b$,
\begin{align*}
\left\Vert T_{\mu}-T_{\nu}\right\Vert _{L^{2}\left(  \rho\right)  }^{2}  &
\leqslant2\int\left\vert T_{\mu}-b\right\vert ^{2}d\rho+2\int\left\vert
T_{\nu}-b\right\vert ^{2}d\rho\\
&  =2\int\left\vert y-b\right\vert ^{2}d\mu+2\int\left\vert z-b\right\vert
^{2}d\nu\leqslant4\max\left\{  M_{\mu,m}\left(  b\right)  ,M_{\nu,m}\left(
b\right)  \right\}  ^{2}.
\end{align*}
The last step uses monotonicity of moment norm and that $m>2$. Taking the
infimum over $b$ gives%
\[
\left\Vert T_{\mu}-T_{\nu}\right\Vert _{L^{2}\left(  \rho\right)  }%
\leqslant2M_{\mu,\nu,m}\leqslant2M_{\mu,\nu,m}\delta^{\alpha\Theta_{\eta,m}%
/2}.
\]
Thus one constant valid for all $\delta\geqslant0$ is $\max\{(2C_{3}%
+2C_{2}C_{1}^{\Theta_{\eta,m}})^{1/2},2M_{\mu,\nu,m}\}$.
\end{proof}

\subsection{Sharpness of Theorem~\ref{thm:both-unbounded} in dimension one}

\label{sec:sharp-1D}

The sharp error bound exponent in Theorem~\ref{thm:unbounded-power} does not
necessarily make the exponent for the Brenier map bound $\left(
\ref{eq:both-unbounded}\right)  $ of Theorem~\ref{thm:both-unbounded}\ sharp,
although the latter is indeed sharp for compact source and target proved in
Theorem \ref{thm:compact-maps}. In dimension one the exponent is in fact
strictly suboptimal, as the following shows.

Let $d=1$. Recall that for any probability measure $\eta$ on $\mathbb{R}$, the
distribution function and its quantile function are given by
\[
F_{\eta}(x)=\eta\left(  ({-\infty},x]\right)  ,\text{ }G_{\eta}(t)=\inf
\left\{  x\in\mathbb{R}:F_{\eta}(x)\geqslant t\right\}  ,\text{ }0<t<1.
\]
Observe that $G_{\eta}(t)\leqslant x$ if and only if $F_{\eta}(x)\geqslant t.$
Hence%
\[
\left\{  t:G_{\eta}(t)>x\right\}  =\left(  F_{\eta}(x),1\right)  .
\]

\begin{theorem}
\label{thm:1d-stability}Let $\rho,$ $\mu,$ $\nu\in\mathcal{P}_{2}\left(
\mathbb{R}\right)  .$ Suppose that $\rho\ll dx$ and $M_{\mu,\nu,m}<\infty$ for
$m>2$. Then with $\delta=W_{1}\left(  \mu,\nu\right)  ,$%
\begin{equation}
\Vert T_{\mu}-T_{\nu}\Vert_{L^{2}(\rho)}\leqslant(2M_{\mu,\nu,m}%
)^{m/[2(m-1)]}\delta^{(m-2)/[2(m-1)]}.\label{eq:1D-two-map}%
\end{equation}
The exponent is sharp. Note that the constant does not depend on the source
parameters. Note that if $M_{\mu,\nu,m}=0$, then both targets are forced to be
the same dirac measure on the point that attains $M_{\mu,\nu,m}$.
\end{theorem}

\begin{proof}
\textbf{1}. Let $X\sim\rho,$ then the absolute continuity of $\rho$ ensures
that $F_{\rho}(X)$ is uniform on $\left(  0,1\right)  $ and
\begin{equation}
\left\{  t:G_{\mu}(t)>x\right\}  =\left(  F_{\mu}(x),1\right)  ,\text{
\ }\left\{  t:G_{\nu}(t)>x\right\}  =\left(  F_{\nu}(x),1\right)  .
\label{eq:distribution-inverse}%
\end{equation}
Moreover the non-decreasing maps $G_{\mu}(F_{\rho}(x)),$ $G_{\nu}(F_{\rho
}(x))$ pushes $\rho$ to the targets $\mu,$ $\nu$, and they are gradient of
convex functions. By uniqueness they equal $T_{\mu}\left(  x\right)  ,$
$T_{\nu}\left(  x\right)  $ for $\rho$-a.e. $x.$

\textbf{2}. Let $\gamma_{G}=(G_{\mu},G_{\nu})_{\#}dt$, here $dt$ denotes
uniform measure on $\left(  0,1\right)  .$ We claim that the joint law
$\gamma_{G}$ is optimal for $W_{1}\left(  \mu,\nu\right)  $. Indeed, for any
$\gamma\in\Pi(\mu,\nu)$, we have by Tonelli's theorem,
\begin{align*}
\int\left\vert y-z\right\vert d\gamma(y,z)  &  =\int_{\mathbb{R}}%
\int\left\vert 1_{\left\{  y>r\right\}  }-1_{\left\{  z>r\right\}
}\right\vert d\gamma(y,z)dr\geqslant\int_{\mathbb{R}}\left\vert \int\left(
1_{\left\{  y>r\right\}  }-1_{\left\{  z>r\right\}  }\right)  d\gamma
(y,z)\right\vert dr\\
&  =\int_{\mathbb{R}}|\mu((r,\infty))-\nu((r,\infty))|dr.
\end{align*}
The lower bound already equals $W_{1}\left(  \mu,\nu\right)  $ by
\cite[Proposition 2.17]{santambrogio2015optimal}. It remains to show that the
inequality is attained by $\gamma_{G}$. Note that $\gamma_{G}\in\Pi(\mu,\nu).$
In $\left(  \ref{eq:distribution-inverse}\right)  ,$ one of those terminal
intervals contains the other, so their symmetric difference is the interval
between the two left endpoints. It follows that%
\[
\int_{0}^{1}\left\vert 1_{\left\{  G_{\mu}(t)>r\right\}  }-1_{\left\{  G_{\nu
}(t)>r\right\}  }\right\vert dt=\left\vert F_{\mu}(r)-F_{\nu}(r)\right\vert .
\]
For real $y\geqslant z$, the indicators $1_{\left\{  y>r\right\}  }$ and
$1_{\left\{  z>r\right\}  }$ differ exactly on $[z,y)$, whose length is $y-z$.
So by symmetry we have
\begin{equation}
\left\vert y-z\right\vert =\int_{\mathbb{R}}\left\vert 1_{\left\{
y>r\right\}  }-1_{\left\{  z>r\right\}  }\right\vert dr.
\label{eq:indicators-difference}%
\end{equation}
Using the definition of $\gamma_{G}$ and Tonelli's theorem, we therefore have%
\begin{align*}
\int\left\vert y-z\right\vert d\gamma_{G}(y,z)  &  =\int_{0}^{1}\left\vert
G_{\mu}\left(  t\right)  -G_{\nu}\left(  t\right)  \right\vert dt=\int_{0}%
^{1}\left(  \int_{\mathbb{R}}\left\vert 1_{\left\{  G_{\mu}(t)>r\right\}
}-1_{\left\{  G_{\nu}(t)>r\right\}  }\right\vert dr\right)  dt\\
&  =\int_{\mathbb{R}}\int_{0}^{1}\left\vert 1_{\left\{  G_{\mu}(t)>r\right\}
}-1_{\left\{  G_{\nu}(t)>r\right\}  }\right\vert dtdr=\int_{\mathbb{R}%
}\left\vert F_{\mu}(r)-F_{\nu}(r)\right\vert dr\\
&  =\int_{\mathbb{R}}|\mu((r,\infty))-\nu((r,\infty))|dr.
\end{align*}
The second equality uses $\left(  \text{\ref{eq:indicators-difference}%
}\right)  $. This shows the inequality is attained by $\gamma_{G}$. Note that
the same calculation applies to squared distance. Since $U=F_{\rho}\left(
X\right)  $ is uniform, so with $Z=\left\vert G_{\mu}\left(  U\right)
-G_{\nu}\left(  U\right)  \right\vert $ we have
\[
\mathbb{E}Z=\delta,\text{ }\mathbb{E}Z^{2}=\int_{0}^{1}\left\vert G_{\mu
}(t)-G_{\nu}(t)\right\vert ^{2}dt=\left\Vert T_{\mu}-T_{\nu}\right\Vert
_{L^{2}\left(  \rho\right)  }^{2}.
\]

\textbf{3}. Recall $M_{\mu,\nu,m}$ from $\left(  \ref{eq:M-mu-nu-pair}\right)
.$ The discussion at the beginning of Section \ref{sec:target-profile}\ shows
$\max\{M_{\mu,m}(b),M_{\nu,m}(b)\}$ is continuous in $b$ and coercive. So it
attains its infimum at a center $b_{\ast}$. Let $\eta\in\left\{  \mu
,\nu\right\}  .$ By the quantile pushforward,
\[
\int_{0}^{1}\left\vert G_{\eta}(t)-b_{\ast}\right\vert ^{m}dt=\int\left\vert
y-b_{\ast}\right\vert ^{m}d\eta(y)=\left[  M_{\eta,m}(b_{\ast})\right]  ^{m}.
\]
Write $G_{\mu}-G_{\nu}=(G_{\mu}-b_{\ast})-(G_{\nu}-b_{\ast})$. By Minkowski's
inequality we have
\[
\left\Vert Z\right\Vert _{m}=\left(  \int_{0}^{1}\left\vert G_{\mu}(t)-G_{\nu
}(t)\right\vert ^{m}dt\right)  ^{1/m}\leqslant M_{\mu,m}(b_{\ast})+M_{\nu
,m}(b_{\ast})\leqslant2M_{\mu,\nu,m}.
\]
Set $r_{1}=\left(  m-2\right)  /\left(  m-1\right)  $, $r_{2}=1/\left(
m-1\right)  $. Then $r_{1}+r_{2}=1$, $r_{1}+mr_{2}=2$ and both $r_{1},$
$r_{2}$ are in $(0,1)$. Therefore we can use H\"{o}lder's inequality with
exponents $1/r_{1},$ $1/r_{2}$ to obtain%
\[
\mathbb{E}Z^{2}=\mathbb{E}\left[  Z^{r_{1}}\left(  Z^{m}\right)  ^{r_{2}%
}\right]  \leqslant\left(  \mathbb{E}Z\right)  ^{r_{1}}\left(  \mathbb{E}%
Z^{m}\right)  ^{r_{2}}\leqslant\delta^{(m-2)/(m-1)}(2M_{\mu,\nu,m}%
)^{m/(m-1)}.
\]
Taking square root yields the desired stability.

\textbf{4}. We now prove the sharpness of the exponent by constructing
probability marginals that saturate the exponent. In the following we fix
$\rho=dx$ as uniform on $(0,1)$, so the density $f$ is constant $1.$ For
$0<q<1$, we let $R=q^{-1/m}$ and define the targets%
\[
\mu_{q}=\delta_{0},\text{ \ \ }\nu_{q}=(1-q)\delta_{0}+q\delta_{R}.
\]
The source has $\Vert f\Vert_{p}=1$ for all $p>0$, and $\int_{0}^{1}%
x^{n}dx=1/(n+1)$. Since $qR^{m}=1,$ both targets have bounded centralized
$m$-th moment. Their Brenier maps are%
\[
T_{\mu_{q}}(x)=0,\text{ \ \ }T_{\nu_{q}}(x)=R1_{\left\{  x>1-q\right\}  }.
\]
The $W_{1}$-distance and the squared error are%
\[
\delta=W_{1}(\mu_{q},\nu_{q})=qR=q^{(m-1)/m},
\]%
\[
\Vert T_{\mu_{q}}-T_{\nu_{q}}\Vert_{2}^{2}=qR^{2}=q^{(m-2)/m}=\delta
^{(m-2)/(m-1)}.
\]
For sufficiently small $\delta>0,$ $\Vert T_{\mu_{q}}-T_{\nu_{q}}\Vert
_{2}=\delta^{(m-2)/\left[  2(m-1)\right]  }$ is attained by choosing
$q=\delta^{m/(m-1)}$. Therefore the proof for sharpness is complete.
\end{proof}

\emph{Improvement discussion}. Let $k_{0}$ be the exponent in $\left(
\ref{eq:both-unbounded}\right)  $. Recall that in the proof of Theorem
\ref{thm:both-unbounded}, we already know that the $\alpha$ in $k_{0}$ is in
$\left(  0,1\right)  .$ Hence
\[
\frac{k_{0}}{(m-2)/\left[  2(m-1)\right]  }=\frac{\alpha\eta(m-1)}%
{m+\eta(m-1)}<1.
\]
So the one-dimensional exponent is sharper than the general exponent given by
Theorem \ref{thm:both-unbounded}.

\appendix

\section{Comparisons with Delalande-M\'{e}rigot
\cite{delalande2023quantitative}}

\label{sec:DM-noncompact}

We compare \cite[Section 4]{delalande2023quantitative}. Its Theorem 4.3,
Remark 4.2, and Corollary 4.4 allow noncompact target supports. Write $m$ for
the target moment order called $p$ in Corollary 4.4, to distinguish it from
our source density exponent. All comparisons concern the $L^{2}(\rho)$ norm of
the map difference. Constants are uniform with the stated source and target
parameters fixed. We compare orders, not optimal constants.

\subsection{Corollary 4.4 with moment bounds alone}

Let $\rho=fdx$ be supported in $B\left(  a,R_{X}\right)  ,$ $f\in L^{\infty
}\left(  \mathbb{R}^{d}\right)  $ and the targets $\mu,$ $\nu$ have
$M_{m,\mu,\nu}\leqslant K_{m}$ for some $K_{m}>0$ and $m>2$. See $\left(
\ref{eq:M-mu-nu-pair}\right)  $ for the definion of $M_{\mu,\nu,m}$ and let
$b_{\ast}$attain the infimum of $M_{\mu,\nu,m}$. Let $\eta\in\left\{
\mu,\text{ }\nu\right\}  .$ On $\{\left\vert y-b_{\ast}\right\vert >R\}$,
$\left\vert y-b_{\ast}\right\vert ^{2}\leqslant R^{2-m}\left\vert y-b_{\ast
}\right\vert ^{m}$ gives%
\begin{equation}
\tau_{2,\eta,b_{\ast}}(R)\leqslant K_{m}R^{2-m}. \label{eq:moment-tail-bound}%
\end{equation}
Hence Theorem \ref{thm:squared-target} applies with $s=1,$ $\beta=m-2$,
$J=K_{m}$, and gives for some $C>0,$
\begin{equation}
\Vert T_{\mu}-T_{\nu}\Vert_{2}\leqslant C\delta^{k_{m}}\text{\ where }%
k_{m}=\frac{m-2}{4m-2}. \label{eq:moment-only-bound}%
\end{equation}
Here we only used a bounded containing source ball, an upper density bound and
a target with moment $m>2.$ Convexity, a positive density lower bound, and
potential regularity were not used. The constant $C$ is uniform for fixed $d,$
$\Vert f\Vert_{\infty},$ $R_{X},$ $m,$ $K_{m}$, independently of the target centers.

In \cite[Corollary 4.4]{delalande2023quantitative}, the source instead has
two-sided positive density bounds on a compact convex set, and the target
moment order satisfies $m\geqslant4$, $m>d$. Its map norm exponent is
$m/(6m+16d)$. On this common class the difference of exponents is,%
\[
k_{m}-\frac{m}{6m+16d}=\frac{m^{2}+(8d-5)m-16d}{(2m-1)(6m+16d)}.
\]
The numerator is strictly positive,%
\[
m^{2}+(8d-5)m-16d=(m-4)^{2}+(8d+3)(m-4)+16d-4>0.
\]
Thus the map upper bound is recovered and strictly improved in small-distance
order. Our moment-only bound $\left(  \ref{eq:moment-only-bound}\right)  $
also allows $2<m<4$ and $m\leqslant d$, and drops the two source density bounds.

Corollary 4.4 in the cited paper can be strengthened. Indeed, Proposition 3.3
there (or \cite[Proposition 2.22]{berman2013real}) gives the potential
exponent $\alpha=1-d/m$ when $m>d$. Substituting $q=m/2$ into Remark 4.2
yields
\[
\frac{q-1}{2(q(7-4\alpha)-3)}=\frac{m-2}{2(m(7-4\alpha)-6)}=\frac
{m-2}{6m+8d-12}\geqslant\frac{m}{6m+16d}\text{.}%
\]
The inequality is valid under Corollary 4.4's $m\geqslant4$. So this gives a
stronger refinement. For every $m>\max\left\{  2,d\right\}  $, our new
exponent $k_{m}$ still strictly exceeds this refinement,
\[
k_{m}-\frac{m-2}{6m+8d-12}=\frac{m-2}{4m-2}-\frac{m-2}{6m+8d-12}%
=\frac{(m-2)(2m+8d-10)}{(4m-2)(6m+8d-12)}>0.
\]

\subsection{Theorem 4.3 and Remark 4.2 with potential regularity}

Since this might be useful beyond the comparison, we state it as a theorem.
Recall that our moment order $m$ is $2q$ there.

\begin{theorem}
Suppose now that $\rho=fdx$ is supported on a compact convex body
$\mathcal{X}$, $f$ is bounded $f\leqslant M_{\rho}$ on $\mathcal{X}$ for some
$M_{\rho}>0$, and the targets $\mu,\nu$ have $M_{\mu,\nu,m}\leqslant
K_{m}<\infty$ for some $m\geqslant2$ and $K_{m}>0$. Suppose that both Brenier
potentials $u_{\eta}$, $\eta\in\left\{  \mu,\text{ }\nu\right\}  $ satisfy,
for fixed $H_{\alpha}>0$ and $0<\alpha<1,$%
\begin{equation}
\left\vert u_{\eta}\left(  x\right)  -u_{\eta}\left(  x^{\prime}\right)
\right\vert \leqslant H_{\alpha}\left\vert x-x^{\prime}\right\vert ^{\alpha
},\text{ }\forall x,\text{ }x^{\prime}\in\mathcal{X}.
\label{eq:Holder-potential}%
\end{equation}
Let%
\[
\kappa=\frac{1}{1-\alpha},\text{ }\ell=\max\{m,\kappa\}.
\]
Then for $m>2,$%
\begin{equation}
\Vert T_{\mu}-T_{\nu}\Vert_{2}\leqslant C\delta^{k_{\ell}}\text{ where
}k_{\ell}=\frac{\ell-2}{4\ell-2}. \label{eq:higher-regularity-bound}%
\end{equation}
The stability continues to hold at $m=2$ if $\alpha>1/2.$ The constant $C>0$
depends only on $m,K_{m},\alpha,H_{\alpha}$, the source geometry, and the
density upper bound. No target support bound is imposed.
\end{theorem}

\begin{remark}
Theorem 4.3 in the cited paper assumes in addition a finite $4$-th moment and
obtains the exponent $1/[2(11-8\alpha)]$. If we do not use the H\"{o}lder
bound and only use its finite $4$-th moment assumption for the targets, then
the moment-only exponent $\left(  \ref{eq:moment-only-bound}\right)  ,$
$k_{4}=1/7,$ dominate the exponent there only when $\alpha<15/16$. The two
exponents equal at $\alpha=15/16$.
\end{remark}

\begin{proof}
\textbf{1}. \emph{The case }$m>2$. If $\kappa\leqslant m$, then $\ell=m.$ We
use the moment bound condition. Since $M_{\mu,\nu,m}\leqslant K_{m}$ implies
that the tail bound $\left(  \ref{eq:moment-tail-bound}\right)  $ holds. Now
using Theorem \ref{thm:squared-target} with $s=1,$ $\beta=\ell-2=m-2$ and
$J=K_{m}$ or simply using $\left(  \ref{eq:moment-only-bound}\right)  ,$ we
obtain the desired stability $\left(  \ref{eq:higher-regularity-bound}\right)
$ for $\kappa\leqslant m.$ If instead $\kappa>m,$ then $\ell=\kappa$. First we
derive a target tail bound from the H\"{o}lder assumption. At an interior
differentiable point $x$, let $y=\nabla u_{\eta}(x)$. If $y\neq0$, the point
$x+r\left(  x\right)  y/\left\vert y\right\vert $ belongs to $\mathcal{X}$.
The subgradient inequality and the H\"{o}lder continuity $\left(
\ref{eq:Holder-potential}\right)  $ give
\[
r\left(  x\right)  \left\vert y\right\vert \leqslant u_{\eta}\left(
x+r\left(  x\right)  y/\left\vert y\right\vert \right)  -u_{\eta}\left(
x\right)  \leqslant H_{\alpha}r\left(  x\right)  ^{\alpha}.
\]
Thus
\begin{equation}
\left\vert y\right\vert \leqslant H_{\alpha}r(x)^{\alpha-1}. \label{eq:y-rx}%
\end{equation}
It also holds for $p=0$. This is the erosion mechanism of \cite[Propositions
4.5--4.6]{delalande2023quantitative}. The calculation here records the
constants needed for our tail estimation. Since a compact convex domain is
Lipschitz, the geometric layer estimate $\left(  \ref{eq:LM-boundary-layer}%
\right)  $ holds with a domain-dependent constant $K_{\mathcal{X}}>0$. The
constant may not be optimal for convex domain, but it is sufficient for our
uses. The layer estimate also gives $\left\vert \partial\mathcal{X}\right\vert
=0$ by letting $t\rightarrow0$, so the preceding argument covers $\rho$-a.e.
$x$. Let $C_{H}=M_{\rho}K_{\mathcal{X}}H_{\alpha}^{\kappa}$. Since $(\nabla
u_{\eta})_{\#}\rho=\eta$, by $\left(  \ref{eq:y-rx}\right)  $ we have for
$R>0$ that
\[
\eta\left(  \left\vert y\right\vert >R\right)  =\rho\left(  \left\vert \nabla
u_{\eta}\right\vert >R\right)  \leqslant\rho\left(  r\left(  x\right)
<\left(  H_{\alpha}/R\right)  ^{\kappa}\right)  \leqslant C_{H}R^{-\kappa}.
\]
Since $m>2,$ so $\kappa>2$. Applying Tonelli's theorem to
\[
|y|^{2}1_{\{|y|>R\}}=R^{2}1_{\{|y|>R\}}+\int_{R}^{\infty}2t1_{\{|y|>t\}}dt
\]
we get
\begin{align}
\tau_{2,\eta,0}(R)  &  =R^{2}\eta\left(  \left\vert y\right\vert >R\right)
+\int_{R}^{\infty}2t\eta\left(  \left\vert y\right\vert >t\right)
dt\label{eq:squared-tail-bound}\\
&  \leqslant C_{H}R^{2-\kappa}+2C_{H}\int_{R}^{\infty}t^{1-\kappa}%
dt=C_{H}\frac{\kappa}{\kappa-2}R^{-(\kappa-2)}.\nonumber
\end{align}
This is a tail bound at the exponent $\kappa-2$. Using Theorem
\ref{thm:squared-target} with $s=1,$ $\beta=\kappa-2=\ell-2$, $J=C_{H}%
\kappa/(\kappa-2)$, we therefore get the stability $\left(
\ref{eq:higher-regularity-bound}\right)  $ for $\kappa>m.$

\textbf{2}. \emph{The case }$m=2$. If $\alpha>1/2$, then $\kappa>2=m$, so the
tail bound $\left(  \ref{eq:squared-tail-bound}\right)  $ we just derived
continues to hold and yields the same stability bound with strictly positive
exponent $k_{\ell}=\left(  \kappa-2\right)  /\left(  4\kappa-2\right)
=(2\alpha-1)/[2(1+\alpha)]>0.$
\end{proof}

\emph{Improvement discussion}. We now check explicitly that $\left(
\ref{eq:higher-regularity-bound}\right)  $ improves the cited Remark 4.2
exponent
\[
k_{\mathrm{DM}}(m,\alpha)=\frac{m-2}{2(m(7-4\alpha)-6)}.
\]
First we compare the case $m>2.$ The denominator of $k_{\mathrm{DM}}$ is
positive, since $7-4\alpha>3$ and $m>2$. If $\kappa\leqslant m$, then
$\ell=m.$ We form the difference%
\[
k_{\ell}-k_{\mathrm{DM}}=\left[  \frac{m-2}{4m-2}-\frac{m-2}{2(m(7-4\alpha
)-6)}\right]  =\left(  m-2\right)  \cdot\frac{2(m(7-4\alpha)-6)-(4m-2)}%
{\left(  4m-2\right)  \left[  2(m(7-4\alpha)-6)\right]  }.
\]
By the definition of $\kappa$ and that $\kappa\leqslant m$ we have
$\alpha\leqslant1-1/m$. So the numerator is positive,
\begin{align*}
2(m(7-4\alpha)-6)-(4m-2)  &  =2[m(5-4\alpha)-5]\\
&  \geqslant2[m(1+4/m)-5]=2(m-1)>0.
\end{align*}
Thus the new exponent $k_{\ell}$ is strictly larger than $k_{\mathrm{DM}}$. If
$\kappa>m$, then $\ell=\kappa.$ Substituting $\alpha=1-1/\kappa$ yields%
\[
k_{\ell}-k_{\mathrm{DM}}=\frac{\kappa-2}{4\kappa-2}-\frac{m-2}{2(3m+4m/\kappa
-6)}=\frac{(\kappa-2)(3m+4m/\kappa-6)-(m-2)(2\kappa-1)}{2(2\kappa
-1)(3m+4m/\kappa-6)}.
\]
Since $m>2,$ the numerator is strictly positive,
\[
(\kappa-2)(3m+4m/\kappa-6)-(m-2)(2\kappa-1)=(m-1)(m-2)+(\kappa-m)(m-2+8/\kappa
)>0.
\]
The common denominator is also strictly positive, since $\kappa>m>2$. This
shows strict improvement in both cases. In particular, under Theorem 4.3's
$4$th-moment assumption, $m=4.$ The equation $\kappa=m=4$ then gives
$\alpha=3/4.$ So our stability exponent is%
\[
\left\{
\begin{array}
[c]{ll}%
1/7, & 0<\alpha\leqslant3/4,\\
(2\alpha-1)/[2(1+\alpha)], & 3/4<\alpha<1.
\end{array}
\right.
\]
In the second line, we have substituted $\kappa=1/(1-\alpha)$ into $k_{\ell
}=k_{\kappa}=(\kappa-2)/(4\kappa-2)$. The two expressions agree at
$\alpha=3/4$. The preceding comparison with $m=4$ proves that they strictly
improve $1/[2(11-8\alpha)]$ throughout $0<\alpha<1$. When $m=2,$ this is
Remark 4.2 at $q=1$, the exponent there is zero. If $\alpha>1/2$, then
$\kappa>2=m$, our exponent $k_{\ell}=\left(  \kappa-2\right)  /\left(
4\kappa-2\right)  =(2\alpha-1)/[2(1+\alpha)]$ is still strictly positive.

\section{Comparisons with Letrouit-M\'{e}rigot \cite{letrouit2026gluing}}

\label{sec:LM-comparison}

We now compare with the Brenier map stability estimates \cite[Theorem 1.4
(1.7), Theorem 1.7 (1.10), Theorem 1.10 (1.13), Theorem 1.11 (1.15) and
Theorem 1.12 (1.17)]{letrouit2026gluing}. As already mentioned in its Remark
1.5, these stabilities are extendible to non-compact targets under additional
moments via the route of \cite[Section 4]{delalande2023quantitative}. For
comparison purpose we only treat compact targets, non-compact targets can be
treated as in Theorem \ref{thm:squared-target} or Theorem
\ref{thm:both-unbounded}. Therefore we assume that throughout this section,
the targets $\mu,$ $\nu$ lie in one fixed common ball $B(b,R_{Y})$, $R_{Y}>0$,
and $\delta=W_{1}(\mu,\nu)$.

\subsection{Source-tail envelopes}

\label{sec:source-envelopes}

For rapidly decaying source density, the annular estimate $\left(
\ref{eq:A-annuli}\right)  $\ with $p=\infty$ retains the linear bound$,$%
\[
\mathfrak{A}_{\rho}\left(  h\right)  \leqslant Bh,\text{ where }B=K_{d,\infty
}^{\mathrm{ann}}W_{p,a}\left(  f\right)
\]
With target measures supported in $B\left(  b,R_{Y}\right)  $, $R_{Y}>0$, the
error bound Theorem \ref{thm:error-bound} therefore gives, for every feasible
$\sigma\in\Pi\left(  \rho,\nu\right)  $,
\[
\mathcal{E}_{\nu}\left(  \sigma\right)  \leqslant4R_{Y}\mathcal{D}_{\nu
}\left(  \sigma\right)  /h+18BR_{Y}^{2}h,\text{ }h>0.
\]
Choosing the minimizing $h$ gives
\begin{equation}
\mathcal{E}_{\nu}\left(  \sigma\right)  \leqslant C_{E}\mathcal{D}_{\nu
}\left(  \sigma\right)  ^{1/2},\text{ where }C_{E}=12\left(  2\right)
^{1/2}B^{1/2}R_{Y}^{3/2}. \label{eq:E-sqrtD}%
\end{equation}

\begin{theorem}
[Gaussian upper envelope]\label{thm:logconcave}Let $\rho=fdx$ and suppose, for
some $a\in\mathbb{R}^{d}$ and $C_{0},$ $\kappa>0$, that
\[
0\leqslant f\left(  x\right)  \leqslant C_{0}e^{-\kappa\left\vert
x-a\right\vert ^{2}/2},\text{ for almost every }x.
\]
Let $\mu,$ $\nu$ be supported in $B(b,R_{Y})$, $R_{Y}>0$. Then, with
$\delta=W_{1}(\mu,\nu),$%
\begin{equation}
\left\Vert T_{\mu}-T_{\nu}\right\Vert _{L^{2}\left(  \rho\right)  }\leqslant
C\delta^{1/4}\left[  \log\left(  c/\delta\right)  \right]  ^{1/8}.
\label{eq:logconcave-map}%
\end{equation}
Constants $C>0$, $c\geqslant e$ depend only on $d,$ $C_{0},$ $\kappa,$ $R_{Y}%
$, not on the center $a$ or on the two target measures.
\end{theorem}

\begin{proof}
The envelope implies finite moments of every order. For the annuli in
Proposition~\ref{prop:annuli} with $p=\infty$,
\[
W_{\infty,a}\left(  f\right)  \leqslant C_{0}\left(  1+\sum_{k\geqslant
1}2^{k\left(  d-1\right)  }e^{-\kappa4^{k}/2}\right)  \triangleq W<\infty.
\]
The sum converges by ratio test. Thus the preceding estimate $\left(
\ref{eq:E-sqrtD}\right)  $ applies with $B=K_{d,\infty}^{\mathrm{ann}}W$.
First compute, for $S\geqslant0$,
\begin{align*}
\tau_{1,\rho,a}\left(  S\right)   &  =\int_{|u|>S}|u|f(a+u)du\leqslant
C_{0}d\omega_{d}\int_{S}^{\infty}r^{d}e^{-\kappa r^{2}/2}dr\\
&  \leqslant C_{0}d\omega_{d}e^{-\kappa S^{2}/4}\int_{S}^{\infty}%
r^{d}e^{-\kappa r^{2}/4}dr\text{ (write }e^{-\kappa r^{2}/2}=e^{-\kappa
r^{2}/4}\cdot e^{-\kappa r^{2}/4}\text{)}\\
&  \leqslant C_{0}d\omega_{d}e^{-\kappa S^{2}/4}\int_{0}^{\infty}%
r^{d}e^{-\kappa r^{2}/4}dr\text{ }\\
&  =C_{0}d\omega_{d}e^{-\kappa S^{2}/4}\left(  1/2\right)  (4/\kappa
)^{(d+1)/2}\int_{0}^{\infty}w^{(d-1)/2}e^{-w}dw\text{ (let }w=\kappa
r^{2}/4\text{)}\\
&  =C_{0}d\omega_{d}e^{-\kappa S^{2}/4}\left(  1/2\right)  (4/\kappa
)^{(d+1)/2}\Gamma((d+1)/2).
\end{align*}
Consequently
\[
\tau_{1,\rho,a}\left(  S\right)  \leqslant C_{1}e^{-c_{1}S^{2}},
\]
where $c_{1}=\kappa/4,$ $C_{1}=C_{0}d\omega_{d}\left(  1/2\right)
(4/\kappa)^{(d+1)/2}\Gamma((d+1)/2).$ Let now $\gamma$ be optimal for
$W_{1}\left(  \mu,\nu\right)  $ and construct $\sigma_{T_{\mu},\gamma}$ as in
Definition~\ref{def:repaired-plan}. For $0<\delta\leqslant1,$ let
\[
S=\left(  \ell/c_{1}\right)  ^{1/2},\text{ }\ell=\log\left(  c/\delta\right)
,\text{ }c=\max\left\{  e\max\left\{  1,2R_{Y}\right\}  ,C_{1}\right\}  .
\]
Since $\delta$ cannot exceed $2R_{Y},$ we have $\ell\geqslant1$ and
\[
\tau_{1,\rho,a}\left(  S\right)  \leqslant C_{1}e^{-\ell}=(C_{1}%
/c)\delta\leqslant\delta.
\]
Proposition~\ref{prop:D-1-tail} therefore yields
\[
\mathcal{D}_{\nu}(\sigma_{T_{\mu},\gamma})\leqslant2\delta(S+2R_{Y}%
)=2\delta(c_{1}^{-1/2}\ell^{1/2}+2R_{Y})\leqslant C_{D}\delta\ell^{1/2},
\]
where $C_{D}=2(c_{1}^{-1/2}+2R_{Y}).$ Since%
\[
\int\left\vert y-z\right\vert ^{2}d\gamma\leqslant2R_{Y}\int\left\vert
y-z\right\vert d\gamma=2R_{Y}\delta,
\]
$\left(  \ref{eq:map-graph}\right)  $ and~$\left(  \ref{eq:E-sqrtD}\right)  $
imply
\[
\Vert T_{\mu}-T_{\nu}\Vert_{L^{2}(\rho)}^{2}\leqslant4R_{Y}\delta+2C_{E}%
C_{D}^{1/2}\delta^{1/2}\ell^{1/4}\leqslant(4R_{Y}+2C_{E}C_{D}^{1/2}%
)\delta^{1/2}\ell^{1/4}.
\]
Taking square roots yields%
\[
\Vert T_{\mu}-T_{\nu}\Vert_{L^{2}(\rho)}\leqslant(4R_{Y}+2C_{E}C_{D}%
^{1/2})^{1/2}\delta^{1/4}\ell^{1/8}.
\]
For $\delta=0$, $\left(  \ref{eq:logconcave-map}\right)  $ is trivial. For
$1<\delta\leqslant2R_{Y}$, the chosen $c$ still gives $\ell\geqslant1$. So%
\[
\left\Vert T_{\mu}-T_{\nu}\right\Vert _{L^{2}(\rho)}\leqslant2R_{Y}%
\leqslant2R_{Y}\delta^{1/4}\ell^{1/8}.
\]
Thus choosing $C=\max\{(4R_{Y}+2C_{E}C_{D}^{1/2})^{1/2},\,2R_{Y}\}$ proves
$\left(  \ref{eq:logconcave-map}\right)  $.
\end{proof}

\begin{corollary}
[Bounded perturbations of strongly log-concave densities]%
\label{cor:strong-logconcave}Let $\rho=fdx$ and suppose that the density
$f=e^{-U-F}$ with $U\in C^{2}\left(  \mathbb{R}^{d}\right)  $, $D^{2}%
U\geqslant\kappa I_{d}$, $\kappa>0$ and $F\in L^{\infty}.$ Let $\mu,$ $\nu$ be
supported in $B(b,R_{Y})$. Then $\left(  \ref{eq:logconcave-map}\right)  $
holds with $C_{0}=e^{-U\left(  a\right)  +\Vert F\Vert_{\infty}}$, where $a$
is the unique minimizer of $U$. No upper Hessian bound is required.
\end{corollary}

\begin{proof}
The lower Hessian bound implies, for every $x$,
\[
U(x)\geqslant U(z)+\nabla U(z)\cdot(x-z)+\frac{\kappa}{2}|x-z|^{2},
\]
where $z$ is the unique minimizer of $U$. Using $-F(x)\leqslant\Vert
F\Vert_{\infty}$, we obtain
\[
f(x)=e^{-U(x)-F(x)}\leqslant e^{-U(a)+\Vert F\Vert_{\infty}}e^{-\kappa
|x-a|^{2}/2},\text{ for almost every }x.
\]
This is the Gaussian upper envelope with the stated $C_{0}$, therefore
Theorem~\ref{thm:logconcave} applies.
\end{proof}

\emph{Improvement discussion}. Under these strong log-concavity and bounded
perturbation assumptions, \cite[Theorem 1.4 (1.7)]{letrouit2026gluing}
additionally assumes $D^{2}U\leqslant\kappa^{\prime}I_{d}$ for its map
estimate and obtains
\[
q_{\mathrm{LM}}=\frac{\kappa}{2\kappa^{\prime}+7\kappa}\leqslant\frac{1}{9}.
\]
The accompanying refinement there also allows any positive exponent
$q<\kappa/\left(  2\kappa^{\prime}+6\kappa\right)  \leqslant1/8$. For either
choice, $q<1/4$. Thus $\left(  \ref{eq:logconcave-map}\right)  $ improves
these map exponents in the small $\delta$ range.

\begin{theorem}
[Polynomial upper envelope]\label{thm:cauchy}Let $\rho=fdx$ and suppose for
some $a\in\mathbb{R}^{d}$, $\beta>d+2,$ $C_{0}>0$ that%
\[
0\leqslant f(x)\leqslant C_{0}(1+|x-a|^{2})^{-\beta/2},\text{ for almost every
}x.
\]
Let $\mu,$ $\nu$ be supported in $B(b,R_{Y})$, $R_{Y}>0$. Then with
$\delta=W_{1}(\mu,\nu)$,%
\begin{equation}
\left\Vert T_{\mu}-T_{\nu}\right\Vert _{L^{2}(\rho)}\leqslant C\delta^{\left(
\beta-d-1\right)  /\left(  4\left(  \beta-d\right)  \right)  }.
\label{eq:cauchy-map}%
\end{equation}
The constant depends only on $d,$ $\beta,$ $C_{0},$ $R_{Y}$. No density lower
bound is assumed.
\end{theorem}

\begin{proof}
Write $t=\beta-d>2$. The envelope implies $\rho\in\mathcal{P}_{2}$ and
\[
W_{\infty,a}(f)\leqslant C_{0}\left(  1+\sum_{k\geqslant1}2^{k(d-1-\beta
)}\right)  \triangleq W.
\]
The geometric series converges, so $\left(  \ref{eq:E-sqrtD}\right)  $ applies
with $B=K_{d,\infty}^{\mathrm{ann}}W$. For $S\geqslant1$,
\begin{align*}
\tau_{1,\rho,a}\left(  S\right)   &  =\int_{|u|>S}|u|f(a+u)\,du\leqslant
C_{0}d\omega_{d}\int_{S}^{\infty}r^{d}(1+r^{2})^{-\beta/2}\,dr\\
&  \leqslant C_{0}d\omega_{d}\int_{S}^{\infty}r^{d-\beta}\,dr=C_{0}d\omega
_{d}\lim_{R\rightarrow\infty}\frac{R^{d-\beta+1}-S^{d-\beta+1}}{d-\beta+1}\\
&  =\frac{C_{0}d\omega_{d}}{\beta-d-1}S^{-(\beta-d-1)}=C_{\tau}S^{-(t-1)}.
\end{align*}
where $C_{\tau}=C_{0}d\omega_{d}\left(  t-1\right)  ^{-1}$. Let $\gamma$ be
optimal for $W_{1}\left(  \mu,\nu\right)  $ and construct $\sigma_{T_{\mu
},\gamma}$ as in Definition~\ref{def:repaired-plan}. For $0<\delta\leqslant1$,
choose $S=\delta^{-1/t}\geqslant1$. Then
\[
\delta S=\delta^{(t-1)/t},\text{ }S^{-(t-1)}=\delta^{(t-1)/t}.
\]
Thus Proposition~\ref{prop:D-1-tail} gives
\[
\mathcal{D}_{\nu}(\sigma_{T_{\mu},\gamma})\leqslant2\{\delta S+2R_{Y}C_{\tau
}S^{-(t-1)}\}\leqslant C_{D}\delta^{(t-1)/t},
\]
where $C_{D}=2(1+2R_{Y}C_{\tau}).$ Using $\left(  \ref{eq:map-graph}\right)  $
and $\left(  \ref{eq:E-sqrtD}\right)  $ gives
\[
\Vert T_{\mu}-T_{\nu}\Vert_{L^{2}(\rho)}^{2}\leqslant4R_{Y}\delta+2C_{E}%
C_{D}^{1/2}\delta^{(t-1)/(2t)}\leqslant(4R_{Y}+2C_{E}C_{D}^{1/2}%
)\delta^{(t-1)/(2t)}.
\]
Here $0<(t-1)/(2t)<1$ and $\delta\leqslant1$. Taking square roots gives the
exponent
\[
\frac{t-1}{4t}=\frac{\beta-d-1}{4(\beta-d)}.
\]
At $\delta=0$, uniqueness gives equality of the maps. For $\delta>1$, use
$\Vert T_{\mu}-T_{\nu}\Vert_{L^{2}(\rho)}\leqslant2R_{Y}\leqslant2R_{Y}\delta
$. Therefore $C=\max\{(4R_{Y}+2C_{E}C_{D}^{1/2})^{1/2},2R_{Y}\}$ works for
$\left(  \ref{eq:cauchy-map}\right)  $.
\end{proof}

\emph{Improvement discussion}. Under two-sided comparability with
$(1+\left\vert x\right\vert ^{2})^{-\beta/2}$, \cite[Theorem 1.12
(1.17)]{letrouit2026gluing} gives the compact target map exponent $\left(
\beta-d-2\right)  /\left(  8\beta-2d-4\right)  $. For $t=\beta-d>2$, the
improvement in $\left(  \ref{eq:cauchy-map}\right)  $ is strict,
\[
\frac{t-1}{4t}-\frac{t-2}{8t+6d-4}=\frac{4t^{2}+\left(  6d-4\right)  \left(
t-1\right)  }{4t\left(  8t+6d-4\right)  }>0.
\]
Our theorem also removes the density lower bound.

\subsection{John domains}

\label{sec:LM-John}

In \cite[Theorem 1.7]{letrouit2026gluing}, the source $\rho$ has a density $f$
bounded above and below by positive constants on a bounded John domain
$\mathcal{X}$. The map stability conclusion there additionally assumes
rectifiable boundary of finite $\mathcal{H}^{d-1}$ measure and has
$L^{2}\left(  \rho\right)  $-norm exponent $1/6$. For this setting, the
small-set condition is satisfied, since for every Borel $E$, it holds%
\[
\rho(E)=\int_{E\cap\mathcal{X}}fdx\leqslant\Vert f\Vert_{\infty}\left\vert
E\right\vert .
\]
So Theorem~\ref{thm:compact-maps} applies with $s=1,$ $C_{V}=\Vert
f\Vert_{\infty}$ and gives
\begin{equation}
\left\Vert T_{\mu}-T_{\nu}\right\Vert _{2}\leqslant C\delta^{1/4}%
.\label{eq:LM-John-improved}%
\end{equation}
Thus the map estimate \cite[Theorem 1.7 (1.10)]{letrouit2026gluing} is
strictly improved in the small $\delta$ order. Neither the source density
lower bound, the John domain condition, connectedness, nor boundary
rectifiability is used in $\left(  \ref{eq:LM-John-improved}\right)  $, only a
ball containing the source and an upper density bound are used. In particular
the same comparison applies to the finite union extension in \cite[Section
4.3]{letrouit2026gluing}. Note that the improvement in this setting has
already been made recently in \cite[Theorem 3.1]{merigot2026sharp} as remarked
in Theorem \ref{thm:compact-maps}.

\subsection{Boundary powers}

\label{sec:LM-boundary}

To avoid conflicting notation, we write $\zeta$ for the boundary exponent
called $\delta$ in \cite[Theorem 1.10 (1.13)]{letrouit2026gluing}, while
keeping $\delta$ for $W_{1}(\mu,\nu)$ here. Let $\mathcal{X}$ be a bounded
Lipschitz domain, $r(x)=\operatorname{dist}(x,\partial\mathcal{X})$, and
suppose the probability density, extended by zero outside $\mathcal{X}$,
satisfies for $\zeta>-1$,%
\begin{equation}
c_{1}r(x)^{\zeta}\leqslant f(x)\leqslant c_{2}r(x)^{\zeta}\text{, }%
x\in\mathcal{X}. \label{eq:boundary-powers}%
\end{equation}
Note that only the upper bound will be used below. The map exponent in
\cite[Theorem 1.10]{letrouit2026gluing} is%
\begin{equation}
q_{\mathrm{LM}}(\zeta)=\left\{
\begin{array}
[c]{ll}%
(1+\zeta)/6, & -1<\zeta\leqslant0,\\
(2+\zeta)/[12(1+\zeta)], & \zeta\geqslant0.
\end{array}
\right.  \label{eq:LM-boundary-old}%
\end{equation}
The two expressions agree at zero.

\emph{Nonnegative boundary powers. }Assume $\zeta\geqslant0$. Let
$D_{\mathcal{X}}=\operatorname{diam}(\mathcal{X})>0$. Since $r(x)\leqslant
D_{\mathcal{X}},$ the density is bounded by $c_{2}D_{\mathcal{X}}^{\zeta}$.
$\left(  \ref{eq:LM-John-improved}\right)  $ therefore applies without using
the lower bound or the Lipschitz condition. The exact exponent difference is
\begin{equation}
\frac{1}{4}-\frac{2+\zeta}{12(1+\zeta)}=\frac{1+2\zeta}{12(1+\zeta)}>0.
\label{eq:LM-boundary-positive-difference}%
\end{equation}

\emph{Negative boundary powers and the endpoint small-set estimate}. Assume
$-1<\zeta<0$. Recall the boundary layer geometric estimate from \cite[3.2.37]%
{federer2014geometric} or \cite[Theorem 2.106]{ambrosio2000functions}: there
are $t_{0}>0$ and $K_{\mathcal{X}}>0$ depending only on the domain, such that
for $0<t\leqslant t_{0},$%
\begin{equation}
V(t)\triangleq\left\vert \left\{  x\in\mathcal{X}:r\left(  x\right)
<t\right\}  \right\vert \leqslant K_{\mathcal{X}}t.
\label{eq:LM-boundary-layer}%
\end{equation}
After replacing $K_{\mathcal{X}}$ with a larger constant, e.g., $\max\left\{
K_{\mathcal{X}},\left\vert \mathcal{X}\right\vert /t_{0}\right\}  ,$ we may
assume that the estimate holds for all $t>0$. For $0<r<t$, the identity%
\[
r^{\zeta}=t^{\zeta}+(-\zeta)\int_{r}^{t}u^{\zeta-1}\,du
\]
has nonnegative terms on the right. Since $r(x)>0$ for $x\in\mathcal{X}$,
multiplying the identity by the layer indicator gives the pointwise identity
on $\mathcal{X}$
\[
1_{\{r(x)<t\}}r(x)^{\zeta}=t^{\zeta}1_{\{r(x)<t\}}+(-\zeta)\int_{0}%
^{t}1_{\{r(x)<u\}}u^{\zeta-1}\,du.
\]
Integrate the two terms over $\mathcal{X}$. Tonelli's theorem allows the
interchange of integtation in the second term
\[
\int_{\mathcal{X}}\int_{0}^{t}1_{\{r(x)<u\}}u^{\zeta-1}dudx=\int_{0}%
^{t}u^{\zeta-1}\left(  \int_{\mathcal{X}}1_{\{r(x)<u\}}dx\right)  du=\int%
_{0}^{t}u^{\zeta-1}V(u)du.
\]
The first term integrates to $t^{\zeta}V(t)$. $\left(
\ref{eq:LM-boundary-layer}\right)  $ consequently gives
\begin{align}
\int_{\{r<t\}}r(x)^{\zeta}dx  &  =t^{\zeta}V(t)+(-\zeta)\int_{0}^{t}%
u^{\zeta-1}V(u)du\label{eq:LM-weighted-layer}\\
&  \leqslant K_{\mathcal{X}}t^{1+\zeta}+\frac{-\zeta}{1+\zeta}K_{\mathcal{X}%
}t^{1+\zeta}=\frac{K_{\mathcal{X}}}{1+\zeta}t^{1+\zeta}.\nonumber
\end{align}
The integral at $t=0$ is finite because $1+\zeta>0$. Fix a Borel set $E.$ If
$\left\vert E\right\vert >0$, choose $t=\left\vert E\right\vert
/K_{\mathcal{X}}$. Split $E\cap\mathcal{X}$ at distance $t$ from the boundary.
On $\{r\geqslant t\}$, $r^{\zeta}\leqslant t^{\zeta}$. On $\{r<t\},$ $\left(
\ref{eq:LM-weighted-layer}\right)  $ applies. Hence%
\begin{equation}
\rho(E)\leqslant c_{2}\left\{  \frac{K_{\mathcal{X}}}{1+\zeta}t^{1+\zeta
}+t^{\zeta}\left\vert E\right\vert \right\}  =c_{2}\frac{2+\zeta}{1+\zeta
}K_{\mathcal{X}}^{-\zeta}\left\vert E\right\vert ^{1+\zeta}.
\label{eq:LM-boundary-small-set}%
\end{equation}
For $\left\vert E\right\vert =0$, absolute continuity gives $\rho(E)=0$. Thus
$\left(  \ref{eq:small-set-density}\right)  $ holds with $0<s=1+\zeta<1$ and
$C_{V}=c_{2}\left(  2+\zeta\right)  K_{\mathcal{X}}^{-\zeta}/\left(
1+\zeta\right)  $. So Theorem~\ref{thm:compact-maps} gives,%
\begin{equation}
\Vert T_{\mu}-T_{\nu}\Vert_{2}\leqslant C\delta^{(1+\zeta)/[2(2+\zeta
)]},\text{ }-1<\zeta<0. \label{eq:LM-boundary-improved}%
\end{equation}
This exponent exceeds $q_{\mathrm{LM}}$ in $\left(  \ref{eq:LM-boundary-old}%
\right)  $, since
\[
\frac{1+\zeta}{2(2+\zeta)}-\frac{1+\zeta}{6}=\frac{(1+\zeta)(1-\zeta
)}{6(2+\zeta)}>0.
\]
The constant $C$ of $\left(  \ref{eq:LM-boundary-improved}\right)  $ are
controlled by the upper density coefficient, $\zeta$, $\mathcal{X}$, $d$ and
the target radius $R_{Y}$. The lower bound of $\left(
\ref{eq:boundary-powers}\right)  $ is not used. We may also replace Lipschitz
geometry by the explicit layer estimate $\left(  \ref{eq:LM-boundary-layer}%
\right)  $.

\subsection{Spherical uniform density}

\label{sec:LM-spherical}

The source in \cite[Theorem 1.11 (1.15)]{letrouit2026gluing} is%
\begin{equation}
f(x)=c_{d}|x|^{1-d}1_{\{0<|x|\leqslant1\}}, \label{eq:LM-spherical-density}%
\end{equation}
where $c_{d}=(d\omega_{d})^{-1}.$ By polar coordinates, for $0\leqslant
r\leqslant1,$%
\begin{equation}
\int_{B(0,r)}f(x)dx=c_{d}d\omega_{d}\int_{0}^{r}t^{1-d}t^{d-1}dt=r.
\label{eq:LM-f-spherical-integral}%
\end{equation}
The exponent of \cite[Theorem 1.11 (1.15)]{letrouit2026gluing} is $1/(6d)$,
which depends on the dimension. We provide two versions of improvement of the
exponent, one depends on the dimension and the other is independent of the dimension.

$L^{p}$\emph{ bound alone gives improved dimension-dependent exponent. }For
$d\geqslant2$, using polar integration we have
\[
\Vert f\Vert_{p}^{p}=c_{d}^{p}d\omega_{d}\int_{0}^{1}r^{d-1-p(d-1)}%
dr=\frac{c_{d}^{p}d\omega_{d}}{d-p(d-1)},\text{ if }1<p<\frac{d}{d-1}.
\]
The integral diverges at $p=d/(d-1)$. Hence $s=1-1/p$ ranges over $(0,1/d)$.
So the bound of $L^{p}$-norm of $f$ depends on the dimension. Applying
Theorem~\ref{thm:compact-maps}, we get a stability with the norm exponent
$s/\left[  2\left(  s+1\right)  \right]  \triangleq q$ ranges over
$0<q<1/[2(d+1)]$. Clearly the exponent obtained this way depends on the
dimension. Since
\[
\frac{1}{2(d+1)}-\frac{1}{6d}=\frac{2d-1}{6d(d+1)}>0,
\]
hence any exponent $q$ strictly between these two numbers give rise to an
improve of \cite[Theorem 1.11 (1.15)]{letrouit2026gluing}. Next we show that
how to obtain a dimension-free exponent.

\emph{A stronger estimate from shrinking shells}. The integrated oscillation
estimate $\left(  \ref{eq:osc-int}\right)  $ retains more information than the
$L^{p}$ bound. Let $d\geqslant1$, $0<h\leqslant1$ and set
\[
N=\lfloor\log_{2}(1/h)\rfloor\text{ and }r_{j}=2^{-j},\text{ }j\geqslant0.
\]
Since $\log_{2}(1/h)-1<N\leqslant\log_{2}(1/h),$ we have $h\leqslant r_{N}%
<2h$. Let $A_{j}=\left\{  r_{j+1}<\left\vert x\right\vert \leqslant
r_{j}\right\}  ,$ then $A_{j}\subset B\left(  0,r_{j}\right)  $ and%
\[
B\left(  0,1\right)  =B\left(  0,r_{N}\right)  \cup\left(  \bigcup_{0\leqslant
j<N}A_{j}\right)  .
\]
For each $1$-Lipschitz convex $v$, we use $\left(  \ref{eq:omega-2L}\right)  $
and $\left(  \ref{eq:LM-f-spherical-integral}\right)  $ to get the integrated
oscillation on $B\left(  0,r_{N}\right)  $,%
\[
\int_{B(0,r_{N})}\omega_{v,h}^{2}fdx\leqslant4r_{N}\leqslant8h.
\]
On each $A_{j}$ with $0\leqslant j<N,$ the density definition yields
$f\leqslant c_{d}r_{j+1}^{1-d}$. So using $\left(  \ref{eq:osc-int}\right)  $
with $r=2$ and $S=r_{j}$, we obtain
\[
\int_{A_{j}}\omega_{v,h}^{2}fdx\leqslant c_{d}r_{j+1}^{1-d}\int_{B(0,r_{j}%
)}\omega_{v,h}^{2}dx\leqslant c_{d}K_{d,2}^{\mathrm{osc}}r_{j+1}^{1-d}%
r_{j}^{d-1}h=c_{d}2^{d-1}K_{d,2}^{\mathrm{osc}}h.
\]
Taking the supremum over $v$, we obtain for $0<h\leqslant1$,
\begin{equation}
\mathfrak{A}_{\rho}\left(  h\right)  \leqslant8h+Nc_{d}2^{d-1}K_{d,2}%
^{\mathrm{osc}}h\leqslant B_{d}h\log\left(  e/h\right)  \leqslant B_{d}%
h\log\left(  c/h\right)  , \label{eq:LM-spherical-A}%
\end{equation}
where $B_{d}=8+c_{d}2^{d-1}K_{d,2}^{\mathrm{osc}}\left(  \log2\right)  ^{-1}$
and $c=e\max\left\{  1,2R_{Y}\right\}  .$ To prove the map stability, let
$\gamma$ be optimal for $W_{1}\left(  \mu,\nu\right)  $. The source is support
in $B\left(  0,R_{X}\right)  $ with $R_{X}=1$, so $\mathcal{D}_{\nu}%
(\sigma_{\gamma})\leqslant2\delta$. As in the proof of Corollary
\ref{cor:compact-target}\ the tail term there vanishes for compact target with
the choice $R=R_{Y}.$ Substituting $\left(  \ref{eq:LM-spherical-A}\right)  $
into $\left(  \ref{eq:error-bound}\right)  $ and substituting the resulting
inequality into $\left(  \ref{eq:map-graph}\right)  $, we get for
$0<h\leqslant1,$
\begin{equation}
\Vert T_{\mu}-T_{\nu}\Vert_{2}^{2}\leqslant4R_{Y}\delta+16R_{Y}\delta
/h+36R_{Y}^{2}B_{d}h\log\left(  c/h\right)  . \label{eq:LM-T2-bound-pre}%
\end{equation}
For $0<\delta\leqslant1$, let $\ell=\log\left(  c/\delta\right)  .$ Then
$\ell>1$ and $\ell>\delta.$ We now choose $h=\left(  \delta/\ell\right)
^{1/2}$, this $h$ belongs to $(0,1]$. By direct calculations we have%
\[
\log(c/h)=\frac{2\log c+\log\ell-\log\delta}{2}=\frac{2\log c+\log\ell
-\log\left(  c/e^{\ell}\right)  }{2}=\frac{\ell+\log c+\log\ell}{2}%
\leqslant\left(  1+e^{-1}\right)  \ell,
\]
where the inequality uses $\log\ell\leqslant\left(  1+e^{-1}\right)  \ell-\log
c$. We also have%
\[
\delta\leqslant\sqrt{\delta\ell},\text{ }\delta/h=\sqrt{\delta\ell},\text{
}h\log(c/h)\leqslant\left(  1+e^{-1}\right)  h\ell=\left(  1+e^{-1}\right)
\sqrt{\delta\ell}.
\]
Substituting this $h$ into $\left(  \ref{eq:LM-T2-bound-pre}\right)
$\ therefore gives,
\begin{equation}
\Vert T_{\mu}-T_{\nu}\Vert_{2}\leqslant\left(  1+e^{-1}\right)  ^{1/2}%
(20R_{Y}+36R_{Y}^{2}B_{d})^{1/2}\delta^{1/4}[\log(c/\delta)]^{1/4}.
\label{eq:LM-spherical-map}%
\end{equation}
For $\delta=0$, the map error is zero. For $1<\delta\leqslant2R_{Y}$ ($\delta$
cannot exceed $2R_{Y}$), the chosen $c$ has $\ell\geqslant1$. So%
\[
\Vert T_{\mu}-T_{\nu}\Vert_{2}\leqslant2R_{Y}\leqslant2R_{Y}\delta^{1/4}%
[\log(c/\delta)]^{1/4}.
\]
Therefore for all admissible $\delta\geqslant0,$ we have the modulus
$\delta^{1/4}[\log(c/\delta)]^{1/4}$. Since $1/4-1/(6d)=\left(  3d-2\right)
/\left(  12d\right)  >0$ for all $d\geqslant1$, so the modulus of $\left(
\ref{eq:LM-spherical-map}\right)  $ improves the cited exponent $1/(6d)$ for
all $d\geqslant1$ in the small $\delta$ range.

\emph{The even stronger one-dimensional case}. When $d=1$, $\left(
\ref{eq:LM-spherical-density}\right)  $ is uniform on $(-1,1)$ with bounded
density $1/2$, and $\left(  \ref{eq:LM-John-improved}\right)  $ gives $1/4$
without a logarithm. In fact all three comparisons from Section
\ref{sec:LM-John}, Section \ref{sec:LM-boundary} to Section
\ref{sec:LM-spherical} have a stronger one-dimensional conclusion. As proved
in Section~\ref{sec:sharp-1D}, the common source increasing maps induce the
monotone $W_{1}$-optimal target coupling. For targets in $B(b,R_{Y})$,
\[
\Vert T_{\mu}-T_{\nu}\Vert_{2}^{2}=\int|T_{\mu}-T_{\nu}|^{2}d\rho
\leqslant2R_{Y}\int|T_{\mu}-T_{\nu}|\,d\rho=2R_{Y}\delta.
\]
Thus the norm exponent is $1/2$ in dimension one, for any absolutely
continuous source and compactly supported targets.

\section{From objective gap to map accuracy}

\label{sec:gap-map-error}

Dual-based algorithms for optimal transport usually use primal-dual gap as
stopping criterion. An approximate optimal map is then constructed, e.g.,
through barycentric projection. The purpose of this section is to demonstrate
turning a primal-dual gap into an error bound for barycentric projection map.
Let $\Psi_{\rho,\nu}(z)$ denote the infimum on the right of $\left(
\ref{eq:error-bound}\right)  $, with objective gap replaced by $z\geqslant0$.
It is nondecreasing, since every expression being minimized is nondecreasing.

\begin{corollary}
\label{cor:primal-dual}Fix $\rho,\nu$ satisfying Assumption~\ref{ass:brenier},
a plan $\sigma\in\Pi(\rho,\nu)$, and integrable functions $\phi\in L^{1}%
(\rho)$, $\psi\in L^{1}(\nu)$ such that $\phi(x)+\psi(y)\leqslant\left\vert
x-y\right\vert ^{2}/2$ for every $x,y$. Let
\[
\epsilon=\frac{1}{2}\int|x-y|^{2}\,d\sigma-\int\phi\,d\rho-\int\psi\,d\nu
\]
be the primal-dual gap. Then $0\leqslant\mathcal{D}_{\nu}(\sigma
)\leqslant\epsilon$ and
\[
\left\Vert \bar{T}_{\sigma}-T_{\nu}\right\Vert _{2}^{2}+V_{\sigma}%
=\mathcal{E}_{\nu}\left(  \sigma\right)  \leqslant\Psi_{\rho,\nu}\left(
\epsilon\right)  .
\]

\end{corollary}

\begin{proof}
It is a result of Theorem~\ref{thm:error-bound}, the monotonicity of
$\Psi_{\rho,\nu}$ and $\left(  \ref{eq:variance-decomposition}\right)  $.
\end{proof}

\begin{corollary}
\label{cor:barycentric-proj}Suppose the setting of
Corollary~\ref{cor:primal-dual} holds for a family $\sigma_{t}\in\Pi(\rho
,\nu)$ of plans with the corresponding primal-dual gaps $\epsilon_{t}\leqslant
C_{0}t^{\zeta}$, where $C_{0},\zeta>0$. If $\Psi_{\rho,\nu}(z)\leqslant
Kz^{\Theta}$ for $z\geqslant0$, then
\[
\Vert\bar{T}_{\sigma_{t}}-T_{\nu}\Vert_{2}\leqslant K^{1/2}C_{0}^{\Theta
/2}t^{\zeta\Theta/2}\text{ and }V_{\sigma_{t}}\leqslant KC_{0}^{\Theta
}t^{\zeta\Theta}.
\]
For example, under Theorem~\ref{thm:unbounded-power} take $\Theta=\Theta
_{\eta,m}$ and $K=Q_{d,\rho,p,n}$. For the logarithmic class, replace
$Kz^{\Theta}$ by $C_{\mathrm{up}}\varphi(z)$ on $0<z\leqslant t_{\mathrm{up}}%
$, as in $\left(  \ref{eq:explicit-upper-constant}\right)  $.
\end{corollary}

\begin{proof}
The bounds follow from the proof of Theorem~\ref{thm:unbounded-power} and
Section~\ref{sec:class-proof}.
\end{proof}

\section{An optimization of sum of powers}

\begin{lemma}
\label{lem:two-power}Let $A,B,\alpha,\beta>0$. Then
\begin{equation}
\inf_{t>0}\left\{  At^{\alpha}+Bt^{-\beta}\right\}  =c_{\alpha,\beta}%
A^{\beta/\left(  \alpha+\beta\right)  }B^{\alpha/\left(  \alpha+\beta\right)
}, \label{eq:two-power}%
\end{equation}%
\begin{equation}
c_{\alpha,\beta}=\left(  \frac{\beta}{\alpha}\right)  ^{\alpha/\left(
\alpha+\beta\right)  }+\left(  \frac{\alpha}{\beta}\right)  ^{\beta/\left(
\alpha+\beta\right)  }. \label{eq:two-power-constant}%
\end{equation}
The unique minimizer is
\begin{equation}
t_{\ast}=\left(  \frac{\beta B}{\alpha A}\right)  ^{1/\left(  \alpha
+\beta\right)  }. \label{eq:two-power-minimizer}%
\end{equation}

\end{lemma}

\begin{proof}
For $\Phi\left(  t\right)  =At^{\alpha}+Bt^{-\beta}$,
\[
\Phi^{\prime}\left(  t\right)  =\alpha At^{\alpha-1}-\beta Bt^{-\beta-1}.
\]
Thus $\Phi^{\prime}\left(  t\right)  =0$ if and only if
\[
\alpha At^{\alpha+\beta}=\beta B,
\]
which gives $\left(  \ref{eq:two-power-minimizer}\right)  $. Since
$\Phi(t)\rightarrow+\infty$ at both endpoints $0$ and $+\infty$, this critical
point is the global minimizer. Substitution gives
\begin{align*}
At_{\ast}^{\alpha}  &  =A\left(  \frac{\beta B}{\alpha A}\right)
^{\alpha/(\alpha+\beta)}=\left(  \frac{\beta}{\alpha}\right)  ^{\alpha
/(\alpha+\beta)}A^{\beta/(\alpha+\beta)}B^{\alpha/(\alpha+\beta)},\\
Bt_{\ast}^{-\beta}  &  =B\left(  \frac{\beta B}{\alpha A}\right)
^{-\beta/(\alpha+\beta)}=\left(  \frac{\alpha}{\beta}\right)  ^{\beta
/(\alpha+\beta)}A^{\beta/(\alpha+\beta)}B^{\alpha/(\alpha+\beta)}.
\end{align*}
Adding these two identities proves $\left(  \ref{eq:two-power}\right)  $.
\end{proof}

\bibliographystyle{plain}
\bibliography{main}
\end{document}